\documentclass[11pt]{amsart}
\usepackage{amsmath,amsfonts,amssymb,amscd,verbatim,xypic,latexsym, xy}
\usepackage{graphicx}

\usepackage{psboxit}
\usepackage{ifthen}
\newdimen\minCDarrowwidth
\usepackage{amsmath}
\usepackage{amsfonts}
\usepackage{amssymb}
\usepackage{amscd}
\usepackage{verbatim}

\def\ZZ{{\mathbb Z}}

\def\CC{{\mathbb C}}
\def\AA{{\mathbb A}}

\def\QQ{{\mathbb Q}}
\def\PP{{\mathbb P}}

\def\cI{\mathcal{I}}
\def\cJ{\mathcal{J}}

\def\cA{\mathcal{A}}

\def\cG{\mathcal{G}}

\def\cK{\mathcal{K}}
\def\cL{\mathcal{L}}
\def\cO{\mathcal{O}}

\def\cN{\mathcal{N}}
\def\cQ{\mathcal{Q}}
\def\cP{\mathcal{P}}
\def\cU{\mathcal{U}}

\def\cT{\mathcal{T}}

\def\Im{\mbox{ Im }}

\DeclareMathOperator{\Bl}{Bl} \DeclareMathOperator{\Coker}{Coker}
\DeclareMathOperator{\Ker}{ker} 
\DeclareMathOperator{\Proj}{Proj} 

 \DeclareMathOperator{\Spec}{Spec}
 \DeclareMathOperator{\codim}{codim}

\newtheorem{lemma}{Lemma}[section]
\newtheorem{theorem}[lemma]{Theorem}
\newtheorem{corollary}[lemma]{Corollary}
\newtheorem{proposition}[lemma]{Proposition}

\theoremstyle{definition}
\newtheorem{definition}[lemma]{Definition}
\newtheorem{example}[lemma]{Example}

\newtheorem{remark}[lemma]{Remark}

\newtheorem*{notation}{Notation}

\numberwithin{equation}{section}

\newcommand{\bean}{\begin{eqnarray}}
\newcommand{\eean}{\end{eqnarray}}
\newcommand{\bea}{\begin{eqnarray*}}
\newcommand{\eea}{\end{eqnarray*}}
\newcommand{\be}{\begin{displaymath}}
\newcommand{\ee}{\end{displaymath}}

\newcommand{\ol}{\overline}

\xyoption{arc}

\begin{document}

\title{The
structure of a local embedding and Chern classes of weighted
blow-ups}

%{The global structure of a local embedding \\and applications to Chern classes of weighted blow-ups}

\author{Anca~M.~Musta\c{t}\v{a}}
\author{Andrei~Musta\c{t}\v{a}}

\address{School of Mathematical Sciences, University College Cork, Cork, Ireland}
%\address{ Max Plank Institute for Mathematical Sciences}
\email{{\tt a.mustata@ucc.ie, andrei.mustata@ucc.ie}}

%\address{}
%\email{{\tt a.mustata@ucc.ie}}

\date{\today}

%\subjclass{14J30 (Primary) 14J32 (Secondary)}
%\keywords{Hilbert scheme, quintic threefold}

\begin{abstract}
For a proper local embedding between two Deligne--Mumford stacks $Y$ and $X$, we find, under certain mild conditions, a new (possibly non-separated)  Deligne--Mumford stack $X'$, with an etale, surjective and universally closed map to the target $X$, and whose fiber product with the image of the local embedding  is a finite union of stacks with corresponding  etale, surjective and universally closed maps to $Y$. Moreover, a natural set of weights on the substacks of $X'$ allows the construction of a universally closed push-forward, and thus a comparison between the Chow groups of $X'$ and $X$.  We apply the construction above to the computation of the Chern classes of a weighted blow-up along a regular local embedding via deformation to a weighted normal cone and localization. We describe the stack $X'$ in the case when $X$ is the moduli space of stable maps with local embeddings at the boundary.  We apply the methods above to find the Chern classes of the stable map spaces.

\end{abstract}

\maketitle
\bigskip

\section*{Introduction}

Local embeddings form an important class of morphisms of algebraic
stacks. For instance, the morphisms from the components of the
inertia stack of a Deligne--Mumford stack into the stack itself are
in general local embeddings. As another fundamental  example, the
diagonal morphism of a stack is a local embedding and thus, the
local study of this type of morphisms has led to a good definition
of an intersection product on smooth Deligne--Mumford stacks by A.
Vistoli \cite{vistoli}, with a subsequent simplification by A.
Kresch \cite{kresch}. Their work relies on the existence, for any
local embedding $Y\to X$, of an \'etale atlas $V\to Y$, such that
the composition $V\to X$ can be factored into a closed embedding
$V\to U$ followed by an \'etale atlas $U \to X$. Based on these
covers, the normal cone of a local embedding (\cite{vistoli}), and a
deformation of the ambient stack to the normal cone (\cite{kresch})
can be constructed. However, these constructions are local in
essence and as such, they fail to completely encode information on
global invariants like Chern classes, or Chow ring structures, and
cannot be directly employed in Riemann-Roch type theorems like the
Riemann-Roch without denominators as formulated for closed
embeddings of schemes.

 Under suitable assumptions, in this paper we replace the local construction above by one more suited to the purposes just mentioned, at the expense of working with non-separated stacks rather than separated schemes. Let $g: Y \to X$ be a proper, local embedding of Noetherian stacks. Assume that $Y$ is reduced and geometrically unibranch, that the morphism on the image $Y \to g(Y)$ is equidimensional, and that its degree is equal to a fixed number $d$ at all generic points.
 Theorems \ref{X'} and \ref{universally closed, probabilistic weight} highlight  the existence, for every such map $g:Y\to X$, of an \'etale cover by a stack $X' \to X$,
such that $Y':=g(Y)\times_XX'$, a finite union of \'etale covers
$Y'_i$  of $Y$, is embedded in $X'$, and such that the morphisms $Y'
\to Y$ and $X' \to X$ are universally closed. Moreover, the morphism
$X' \to X$ is an isomorphism outside the image of $Y$. Thus the
study of proper local embeddings can be reduced for practical purposes to
that of closed embeddings of stacks.

  While the map $p: X'\to X$ is forcefully not
  proper, it is universally closed. A weight function $w$ on the set of substacks of $X'$ is naturally attached to the map $p$, referring to the number of possible extensions of maps involved in the valuative criterion of properness.
  This "probabilistic weight" contributes to a good definition of
  push-forward $p_*$ between the Chow groups of $X'$ and $X$, an extension of the usual definition of proper
  push-forward to this type of universally closed maps.  In effect, the Chow group  $A(X)$  can be recovered from $A(X')$
  via the universally closed push-forward $p_*$.

   The definition of the \'etale lift $X'$ associated to the proper local embedding $g:Y \to X$ and the subsequent constructions are based on a network of local embeddings associated to $g$. This network in turn depends on the choice of a suitable atlas $U$ of $X$ and a partition of $Y\times_XU$, with properties specified in Proposition \ref{U}.
   The dependence is only partial: while the number of spaces which are nodes in the network may vary, the spaces themselves are intrinsically associated to $g$. For example, by replacing $U$ with a number of copies of itself, one increases the number of components in $Y\times_XX'$. In view of this, one could enquire on the existence of a "minimal" choice of \'etale atlas for $X$, that would yield a canonical \'etale lift $X'\to X$. This problem is addressed  by the authors in \cite{noi4}. However, there are contexts where other factors, like, e.g., a moduli problem,  may determine the naturalness of a (non-minimal) choice of \'etale atlas, and thus of an \'etale lift $X'$. This is the case for the moduli spaces of stable maps with their boundary, which we study in the third part of this article.

%  As a possible application of these results, one could study the intersection theory of Deligne--Mumford   stacks in terms of the intersection theory of certain algebraic spaces. Indeed, by repeatedly applying the above construction for all the morphisms among the components of the inertia stack $I_X$ of a Deligne--Mumford stack $X$, one constructs an \'etale, universally closed morphism $X' \to X$, where $X'$ is an algebraic space.

 One useful feature of our construction comes from the fact that for a suitable \'etale atlas $U$ of
 $X$, we take into account  the entire pullback of that atlas to the locally embedded $Y$. Locally, in the
 neighborhoods of some points in $U$, this pullback may consist of a number of intersecting components.
 Their intersections contribute essentially to the structure of the morphism $Y \to X$;  in a first
 instance, to the associated flat stratification of $X$. For this reason we encode them in a network of
 morphisms of stacks, a stack version of the configuration schemes of \cite{lunts}. In \cite{noi1}, we have
 defined the extended Chow ring of such a network. In Theorem \ref{compare to network} we prove that this extended Chow ring
 is isomorphic to $A(X')$. In particular, in the case of the moduli space of stable maps and the local
 embeddings of its boundary divisors, the extended Chow ring  has been calculated in \cite{noi1}. In
 Theorem \ref{rigid spaces} we now identify the corresponding stack $\ol{M}'_{0,m}(\PP^n,d)$ and
 formulate its moduli
 problem in terms of stable maps with marked components.

 Our approach is relevant, for instance, when considering a blow-up along a local embedding. Let us return to the initial picture of a commutative diagram
$$\diagram
V \dto^{} \rto^{ } & U \dto^{} \\
Y \rto^{} & X,
\enddiagram $$
where the top map is an embedding. Then simply taking the blow-up of $U$ along $V$ does not lead to an \'etale atlas  of a complete blow-up of $X$ along $Y$, due to a break in symmetry at the level of relations. Indeed, in small enough neighborhoods around each point in $U$, one needs to consider blow-up along each of the components of $Y\times_XU$ before defining the \'etale atlas for $\Bl_YX$. In section 2, we show how this can be done for smooth stacks $Y$ and $X$.

Weighted blow-ups form a class of morphisms with a variety of
applications. They come up, for example, when considering variation
of GIT. As algebraic stacks provide a natural context for the study
of weighted blow-ups, this class may be extended to weighted
blow-ups along regular local embeddings. Moduli spaces of (weighted)
curves and (weighted) stable maps are examples for which this type
of morphisms comes up naturally.

%For any smooth variety $X$ with the action of a reductive group $G$, an \'etale cover over the Chow quotient will be constructed by considering a sequence of weighted blow-ups of local embeddings of an open set in the stack quotient $[X/G]$. The construction of moduli spaces of stable maps in \cite{noi1} could be interpreted as the construction of the universal family over the Chow quotient of a projective space by the action of $Sl_2$ via a sequence of weighted blow up of local imbeddings. This is just a particular case of a more general phenomenon which we will treat in \cite{3} where

As an application to the universally closed \'etale cover
construction, blowing up Chern classes along local
embeddings of smooth stacks is reduced to the case of smooth embeddings. The basic
idea of this computation, like in the case of schemes, is to
retrieve the Chern classes from their pullback to the exceptional
divisor, for example via the Riemann-Roch without denominators
formula (\cite{fulton}). However, when weights are considered, a
less standard approach is necessary for the retrieval step.
%Indeed, in the case of a simple blow-up, the relative tangent bundle of the exceptional divisor over the blow-up locus coincides to a $(-1)$--twist of the bundle...
  By a deformation to a "weighted normal cone", we reduce the problem to the case when both the blow-up locus and the exceptional divisor are fixed loci for $\CC^*$--actions making the blow-down morphism equivariant. The Atiyah--Bott localization theorem then means that the class of the exceptional divisor can be inverted, allowing us to retrieve a class on the blow-up from its pullback to the exceptional divisor.

The paper is organized as follows: In the first section we construct
the universally closed morphism which turns a proper local embedding $Y \to
X$ into an embedding. The first step is the case when the local
embedding is \'etale on its image. The general case is reduced to
this situation by flat stratification. A network of local embeddings
depending on the \'etale structure of the strata is highlighted in
section 1.2, and $X'$ is constructed by induction on strata, such
that all the local embeddings in the network are replaced by
embeddings. In section 1.3, universally closed morphisms and
push-forwards associated to probabilistic weights are defined. The
existence of an isomorphism between the Chow rings of $X'$ and of
the corresponding  network, as introduced in \cite{noi1}, is proven,
and the relation between the Chow rings of $X$ and $X'$ is
discussed. The second section of the paper contains the
calculation of Chern classes for weighted blow-ups. The third
section is dedicated to the example of the stable map spaces and its
intermediate weighted stable map spaces. The Appendix discusses the
Euler sequence of a weighted projective bundle. Although this
sequence is most likely known, we could not find a proof in the
literature, and so decided to carefully trace the sequence through
the groupoid presentation of the weighted projective bundle.

\textbf{Acknowledgements} The initial motivation for this article came from a question of Jason Michael Starr. We also wish to thank David Rydh
and the referee for most useful comments and suggestions.
The first ideas for this paper were formulated during our stay at the Mathematical Sciences Research Institute.
A first draft was finished during the second author's stay at
the Max Planck Institute for Mathematics in Bonn. We are grateful
for the hospitality of both institutes.
The authors acknowledge support by Science Foundation Ireland via
grant 08/RFP/MTH1759.

\section{The universal lift of a local embedding}

The stacks in this article are assumed to be algebraic in the sense
of Deligne--Mumford,  Noetherian, and all morphisms considered
between them are of finite type.
%ca sa avem "flat stratification"
%trebuie sa cerem " closed local embedding"? da
%sa nu uit ca vrem fibrele finite mai tarziu; rezulta din noetherian? +finite type

\subsection{The lift of a local embedding \'etale on its image.}

\begin{definition}
Following \cite{vistoli}, we will call local embedding any
representable unramified morphism of finite type of stacks. A
regular local embedding is a local embedding which is also locally a
complete intersection.
\end{definition}
%referintele din kresch and vistoli!

Given a proper local embedding $Y\to X$,
there is a diagram
$$\diagram
V_1 \dto^{p_1} \rto^{ g_1 } & U \dto^{p} \\
Y \rto^{g} & X,
\enddiagram $$
the vertical morphisms being \'etale atlases and $g_1$ being a closed
embedding of schemes (Lemma 1.19 in \cite{vistoli}). Let $V$ be the scheme representing the fiber
product $Y\times_X U$, with the induced map $g_U: V \to U$, and the
image of $g_U$ denoted by $W$.
 As $V_1\to Y$ and
$V \to Y$ are \'etale, then so must be the induced morphism $i_1:
V_1\to V$. Since $V_1\to U$ is an embedding, then so must be $V_1\to
V$.
%and thus the image of $V_1$
%in $V$ is a component of $V$ (\cite{SGA},
%Expos\'e 1, Corollary 5.2).
% atentie daca $V_1\to V$ poate fi open embedding or must be closed
  We write $V=V_1\cup V_2$, where $V_2$ is the closure in $V$ of $V\setminus V_1$.
  In fact, $V_1$ and $V_2$ are disjoint, as the map $V=V_1\cup V_2\ \to Y$ is \'etale.
  Denote by $p_i$ the restriction $p_i: V_i\to Y$,  by $W_i$ the image of $V_i$ in
  $U$.

Furthermore, if the morphism $g$ is \'etale on its image, then by the \'etale lifting property
(\cite{sga} I, Proposition 8.1.), $V_1$ and $U$ can be chosen such
that $$V_1=g(Y)\times_X U,$$ and as such, there is an \'etale
morphism $ h: V\to V_1$ with the property that $h\circ i_1= \mbox{
id }_{V_1}$.

 We recall the \'etale groupoid presentation $\left[ R\rightrightarrows U\right]$
 of the stack $X$, given by the two projection morphisms $ R:= U\times_X U \rightrightarrows U$,
 together with canonical morphisms $e$, $m,$ and $i$. Here the identity $e$ is the diagonal
 morphism $e: U \to U\times_XU$, the multiplication $m$ is
 $$m:=\pi_{13}: U\times_XU\times_XU \cong ( U\times_X U)\times_U (U\times_X U) \to U\times_X U,$$
 and the inverse morphism $ i: U\times_X U \to U\times_X U $ switches the two terms of the product.

\begin{notation}
 Let $$S_{ij}:=\Im (\phi_{ij}: V_i\times_Y V_j \to U\times_X U ),$$
for the map $\phi_{ij}$ given as a composition
 $$V_i\times_Y V_j \hookrightarrow V\times_Y V = V\times_Y(
 Y\times_X U) \cong V \times_X U \to U\times_X U.$$
  We denote by $R'$ the subset
$$R':= R \setminus (S_{12}\cup S_{21}\cup (S_{22}\setminus S_{11})) \cup \Im e.$$

From now on let $g: Y \to X$ be a proper local embedding. Then
$S_{ij}$ are closed subschemes of $R$. In general $S_{22}\setminus
S_{11}$ might not be closed, but it is so in the case when $g$ is
\'etale on its image, for dimension reasons. As $U$ is an \'etale
atlas of $X$, then the diagonal $e: U\to U\times_XU$ is an open
embedding, and thus, under the above condition on $g$, the subscheme
$R'$ of $R$ is open.

\end{notation}
\begin{proposition} \label{etale} Let $g: Y \to X$ be a proper morphism, \'etale on its
image. The restrictions $s_1, s_2: R' \to U$ of the two projection
morphisms $R\to U$, together with $e$, and with the restriction of
$m$ to $R'\times_U R'$ and of $i$ to $R'$ form the groupoid
presentation of a Deligne-Mumford stack $X_Y$.
\end{proposition}

\begin{proof}

 First note that
 %if $\{i, j\}= \{ 1, 2\}$, then
 $S_{12}\cap S_{11} =
 \emptyset$ and $S_{21}\cap S_{11} = \emptyset.$
Indeed, $$ U \times_XU  \hookleftarrow V_1 \times_XU\cong V_1
\times_Y \times (Y\times_XU)= V_1 \times_Y ( V_1 \sqcup V_2)= (V_1
\times_Y V_1) \sqcup (V_1 \times_Y V_2). $$
 In fact $R'$ may be written alternatively as the difference
$$R':= R \setminus (S_{12}\cup S_{21}\cup S_{22}\setminus S_{11}).$$
Indeed, $S_{12}\cap \Im e \subseteq S_{12} \cap S_{11}=\emptyset$, and
symmetrically $S_{12}\cap \Im e = S_{21}\cap \Im e =\emptyset.$
Furthemore, there is a sequence of consecutive Cartesian diagrams \bea
\begin{array}{ccccc} V_2 & \hookrightarrow & V & \longmapsto & U \\
                     \downarrow & & \downarrow & & \downarrow  \\
 V_2\times_YV & \hookrightarrow & V\times_YV & \longmapsto  & U\times_X U,
\end{array}
\eea where the vertical maps are diagonal morphisms, with
$V_2\times_YV_2 \subseteq V_2\times_YV$ and $V\times_YV\cong
V\times_X U$. This implies that $S_{22} \cap \Im e =  \Im e_2$, where for each $i\in \{1,2\}$,
 $e_i$ is the composition of the diagonal $V_i\to V_i\times_YV_i$ with the morphism $\phi_{ii}: V_i\times_YV_i \to U\times_X U$.
But $\Im e_2 = \Im e_1 \subseteq S_{11}$, due to the commutative diagram
\bea \diagram V_2 \rto\dto & V \rto^{h}\dto & V_1 \rto^{g_1}\drto^{e_1} & U \dto^e\\
V_2\times_YV_2 \rto & V\times_YV \rrto & & U\times_XU.
\enddiagram
\eea
As  $R'$ is both symmetric and reflexive by construction, the
existence of a groupoid structure $\left[ R' \rightrightarrows U
\right]$ reduces to checking the closure of $R'$ under
multiplication. We note that for $k,l\in \{1, 2 \}$, the preimage
$$\pi_{13}^{-1}(S_{kl}) \subseteq W \times_XU\times_X W \cong (W
\times_XU)\times_U (U\times_X W).$$ But $W \times_XU$  is the image
of $V \times_XU\cong \bigsqcup_{i,j \in \{ 1,2 \}  } V_i \times_Y V_j,$
namely $\bigcup_{i,j \in \{ 1,2\}  } S_{ij}$.

The injectivity of $g_1:V_1\to U$ directly implies $\pi_{13}(
S_{11}\times_US_{11})\subseteq S_{11}$.
% while $\pi_{13}(S_{11}\times_U \Im e_2)\subseteq S_{11}\cap S_{22}$ and $\pi_{13}(\Im e_2 \times_U \Im e_2)\subseteq \Im e$.
Thus the multiplication $m$ is well defined on $R'$.

\end{proof}

\begin{example}
Let $Z\hookrightarrow X$ be a closed embedding of stacks and let $Y:=\bigsqcup_{i=1}^nZ_i$, where $Z_i\cong Z$, with the natural morphism to $X$ identifying all copies of $Z$. Then $X_Y$ is obtained by gluing $n$ copies of $X$ along $X\setminus Z$.
\end{example}

With the notations from the previous proposition, the following also
holds.

\begin{proposition}\label{Y cong g(Y)xX}
There exist a natural embedding $Y \hookrightarrow X_Y$ and an
\'etale map $X_Y\to X$ making the following diagram Cartesian
$$   \diagram     Y  \rto \dto & X_Y \dto \\ g(Y) \rto  & X \enddiagram.  $$
\end{proposition}

\begin{proof}
 The composition of  \'etale morphisms $R' \hookrightarrow R \rightrightarrows U$
being \'etale, the natural morphism of groupoids
$$   \begin{array}{ccc}  R' & \longrightarrow  & R \\
\downdownarrows &  &\downdownarrows \\
                U  & \longrightarrow & U
\end{array} $$
 induces an \'etale morphism of stacks $X_Y\to X$. Next we check that
 $V_1= Y\times_{X_Y} U$, namely that the diagram defining a groupoid
 morphism
  $$   \begin{array}{ccc} V_1\times_Y V_1  & \longrightarrow  & R' \\
\downdownarrows &  &\downdownarrows \\
                V_1 & \longrightarrow & U
\end{array} $$
is Cartesian. Indeed,
 $$V\times_YV \cong V\times_XU\cong V\times_U (U\times_X U)= V\times_U R$$
 and therefore $$   V\times_U
 R'\cong (V\times_YV)\times_{R} R'\cong  (V_1\times_Y V_1) \sqcup
\phi_{22}^{-1}(S_{11}),$$ and by restriction to $V_1$,
$$ V_1\times_{U} R' \cong V_1\times_Y V_1.$$ This proves the existence of a
natural embedding of  $Y$ in $X_Y$ such that its composition with the
\'etale map $X_Y\to X$ yields the local embedding of $Y$ into $X$.

 At the beginning of this section, the cover $V_1$ of $Y$ was constructed as a
 fiber product
$V_1=g(Y)\times_X U,$ and thus $  V_1\times_YV_1 \cong V_1\times_{U}
R' \cong g(Y)\times_X R'$, forming an isomorphism  of groupoids
 $$    \left[ V_1\times_Y V_1 \rightrightarrows V_1 \right]   \cong
 \left[  g(Y)\times_X R' \rightrightarrows g(Y)\times_X U \right],$$
 which induces an isomorphism $Y\cong  g(Y)\times_XX_Y$.
\end{proof}

Although the groupoid presentation of the stack $X_Y$ depended on
particular choices of covers for $X$ and $Y$, the stack $X_Y$ is
uniquely defined by a universality property, which can be expressed
in terms of moduli problems as follows

%mai jos sa schimb schema $Z$ cu stack

\begin{theorem}\label{universality}
There is an equivalence of categories from the category of morphisms
$Z \to X_Y$ defined on stacks $Z$ of finite type to that of
morphisms $Z \to X$ endowed with a section
$$s: g(Y) \times_X Z \to Y \times_X Z $$
for the \'etale map $Y \times_X Z  \to g(Y) \times_X Z$.

\end{theorem}

\begin{proof}

Indeed, at the level of objects, given a map $Z \to X_Y$ from a
scheme $Z$ of finite type  to $X_Y$, and its composition with the
\'etale map $X_Y \to X$, there is an induced isomorphism
  $$ g(Y) \times_X Z \cong (g(Y) \times_X X_Y) \times_{X_Y} Z \cong Y \times_{X_Y}
  Z,$$
and a natural map $Y \times_{X_Y}Z \to Y \times_{X} Z$ which, when
composed with the \'etale map $g_Z: Y \times_{X} Z \to g(Y) \times_X
Z \cong Y \times_{X_Y} Z$, yields the identity map.

Consider now a scheme of finite type $Z$ and a morphisms $Z \to X$
represented by a morphism of groupoids
$$   \begin{array}{ccc}  R_Z & \longrightarrow  & R \\
\downdownarrows &  &\downdownarrows \\
                U_Z  & \longrightarrow & U
\end{array} $$
 Assume that there is a section $s: g(Y) \times_X Z \to Y
\times_X Z $ of $g_Z: Y \times_{X} Z \to g(Y) \times_X Z$,
represented by the groupoid map
$$   \begin{array}{ccc} V_1\times_{U}R_Z & \longrightarrow  & V \times_{U}R_Z \\
\downdownarrows &  &\downdownarrows \\
               V_1\times_{U} U_Z  & \longrightarrow & V \times_{U}
               U_Z.
\end{array} $$
 Then the map $ R_Z \to R$ factors through $ R_Z \to R'\to R$. Indeed, as the
restrictions of $R$ and $R'$ to $U\setminus \Im V_1$ coincide, it is
enough to show the factorization of groupoid maps
$$   \begin{array}{ccccc} V_1\times_{U}R_Z & \longrightarrow  & V_1 \times_{U}R' & \longrightarrow  & V_1 \times_{U}R \\
\downdownarrows &  &\downdownarrows &  &\downdownarrows \\
               V_1\times_{U} U_Z  & \longrightarrow & V_1 & \longrightarrow & V_1.
\end{array} $$
   The map $V_1\times_{U} U_Z \to V_1$ is the composition of  $V_1\times_{U} U_Z \to V \times_{U} U_Z$ above with the natural maps $V \times_{U} U_Z \to V \to V_1$, while
     $V_1\times_{U}R_Z \to V_1 \times_{U}R'$ is the composition
   $$      V_1\times_{U}R_Z \to V \times_{U}R_Z \to V \times_{U}R\cong V\times_YV \to V_1\times_YV \cong  V \times_{U}R' \to V_1\times_UR'\cong V_1\times_YV_1 .$$
 Together these maps induce the desired morphism $Z \to X_Y$.
A direct check shows that the constructions above are functorial, and that the two functors constructed between the category $\mbox{ Hom } (Z, X_Y)$ and that of morphisms $Z\to X$ with the extra property specified in the hypothesis are inverse to each other.
\end{proof}

\begin{remark}
Although the embedding $V_1 \hookrightarrow V$ does induce morphisms $V_1\times_{U} U_Z \to V \times_{U} U_Z$ and $V_1\times_{U}R_Z \to V \times_{U}R_Z$, together these induced morphisms do not in general form a morphisms of groupoids, as they are not necessarily compatible with the multiplication on $R_Z$. In particular, when $Z=X$, our construction does not imply the existence of a section for $ Y\to g(Y)$, or for $X_Y\to X$.
\end{remark}

\begin{corollary} \label{transitive}
a) Let $g:Y\to X$ be a morphism \'etale on its image, and let $Z \hookrightarrow Y$ be a closed embedding such that $g_{|Z}$ is proper and $Z\cong g(Z)\times_X Y$. Then there exists a natural morphism $g': Y \to X_Z$ \'etale on its image.

b) Furthermore, if $g$ is proper, then $g'$ is proper as well and
\bea (X_Z)_Y \cong X_Y.\eea
\end{corollary}

\begin{proof}  A canonical section \bea  s: g(Z) \times_X Y \to  Z \times_X Y\eea is given by the diagonal morphism $Z \to Z\times_X Z$ via the isomorphisms $ g(Z)\times_X Y\cong Z$ and $Z\times_X Z\cong Z\times_X Y$. By the previous theorem, this induces a morphism $g':Y \to X_Z$, making the triangles in the following diagram commutative.
\bea  \diagram  & {X_Z }  \rto   & X\\ Z \urto \rto & Y \uto^{g'} \urto_g \enddiagram \eea
 Also, $g'(Y)\cong X_Z\times_X g(Y)$. Indeed, by Proposition \ref{Y cong g(Y)xX}, $ g(Z)\times_XX_Z\cong Z \cong g'(Z)$, while $X\setminus g(Z) \cong X_Z\setminus g'(Z)$.
% Thus $g': Y \to X_Z$ is \'etale on its image.

 b) Assume $g$ is proper. As a consequence of the relation $g'(Y)\cong X_Z\times_X g(Y)$, $g'$ is closed. It is separated, because $g$ is. Moreover, $g'$ satisfies the valuative criterium of properness (\cite{vistoli}). Indeed, consider a valuation ring $R'$ with fraction field $K'$, and a commutative diagram
 \bea  \diagram  {\Spec K'} \rto \dto & {\Spec R'}\dto^{ r'} \\  Y \rto^{g'} & {X_Z}.   \enddiagram \eea
Since the composition $Y\to X_Z\to X$ is a proper morphism, there exists a finite extension $K$ of $K'$ such that, for the integral closure $R$ of $R'$ in $K$, the composition morphism $\Spec(K)\to Y$ extends to $\phi: \Spec R \to Y$.
Let $r$ denote the composition $\Spec R \to \Spec R' \to X_Z$.   It remains to check that $g'\circ \phi = r$, or, equivalently, via Theorem \ref{universality}, that the section
  \bea g(Z) \times_X \Spec R \to  Z \times_X \Spec R \eea
induced by $\phi$ is obtained from the canonical section $s$, after fiber product with $\Spec R$ over $Y$ via $r$:
\bea g(Z) \times_X Y\times_Y \Spec R  \cong g(Z)\times_X \Spec R  \to  Z \times_X Y\times_Y\Spec R \cong  Z\times_X\Spec R.\eea
  Indeed, this is the case as both morphisms $ g'\circ \phi$ and $r$ give the same morphism when composed with the projection $X_Z\to X$.  This proves the properness of $g'$.
Furthermore, as $g$ is \'etale on its image and $X_Z \to X$ is \'etale, $g'$ will be \'etale on the image as well. Thus we can construct $(X_Z)_Y$ as in Proposition \ref{etale}.

Note that, by the same argument as above with $\Spec R$ replaced by $Y \times_XX_Z$, the triangles in the following diagram are commutative:
\bea  \diagram     {Y \times_XX_Z} \rto^{p_1} \dto_{p_2} & {X_Z} \dto \\
Y \rto \urto^{g'} & {X.}  \enddiagram  \eea
We apply this to show that the functors of $(X_Z)_Y$  and $X_Y$ are equivalent. Indeed, any map $T \to X_Y$ induces a map $ T \to X$ and a section
\bea t:  g(Y)  \times_X T\to  Y \times_X T, \eea
which by restriction yields a section  $g(Z) \times_X T \to  Z \times_X T,$
and thus a map $T \to X_Z$. For the existence of a morphism $T\to (X_Z)_Y$, a section $t': g'(Y) \times_{X_Z} T \to  Y \times_{X_Z} T$ is necessary  and sufficient.
But  \bea & g'(Y)\times_{X_Z} T\cong ( g(Y)\times_X X_Z)\times_{X_Z} T\cong  g(Y)\times_X T, \mbox{ and }& \\
& Y\times_XT \cong   (Y \times_XX_Z)\times_{X_Z}T \cong Y  \times_{X_Z}T,  &\eea
due to the two commutative triangles above, so via these isomorphisms we can identify $t'$ with $t$.
Conversely, any map  $T\to (X_Z)_Y$ induces a map $T\to X_Z$ and a section $t': g'(Y)\times_{X_Z}T \to  Y\times_{X_Z}T$ which yields the section $t$ canonically via the above isomorphisms.

\end{proof}

\begin{corollary}\label{X_Y functorial} Consider any morphism of Noetherian stacks $f: X_1 \to X_2$ and any proper map $g_2:Y_2\to X_2$ \'etale on its image,  and let  $Y_1:= Y_2\times_{X_2}X_1$.
%Moreover, if the induced map $h:Y_1 \to Y_2$ is a closed embedding,
Then there exists a morphism $f':(X_1)_{Y_1} \to {(X_2)_{Y_2}}$ making the following diagram Cartesian
\bea     \diagram      (X_1)_{Y_1} \rto^{f'} \dto & {(X_2)_{Y_2}}  \dto \\
X_1 \rto^{f} & {X_2}. \enddiagram  \eea
 In particular, if $f$ is proper, then $f'$ is proper as well.
\end{corollary}

\begin{proof}
 The induced map $g_1: Y_1\to X_1$ is also \'etale on its image so that $(X_1)_{Y_1}$ satisfying the functorial property in Theorem \ref{universality} exists. Also, for any stack $Z$ with a morphism $Z \to X_1$, there are isomorphisms
 \bea     Y_2\times_{X_2} Z\cong  Y_2\times_{X_2} (X_1\times_{X_1}Z) \cong  ( Y_2\times_{X_2}X_1)\times_{X_1}Z \cong Y_1 \times_{X_1}Z \mbox{ and }  \\
 g_2(Y_2)\times_{X_2} Z\cong  g_2(Y_2)\times_{X_2} (X_1\times_{X_1}Z) \cong  ( g_2(Y_2)\times_{X_2}X_1)\times_{X_1}Z \cong g_1(Y_1) \times_{X_1}Z.
   \eea
Thus any morphism $Z\to (X_1)_{Y_1}$, corresponding to a section $g_1(Y_1) \times_{X_1}Z \to Y_1 \times_{X_1}Z$ by Theorem \ref{universality}, induces a section $g_2(Y_2) \times_{X_2}Z \to Y_2 \times_{X_2}Z$ and thus a morphism $Z\to (X_2)_{Y_2}$. Conversely, any two morphisms  $Z\to (X_2)_{Y_2}$ and $Z\to X_1$ making the diagram
\bea     \diagram      Z \rto^{} \dto & {(X_2)_{Y_2}}  \dto \\
X_1 \rto^{f} & {X_2}\enddiagram  \eea
commutative, factor through a unique morphism $Z\to (X_1)_{Y_1}$.
\end{proof}

% ca aplicatie: riemann-roch type theorems? R-R without denominators?

\subsection{ The \'etale structure of a local embedding. }
\label{etale cover}

Under certain assumptions, we can find a more precise description for the local \'etale structure of a local embedding.
We recall the following definitions which we will need for our assumptions.
%Let $\{S_i\}_i$ be the flattening stratification for the morphism $Y \to g(Y)$, and let $ Y_n \hookrightarrow Y_{n-1} \hookrightarrow ... \hookrightarrow  Y_1 \hookrightarrow Y_0=Y$ be a
%filtration of $Y$ consisting of the closures $\ol{g^{-1}(S_i)}$ in $Y$, for all $i$ as above. Thus the restrictions of $g$ $$ \begin{array}{cccc}  g_n: Y_n \to g(Y_n), & g^o_{n-1}: Y_{n-1} \setminus Y_{n} \to g(Y_{n-1} \setminus Y_{n}), & ..., & g^o_0:Y\setminus Y_1 \to  g(Y\setminus Y_1)    \end{array} $$are all \'etale.

%enumerate properties

\begin{definition}
A morphism of schemes $g : Y \to X$ is called an equidimensional
morphism if the following conditions hold:
\begin{enumerate}
\item $g$ is a morphism of finite type.
\item The function $\dim_yg^{-1}(g(y))$ is constant for all points of $Y$.
\item Any irreducible component of $Y$ dominates an irreducible component
of $X$.
\end{enumerate}
\end{definition}
As these conditions are stable under \'etale base change, (\cite{ega}, 13.3.8) the notion of equidimensional
morphism extends canonically to morphisms of stacks.

\begin{definition}
Let $Y$ be a Noetherian stack. It is called geometrically unibranch if it has an \'etale atlas $V$ which is geometrically unibranch, meaning that
 for any point $v$  of $V$, the scheme $\Spec \cO^{sh}_{v,V}$ is irreducible, where $\cO^{sh}_{v,V}$ denotes the strict henselization of the local ring of $u$ in $U$.
\end{definition}

\begin{notation} Consider a proper local embedding of Noetherian stacks $g: Y \to X$, with $Y$ reduced. Let $Y_1$ denote the locus of points in $Y$ where the map $g$ is not \'etale on its image. \end{notation}

\begin{proposition} \label{U} Let $g: Y \to X$ be a proper, local embedding of Noetherian stacks. Assume that $Y$ is reduced and geometrically unibranch, that the morphism on the image $Y \to g(Y)$ is equidimensional, and that its degree is equal to a fixed number $d$ at all points of $g(Y\setminus Y_1)$.

There exist an \'etale atlas $U$ of $X$ and closed, isomorphic subschemes $\{ W_l\}_{l\in L}$ of $U$,
with intersections $W_I= \bigcap_{l\in I} W_l$
%egal sau incluziune?
for each $I \in \cP(L)$, as well as isomorphic subschemes $\{ V_l^{a} \}_{l\in L, a\in A_l:=\{1,...,d\} }$ of $Y\times_XU$, each mapping onto $Y$, with isomorphisms $ V_l^{a} \to W_l$ standing over $g: Y \to g(Y)$, such that, at the level of supports,
 \bea  \begin{array} {ccc}  g(Y)\times_XU= \bigcup_{l\in L} W_l, &  Y\times_XU= \bigsqcup_{l \in L, a\in A_l} V_l^{a}, &  g(Y_1)\times_XU= \bigcup_{l,j\in L; l\not=j} W_{lj}.\end{array}
   \eea
   In particular $Y_1=g^{-1}(\{ p\in X; \mbox{  deg } g_p\geq 2d \})$.

More generally, for integers $k>1$, each $Y_{k-1}:=g^{-1}(\{ p\in X; \mbox{  deg } g_p\geq dk \})$ satisfies
\bea  \begin{array} {ccc}  g(Y_{k-1})\times_XU=\bigcup_{I \in \cP_{k}} W_I  & \mbox{ and }  & Y_{k-1}\times_XU = \bigcup_{I\in \cP_{k}, a\in A_I }V_I^{a},  \end{array} \eea
at the level of supports, where  $\cP_k \subset \cP(L)$ is the set of cardinality $k$--subsets of $L$; for each $I\in \cP_k$ as above, $A_I:=\bigsqcup_{l\in I} A_l$; for $l\in I$ and $a\in A_l$, the scheme $V_I^{a}$ is the preimage of $W_I$ in $V_l^{a}$.

 % mapping surjectively to $Y_{k-1}$.

% there exist sets $\{ V_I^{a} \}_{a\in A_I }$ mapping onto $Y_k$, with isomorphisms $ V_I^{a} \to W_I$ standing over $g_k: Y_k \to g(Y_k)$, and satisfying  $$Y_k\times_XU = \bigcup_{I\in \cP_k, a\in A_I }V_I^{a}.$$  Here $A_I=\bigsqcup_{k\in I} A_k$ and $V_I^{a} \subseteq V_k^{a}$ if $k\in I$ and $a\in A_k$.
%  $$Y_k\times_XU = \bigsqcup_{I\in \cP_k, a\in \{1,..., d_k\} }V_I^{a}.$$
\end{proposition}

\begin{proof}
%We first assume that $g(Y)$ is irreducible.

For any quasi-finite map $f$, we denote by $\mbox{md}(f)$ the maximum degree of $f$:
\bea\mbox{md}(f):= \mbox{max}_{p\in \Im f} \deg f_p.\eea
  The degree function $\deg g_p$ is upper semi-continuous. Let $Z_m$ denote the locus of points in $y\in Y$ where the degree of $g$ reaches its maximum, so that, unless $g$ is \'etale on its image,
\bean \label{degree ineq}  \mbox{md}(g_{|Z_m}) > \deg (g_{|Y\setminus Y_1})=d. \eean

 We will first focus on proving the existence of an \'etale atlas $U^0$ of $X$ and $\{W^0_{l}\}_{l\in L}$, $\{V^{0 a}_{l}\}_{l\in L, a \in A_l}$, such that
 \bea  Y\times_XU^0= \bigsqcup_{l, a\in A_l}V^{0 a}_{l},   \eea
and such that each $V^{0 a}_{l}$ is mapped isomorphically to $W_l^0\subseteq  g(Y)\times_XU^0$,  and \'etale surjectively to $Y$ by natural morphisms. We note that once such an atlas $U^0$ has been found, any other scheme $U'$ mapping \'etale surjectively to $U^0$ will work as well, as we can construct $W'_l$ and $V'_l$ as pullbacks of $W^0_{l}$ and $V^{0 a}_{l}$, respectively.

We proceed by induction on $\mbox{md}(g)$. If $\mbox{md}(g)=1$, then $g$ is an embedding and any \'etale atlas $U^0$ of $X$ would work, with $V^0=Y\times_XU^0$ \'etale atlas of $Y$, embedded in $U_0$. Assume now that $\mbox{md}(g)\geq 2$, and that the statement above can be proven for morphisms of smaller maximum degree.

Consider a commutative diagram like in Lemma 1.19 in \cite{vistoli},
\bea\diagram
{V_1} \dto_{\pi_{V_1}} \rto^{ g_1} & U \dto^{\pi_U} \\
Y \rto^{g} & X,
\enddiagram \eea
such that the vertical morphisms are \'etale and surjective, and $g_1$ is a closed embedding. The induced morphism ${V_1}\to Y\times_XU$ is a closed embedding, and \'etale, and so ${V_1}$ is a union of connected components of $Y\times_XU$. Let $g_U$ denote the morphism $Y\times_XU \to U$. The fibered product $Y\times_XU$ can be split into a disjoint union
                        \bea Y\times_XU = {V_1} \bigsqcup V'\bigsqcup V'',\eea
where $V''$ is the union of all components $V_0 \nsubseteq V_1$ with the property that $g_U(V_0) \subseteq g_U({V_1})$. We note that $Y\times_XU$ is geometrically unibranch, and so its connected components are also irreducible.
Let ${W_1}:= g_U({V_1})$, $W':= g_U(V')$ and $W'':= g_U(V'')$. By construction, $W'$ and $W_1$ are closed subsets of $U$ and none contains an irreducible component of the other.

First we notice that \bean \label{Y1} g(Y_1)= \pi_U( {W_1} \bigcap W'). \eean
 In other words, for any closed point $y$ of $Y$ and for any $v \in {V_1}$ such that $\pi_{V_1}(v)=y$, the map $g$ is \'etale on its image at $y$ if and only if $w:= g_U(v) \in W_1 \setminus W'$. Indeed, if $g$ is \'etale on its image at $y$, then the composition of the natural morphisms
 $ W_1 \cong V_1 \to Y \to g(Y)$, as well as the projection $W_1 \bigcup W' = g(Y) \times_XU \to g(Y)$ are \'etale at $v$, which makes the inclusion $W_1 \to W_1 \bigcup W'$ \'etale at $w=g_U(v)$, so $w \in W_1 \setminus W'$. Conversely, in the natural Cartesian square
 \bea  \diagram
{ Y\times_X(W_1\setminus W')} \rto^{} \dto & {W_1\setminus W'}\dto \dlto^{f_1}\\
 Y \rto & {g(Y)}
 \enddiagram
 \eea
 both the vertical morphisms and $f_1$ are \'etale, making  the upper horizontal map \'etale as well. Thus the restriction of $g$ to $g^{-1}(\pi_U(W_1\setminus W'))$ must be  \'etale.
 Here $f_1$ is the composition $W_1\setminus W' \to W_1\cong V_1 \to Y$.

% $\cO_{{W_1},g_U(v)} \cong \cO_{g(Y), g(y)}\otimes_{\cO_{X, g(y)} } \cO_{U, g_U(v)}$.
%de verificat !
% there exists a neighboorhood $U_y$ of $g_U(v)$ in $U$ such that  $g(Y)\times_XU_y= {W_1} \bigcap U_y$.  trebuie transcrisa infinitezimal\oint

  Both $g_{U|_{V''}}$ and $g_{U|_{V'}}$ are proper local embeddings. Clearly the definition of $V''$ implies an equation in maximum degrees:
  \bean \label{md''} \mbox{md}(g_{U|_{V''}})<\mbox{md}(g_U). \eean

 Assuming $V''\not=\emptyset$, we note that $g_U(V'')={W_1}$. This is true as $g_U(V'')$ is closed in $U$ due to the properness of $g_U$, and ${W_1}\setminus W' \subseteq g_U( V'')$.
  Indeed, as $W_1 \bigcap W'$ has positive codimension in $W_1$, while $\dim g_U(V'')=\dim  W_1$, it follows that $ g_U(V'') \bigcap ({W_1}\setminus W') \not=\emptyset$. Due to the definition of $V''$, the degree of $g_U$ on this set is at least $2$.  But  equation (\ref{Y1}) and the condition that $\deg g_{|Y\setminus Y_1}$ is constant imply that $\deg g_U$ is constant on ${W_1}\setminus W'$ and as such, at least $2$ at every point there. This can only happen if  ${W_1}\setminus W' \subseteq g_U( V'')$ and therefore, if $g_U(V'')={W_1}$.

 Next we show that for appropriate choices of $U$, there is also an inequality in degrees
 \bean \label{md'} \mbox{md}(g_{U|_{V'}})<\mbox{md}(g_U).\eean
Consider $E$ a connected, closed subscheme of $g(Y)\times_XU$, and assume that the restriction of $g_U$ on  $g_U^{-1}(E)$ is \'etale on its image. Then we can assume that either $E \bigcap {W_1}=\emptyset$ or $E\subseteq {W_1}$. Furthermore, in the
 first case, $U$ can be replaced by $U \setminus E$ without changing ${V_1}$, or any of the properties of $U$ above.

 Indeed, the restriction of $g_U$ on  $g_U^{-1}(E)$ can be split as
\bea  (g_U^{-1}(E)\bigcap ({V_1}\bigsqcup V''))     \bigsqcup  (g_U^{-1}(E)\bigcap V')   \to     (E\bigcap {W_1}) \bigcup (E\bigcap W')=E,\eea
and thus either $E\bigcap {W_1}=\emptyset$, or $E\bigcap W'=\emptyset$ or $E \subseteq {W_1} \bigcap W'.$

 In particular, taking  $E=\pi_U^{-1}(g(Z_m))$, for an appropriate choice of $U$ we may assume that
 \bea  \pi_U^{-1}(g(Z_m)) \subseteq {W_1}.\eea
 This, together with (\ref{degree ineq}), imply that $\mbox{md}(g_{U|_{V'}})<\mbox{md}(g_U).$

 Moreover, the degrees of $g_{U|_{V'}}$ and $g_{U|_{V''}}$ are constant over all points in $W'\setminus W_1$ and $W_1 \setminus W'$, respectively.
 Due to equations (\ref{md''}) and (\ref{md'}), we can now apply the induction hypothesis to $g_{U|_{V'}}$ and $g_{U|_{V''}}$, obtaining two surjective, \'etale morphisms $U'\to U$ and $U''\to U$ and $\{V^{a'}_{l'}\}_{l'\in L'}$,  $\{V^{a''}_{l''}\}_{l''\in L''}$ such that
 \bea  \begin{array}{ll}  V'\times_UU' =\bigsqcup_{l', a'}V^{a'}_{l'},    &  V''\times_UU'' =\bigsqcup_{l'', a''}V^{a''}_{l''},  \end{array} \eea
the maps $V^{a'}_{l'} \to U'$ and $V^{a''}_{l''} \to U''$ are closed embeddings, while $V^{a'}_{l'} \to V'$ and $V^{a''}_{l''} \to V''$ are surjective, \'etale. Consider now $U^0=U'\times_UU''$, as well as $V^{0 a'}_{l'}:=V^{a'}_{l'}\times_{U'}U^0$ and $V^{0 a''}_{l''}:=V^{a''}_{l''}\times_{U''}U^0$. Finally, let $V^0_1=V_1\times_UU^0$ and $L=L'\sqcup L''\sqcup\{1\}$. The schemes $\{V^{0 a}_{l}\}_{l\in L}$ thus obtained satisfy
\bea  Y\times_XU^0= \bigsqcup_{l, a\in A_l}V^{0 a}_{l}   \eea
and each $V^{0 a}_{l}$ is naturally embedded into $U^0$.

The surjectivity of the maps $V^{0 a}_{l}\to Y$ remains to be established. Since $V^{0 a'}_{l'} \to V'$ and $V^{0 a''}_{l''} \to V''$ are known to be surjective from above, this reduces to showing that both $V'$ and $V''$ map surjectively to $Y$, (unless $V''=\emptyset$).

 Now, if for any  closed point $y_1$ of  $Y\setminus Y_1$ we have $(\{y_1\}\times_XU)\bigcap V'=\emptyset$, then consider $v_1\in (\{y_1\}\times_XU)\bigcap {V_1}$  and replace $U$ by $U\bigsqcup U_{g_U(v_1)}$, where  $U_{g_U(v_1)}\subset U\setminus W'$ is an open neighboorhood of $g_U(v_1)$ in $U$. Thus $V'$ is replaced by $V' \bigsqcup g_U^{-1}(U_{g_U(v_1)})$, while the maximum degree of $g_U$ remains unchanged. Since $Y$ is Noetherian,  we may thus assume that all points in $Y\setminus Y_1$ have  preimages in $V'$. They have preimages in $V''$ as well, as long as $V''\not=\emptyset$.
 Indeed, if $\deg g_p\geq 2$ for any $p\in g(Y)$, then for any $y_1$ closed point in  $Y\setminus Y_1$, choose $y_2\in Y$ such that $g(y_1)=g(y_2)$, and $v_2\in (\{y_2\}\times_XU)\bigcap {V_1}$. Then since $g_{U_{|V_1}}$ is injective, $\emptyset\not=\{y_1\}\times_X\{ g_U(v_2) \} \subset V''$.

If $y_1\in Y_1$, we can prove $(\{y_1\}\times_XU)\bigcap V'\not=\emptyset$ due to the structure of $g(Y)\times_XU$ around $\{g(y_1)\}\times_XU$. For this, first let $y_2\in Y_1$ be such that $g(y_2)=g(y_1)$, and $v_2\in (\{y_2\}\times_XU)\bigcap V'$ (conform relation (\ref{Y1})). Consider a valuation ring $R_1$ and a map $T_1=\Spec R_1 \to Y$, taking the closed point to $y_1$ and the generic point $(0)$ to $z_1\in Y\setminus Y_1$. Construct another map $T_2=\Spec R_2 \to Y$ taking the closed point of the valuation ring $R_2$ to $y_2$ and the generic point  to $z_2\in Y\setminus Y_1$, such that $g(z_1)=g(z_2)$ (such a map exists because $g$ is proper), as well as $T'_2=\Spec R'_2 \to V'$ taking the closed point to $v_2$ (since $V'\to Y$ is \'etale). Let $C_2$ denote its image in $U$. Then the generic point of $\ol{\{z_1\} }\times_X C_2$ must be in $V'$ and by specialization, we get a point in $V'$ whose image in $Y$ is $y_1$.

  The same argument works for $V''$ as well. With this, the proof that $V'$ and $V''$ map surjectively to $Y$ is complete.

   Consider now the \'etale atlas $U^0$ of $X$ constructed above and let $W^0_l$ denote the image of $V^{0 a}_{l} \to U$ for each $l$.
    The arguments in the proof of relation (\ref{Y1}), when applied successively to $g_{U^0}$ and its restrictions to each $Y_k\times_XU_0$,     yield
\bea  g(Y_k)\times_XU^0 = \bigcup_{I \in \cP_k} W^0_I.   \eea
However, each $ W^0_I$ may not necessarily map surjectively to $g(Y_k)$. To adjust this, we consider the permutation group $S_L$ of $L$ and relabel \bea U:=\bigsqcup_{\sigma\in S_L} U^0, \mbox{ while } W_l:= \bigsqcup_{\sigma\in S_L} W^0_{\sigma(l)} \mbox{  and }   V^a_l:= \bigsqcup_{\sigma\in S_L} V^{0 a}_{\sigma(l)} \mbox{ for each } l\in L, a\in A_l,\eea
(for fixed bijections between the sets $A_l$). Then $U$ satisfies all the conditions required in the Proposition. Moreover, we note that now $W_l\cong W_j$ for any $l$ and $j$ as above.

%If $g(Y)$ is not irreducible, $U^0$ may be constructed in successive steps for each component of $g(Y)$, such that at the last step, each $W^{0}_{l}$ maps surjectively to one of the components of $g(Y)$. Then $U$ defined as above satisfies all the desired properties.

\end{proof}
% reduced?

%de comparat cu constructia lui Kresch. De discutat existenta mai multor embeddings posibile --mai multe deformari la conul normal?

\begin{definition} \label{notations for covers}
Consider a proper local embedding of Noetherian stacks $g: Y \to X$. Assume that there exists  an \'etale atlas $U$ of $X$ with all the properties listed in Proposition \ref{U}.
With the notations from the same Proposition, let
\bea V:=Y\times_XU=\bigsqcup_{i,a}V_i^a \mbox{ and } V_i:= \bigsqcup_aV_i^a\eea for each $i \in L$, where $a\in A_i$. For any $i,j\in L$, $a\in A_i$ and $b\in A_j$,  we denote \bea  S_{ij}^{ab}:= \Im ( V_i^a \times_Y V_j^b \to W_i \times_X W_j).\eea
 For some fixed  $i\in L$ and $a\in A_i$, we define $R_{Z,i}:=\bigsqcup_{b\in A_i} S_{ii}^{ab}$.
\end{definition}

\begin{lemma}
 With the notations from Definition \ref{notations for covers}, $R_{Z,i}$ does not depend on the choice of $a$.
\end{lemma}

 \begin{proof} For any $i,j\in L$, $a,c\in A_i$ and $b\in A_j$,  the sequences of Cartesian diagrams
 \bea \diagram {(V_i^a\times_Y V_j^b)\times_{U\times_XU}(V_i^c\times_Y V)} \rto \dto & {V_i^c\times_Y V} \rto \dto & V_i^c  \dto \\
 V_i^a\times_Y V_j^b \rto & { U\times_XU} \rto & U \enddiagram
 \eea
yield a canonical isomorphism $F: V_j^b \times_Y V_i^a\times_{U}V_i^c   \cong  (V_i^a\times_Y V_j^b)\times_{U\times_XU}(V_i^c\times_Y V)$. These spaces are also isomorphic to $(V_i^a\times_U V_i^c)\times_{Y\times_XY} (V_j^b \times_U V)$ over $V$, e.g. due to the sequence of Cartesian diagrams
\bea \diagram {(V_i^a\times_U V_i^c)\times_{Y\times_XY}(V_j^b\times_U V)} \rto \dto & {V_j^b\times_U V} \rto \dto & V_j^b  \dto \\
 V_i^a\times_U V_i^c \rto & { Y\times_XY} \rto & Y. \enddiagram
 \eea
 \bean \label{indep2} (V_i^a\times_Y V_j^b)\times_{U\times_XU}(V_i^c\times_Y V_j)\cong (V_i^a\times_U V_i^c)\times_{Y\times_XY}(V_j^b\times_U V_j)\eean
 contains the image of  $(V_j^b\setminus \bigcup_{k\not=j}V^b_{jk})\times_Y V_i^a\times_{U}V_i^c$ through $F$ and, as the codimension of $\bigcup_{k\not=j}V^b_{jk}$  in $V_j^b$ is at least 1,  $F$ can also be understood as
 \bean \label{indep of upper indices} V_j^b \times_Y V_i^a\times_{U}V_i^c  \cong  (V_i^a\times_Y V_j^b)\times_{U\times_XU}(V_i^c\times_Y V_j).\eean  In particular, when $i=j$, this yields $\bigsqcup_{b\in A_i} S_{ii}^{ab}= \bigsqcup_{b\in A_i} S_{ii}^{cb}$ for any $a, c\in A_i$.

\end{proof}

\begin{lemma} \label{split into etale and generically deg. 1}
Let $g: Y \to X$ be a proper local embedding of Noetherian stacks,   and assume that the degree of $g$ is the same at all points of $g(Y\setminus Y_1)$.
 Assume that there exists  an \'etale atlas $U$ of $X$ satisfying the conditions in Proposition \ref{U}. Then there exists a stack $Z$, an \'etale morphism $f: Y\to Z$ and a proper local embedding $h: Z\to X$ of generic degree 1, such that $g=h\circ f$.

Moreover, with the notations from Proposition \ref{U}, \bea \bigsqcup_{l\in L} W_l\cong  Z\times_XU.\eea
\end{lemma}

\begin{proof}
 The restrictions to $R_{Z,i}$ of the two natural projections on $W\times_XW \rightrightarrows W$ yield a groupoid scheme $ R_{Z,i} \rightrightarrows W_i$. Indeed, the identity $e: W_i \to S_{ii}^{aa}\subset R_{Z,i}$ is induced by the diagonal $V_i^a \to V_i^a\times_YV_i^a$, and there are natural inverse and multiplication maps, built from $i: S_{ii}^{ab}\to S_{ii}^{ba}$ and  $m:  S_{ii}^{ab}\times_{U}S_{ii}^{bc} \to S_{ii}^{ac} $, the last of which can be identified with the projection on the first and third factors $V_i^a\times_Y V_i^b \times_Y V_i^c \to V_i^a\times_YV_i^c$. We will denote by $Z$ the Deligne-Mumford with groupoid presentation $ \left[ R_{Z,i} \rightrightarrows W_i \right]$.

 Due to (\ref{indep of upper indices}), the following is a Cartesian diagram of groupoid schemes,
\bea    \begin{array}{ccc}  \bigsqcup_{c,d} V_i^c\times_YV_i^d & \longrightarrow  & R_{Z,i} \\
\downdownarrows &  &\downdownarrows \\
              \bigsqcup_{c} V_i^c    & \longrightarrow & W_i,
\end{array}  \eea
proving the existence of an \'etale morphism $f:Y\to Z$.

There is also a natural morphism of groupoid schemes from $ R_{Z,i} \rightrightarrows W_i$ to $W_i\times_XW_i  \rightrightarrows W_i$, which is moreover an isomorphism over $W_i\setminus (\bigcup_{j\not=i}W_j)$. This yields $g: Z\to g(Y) \hookrightarrow X$, generically one-to-one.

 It remains to construct an \'etale map $\bigsqcup_jW_j \to Z$ and prove that \bea Z\times_XU \cong Z\times_{g(Y)}(g(Y)\times_XU)\cong Z\times_{g(Y)}(\bigcup_{j}W_j)\cong \bigsqcup_{j}W_j. \eea Let $W:=\bigcup_{j}W_j$. As $W_i\times_XW \cong V_i^a\times_XU\cong V_i^a\times_Y(Y\times_XU) \cong \bigsqcup_{j,b} V_i^a\times_YV_j^b,$ the projection from $W_i\times_XW$ to the second term factors through $p_2: W_i\times_XW \to \bigsqcup_jW_j$. Moreover, $p_2$ is \'etale and surjective. Next, we check the existence of a canonical isomorphism of groupoids
\bean  \label{W_itimesW}  \begin{array}{ccc}  (W_i\times_XW)\times_{\bigsqcup_jW_j}(W_i\times_XW)  & \longrightarrow  & \bigsqcup_bS_{ii}^{ab}\times_XW \\
\downdownarrows &  &\downdownarrows \\
             W_i\times_XW     &  \longrightarrow & W_i\times_XW,
\end{array}
 \eean
which, after descent, will induce the map $\bigsqcup_jW_j \to Z$. Equation (\ref{W_itimesW}) follows from a sequence of isomorphims
\bea  &  (W_i\times_XW)\times_{\bigsqcup_jW_j}(W_i\times_XW) \cong (V_i^a\times_YV)\times_{\bigsqcup_jW_j}(V_i^e\times_YV) \cong &  \\
&\cong \bigsqcup_{j,c,d} (V_i^a\times_YV_j^c)\times_{U}(V_i^e\times_YV_j^d)  \cong  \bigsqcup_{j,c,d} V_i^a\times_Y\left( (V_j^c\times_{U}V_j^d)\times_{Y\times_XY}(V_i^e\times_UV_i) \right) & \\
&  \cong \bigsqcup_{j,c}V_i^a\times_Y\left( (V_j^c\times_{U}V)\times_{Y\times_XY}(V_i^e\times_UV_i) \right) \cong \bigsqcup_{j,c}V_i^a\times_Y(V_j^c\times_{Y}(V_i^e\times_UV_i)) &\eea
(based on isomorphisms (\ref{indep2}) and (\ref{indep of upper indices})), and furthermore
  \bea  \cong \bigsqcup_{j,c}V_i^a\times_Y(V_j^c\times_{Y}(V_i^e\times_{W_i}V_i)) \cong V_i^a\times_Y V\times_{Y}V_i \cong \bigsqcup_bS_{ii}^{ab}\times_XW.
\eea
  Finally, we note that $(\bigsqcup_jW_j)\times_XW\cong (\bigsqcup_jW_j)\times_U(U\times_XU)\rightrightarrows \bigsqcup_jW_j$ is the pull-back on  $\bigsqcup_jW_j$ of the groupoid presentation of $X$, and that the fibered product via $p_2$ on the left, and then $p_1: W_i\times_XW \to W_i\to \bigsqcup W_j$ on the right end,
  \bea  & (W_i\times_XW)\times_{\bigsqcup_jW_j}(\bigsqcup_jW_j\times_XW)\times_{\bigsqcup_jW_j}(W_i\times_XW)\cong &\\
  &\cong W\times_X (W_i\times_XW)\times_{\bigsqcup_jW_j}(W_i\times_XW) \cong &\\
  &\cong  W\times_X (\bigsqcup_bS_{ii}^{ab}) \times_XW \cong (W_i\times_XW)\times_{W_i}(\bigsqcup_bS_{ii}^{ab})\times_{W_i}(W_i\times_XW) \eea
give alternative groupoid presentations of $Z$, and so $Z\times_XU\cong \bigsqcup_{j}W_j.$

\end{proof}

\begin{definition}\label{network} Let $g: Y \to X$ be a proper local embedding of Noetherian stacks,  and assume that the degree of $g$ is the same at all points of $g(Y\setminus Y_1)$. Assume that there exists  an \'etale atlas $U\to X$ satisfying the properties from Proposition \ref{U}. The network of local embeddings of $g$ and $U$  is a set of stacks $\{Y_I\}_{I\in \cP(L)}$ and morphisms
$\phi_J^I: Y_J \to Y_I$ for each pair $I\subseteq J$, $I \in \cP_i$ and $J\in \cP_j$, constructed as follows:
\begin{enumerate}
\item If $g$ factors through $f:Y\to Z$ \'etale and $h:Z\to X$ of generic degree one on its image, like in the previous Lemma, then we define $Y_I:=Z_I\times_ZY$, where the network $\{ Z_I,  \varphi_J^I\}$ of $Z$ is constructed as below. The morphisms  $\phi_J^I: Y_J \to Y_I$ are also obtained by pull-back from the network of $Z$.
\item If $g$ is generically one-to-one on its image, then:
\end{enumerate}
 Let $Y_{\emptyset}:=X$ with the given presentation $ \left[ R_{\emptyset}:= U\times_XU \rightrightarrows U \right]$. When $I=\{ i\} \in \cP_1$, then $Y_i\cong Y$, having thus a groupoid presentation $R_i \rightrightarrows V^a_i $ where \bea R_i = V^a_i\times_YV^a_i\cong W_i\times_XW_i\setminus \bigcup_{j\not=i} S_{ij}^{ab}, \eea
for $ S_{ij}^{ab}= \Im ( V_i^a \times_Y V_j^b \to W_i \times_X W_j)$. (In this case the indices $a$ and $b$ are uniquely associated to $i$ and $j$, respectively, but we keep employing upper indices as for $i, j\in I$, we can thus discriminate between $V_i^a \supseteq V_I^a\not=V_I^b\subseteq V_j^b$).

  For any  $I \in \cP_k$,  define $Y_I$ as the stack of groupoid presentation
\bea  \left[  R_I = (\prod_{i\in I})_{R_{\emptyset} } R_i \rightrightarrows   V_I^a
  \cong W_I =( \prod_{i\in I} )_{U}W_i \right], \eea
  where $(\prod)_{R_{\emptyset} }$ denotes the fiber product over  $R_{\emptyset}$, the groupoid structure is induced from  $ \left[ R_i \rightrightarrows V_i^a \cong W_i \right]$, and $a$ is the index associated to some $i\in I$.
\end{definition}

  For any $K, J$ there is a natural isomorphism $R_{J \cup K} \cong R_J\times_{R_{J\cap K}}R_K$.  By induction
  we obtain that for a fixed index $a$,
  \bean \label{R_I} R_I = V_I^a\times_{Y_I} V_I^a \cong W_I\times_X W_I \setminus \bigcup_{j\not= i \in I} S_{ij}^{ab}. \eean

  The morphisms $\phi_J^I: Y_J \to Y_I$ for $J\supseteq I$ correspond to the natural morphism between the groupoid presentations
  $ \left[  R_J \rightrightarrows V^a_J \right]$ and $\left[  R_I \rightrightarrows V^a_I \right]$. In particular, $\phi_I^I=\mbox{id}_{Y_I}$.

 \begin{notation} If $I$ is a set and $h\not\in I$, we will write $Ih:= I \cup \{ h\}$.

 Alternatively, the objects $Y_I$ and the morphisms $\phi_J^I: Y_J \to Y_I$ are uniquely defined by the following Lemma.

 \begin{lemma} \label{fibered product} For any $I$, $K\in \cP_k$ and $b\in A_K$, there exists a natural \'etale map $V^b_K \to Y_I$ and
  \bean \label{fibered product1} Y_{Ih }\times_{Y_I} V_K^b \cong \bigsqcup_{j\in L\setminus K} V_{Kj}^b. \eean
  More generally, for all
$J\in \cP_l$ and $I\in \cP_k$ such that $J\supset I$,
\bean \label{fibered product2} Y_J\times_{Y_I} V_I^a \cong \bigsqcup_{K=K_l\supset K_{l-1} \supset ... \supset K_{k+1} \supset K_{k} =I } V_K^a,   \eean
 for all chains $K_l\supset K_{l-1} \supset ... \supset K_{k+1} \supset K_{k} =I$ with $K_s\in \cP_s$ for each $s$. When $J=\emptyset,$
 \bean \label{fibered product3} Y_J\times_{X} U \cong \bigsqcup_{a, K=K_k\supset K_{k-1} \supset ... \supset K_{1} \supset K_{0} =\emptyset } V_K^a, \eean
for $a\in K_1$.
 \end{lemma}

 \begin{proof} It is enough to prove the Lemma in the case when $g$ is generically one-to-one, due to Lemma \ref{split into etale and generically deg. 1} and the construction of the network. Once (\ref{fibered product1}) is known for all $I\in \cP_{k-1}$, the  isomorphisms (\ref{fibered product2}) and (\ref{fibered product3}) follow for all $I\in \cP_{k}$ by successive applications of (\ref{fibered product1}). We will prove (\ref{fibered product1}) by increasing induction on $k$. The case $k=1$ results from the definition of $Y_i$-s. Assume that (\ref{fibered product1}) holds for all $I\in \cP_{k-1}$. Choose $I$ and $K\in \cP_{k}$.

 We will denote $V:=Y\times_XU$ and $V_I:=\bigsqcup_{a, K=K_k\supset K_{k-1} \supset ... \supset K_{1} \supset K_{0} =\emptyset } V_K^a \cong Y_I\times_{X} U$ (from the induction hypothesis).
%The proof is largely similar to that of Lemma \ref{split into etale and generically deg. 1}.
  We will construct an \'etale, surjective morphism $\bigsqcup_{j\in L} V_{Kj}^b \to Y_{Ih}$. For this, recall the groupoid presentation $\left[ R_{Ih} \rightrightarrows V^a_{Ih}  \right] $ of $Y_{Ih}$. We note that the first projection $p_1: V^a_{Ih}\times_{Y_I}V_K^b \to V^a_{Ih}$ is \'etale and surjective. We will prove the existence of another \'etale, surjective morphism
 \bean  \label{moph of groupoids} p_2:   V^a_{Ih}\times_{Y_I}V_K^b = V^a_{Ih}\times_{Y_I}\bigcup_jV^b_{Kj} \to \bigsqcup_{j\in L} V_{Kj}^b,\eean
 and a morphism of groupoid schemes
 \bean \label{R_{Ih}} \begin{array}{ccc}  (V^a_{Ih}\times_{Y_I}V_K^b)\times_{\bigsqcup_jV^b_{Kj}}(V^a_{Ih}\times_{Y_I}V_K^b) & \longrightarrow  & R_{Ih} \\
\downdownarrows &  &\downdownarrows \\
              V^a_{Ih}\times_{Y_I}V_K^b    & \longrightarrow & V^a_{Ih}.
\end{array} \eean
  Indeed, $V^a_{Ih}\cong W_{Ih}\cong V^a_I\times_U V^c_h$, inducing an isomorphism \bea \varphi: V^a_{Ih}\times_{Y_I}V_K^b \to (V^a_I\times_U V^c_h)\times_{Y_I}V_K^b.\eea As $V_K^b\times_UV\cong V_K^b\times_U (Y\times_XU)\cong V_K^b\times_XY$, there is a sequence of Cartesian diagrams
 \bea  \diagram  {(V_{I}^a\times_{U}V_h^c)\times_{Y_I\times_XY}(V_{K}^b\times_{U}V)}\rto \dto^{\pi_2} & {V_{K}^b\times_{U}V} \rto \dto  &  V_K^b
\dto\\
              {V_{I}^a\times_{U}V_h^c}       \rto &   {Y_I\times_XY }  \rto & {Y_I},
\enddiagram \eea
 giving a canonical isomorphism $F:(V_{I}^a\times_{U}V_h^c)\times_{Y_I}V_K^b \cong (V_{I}^a\times_{U}V_h^c)\times_{Y_I\times_XY}(V_{K}^b\times_{U}V)$. Via the isomorphisms $\varphi$ and $V^b_K\times_U V^d_j \cong V_{Kj}^b$ for any $j$,  we can now define $p_2:  V^a_{Ih}\times_{Y_I}V_K^b \to \bigsqcup_{j\in L} V_{Kj}^b$ as $\pi_2$, which is moreover \'etale.

  Via the isomorphism $F$, the fiber product space $(V^a_{Ih}\times_{Y_I}V_K^b)\times_{\bigsqcup_jV^b_{Kj}}(V^a_{Ih}\times_{Y_I}V_K^b)$ is isomorphic to
  \bea & (V_{I}^a\times_{U}V_h^c)\times_{Y_I\times_XY}(V_{K}^b\times_{U}V)\times_{\bigsqcup_jV^b_{Kj}}(V_{K}^b\times_{U}V)\times_{Y_I\times_XY}(V_{I}^a\times_{U}V_h^c) \cong &\\
  &   (V_{I}^a\times_{U}V_h^c)\times_{Y_I\times_XY}(V_{K}^b\times_{U}V)\times_{Y_I\times_XY}(V_{I}^a\times_{U}V_h^c)  \cong  &\\ &  (V_{I}^a\times_{U}V_h^c)\times_{Y_I\times_XY}(V_{I}^a\times_{U}V_h^c)\times_{Y_I} V_{K}^b. & \eea
 On the other hand,  as \bea R_{Ih}\cong (V_I^a\times_{Y_I}V_I^a)\times_{U\times_XU}(V^c_h\times_YV^c_h) \cong (V_I^a\times_UV^c_h)\times_{Y_I\times_XY}(V_I^a\times_UV^c_h), \eea there is a natural projection
  \bea (V^a_{Ih}\times_{Y_I}V_K^b)\times_{\bigsqcup_jV^b_{Kj}}(V^a_{Ih}\times_{Y_I}V_K^b)  \longrightarrow   R_{Ih} \eea
making (\ref{R_{Ih}}) into a morphism of groupoid schemes. Finally, due to the isomorphisms above, the \'etale atlas $V_{Ih}^a\times_{Y_I}V_K^b$ of $Y_{Ih}$ yields a groupoid presentation with relations
\bea  & (V_{Ih}^a\times_{Y_I}V_K^b)\times_{Y_{Ih}}(V_{Ih}^a\times_{Y_I}V_K^b) \cong
 (V_{Ih}^a\times_{Y_I}V_K^b)\times_{V_{Ih}^a}R_{Ih}\times_{V_{Ih}^a}(V_{Ih}^a\times_{Y_I}V_K^b) \cong  &\\
&\cong V_K^b\times_{Y_I}(V_I^a\times_UV^c_h)\times_{Y_I\times_XY}(V_I^a\times_UV^c_h)\times_{Y_I} V_{K}^b \cong &\\
&\cong V_K^b\times_{Y_I}(V^a_{Ih}\times_{Y_I}V_K^b)\times_{\bigsqcup_jV^b_{Kj}}(V^a_{Ih}\times_{Y_I}V_K^b) \cong & \\
&\cong (V^a_{Ih}\times_{Y_I}V_K^b)\times_{\bigsqcup_jV^b_{Kj}}(\bigsqcup_jV^b_{Kj} \times_{Y_I}V_K^b)\times_{\bigsqcup_jV^b_{Kj}}(V_K^b\times_{Y_I}V^a_{Ih}),& \eea
while $(\bigsqcup_jV^b_{Kj} \times_{Y_I}V_K^b) \rightrightarrows \bigsqcup_jV^b_{Kj}$  is the pull-back of $V_K^b\times_{Y_I}V_K^b \rightrightarrows V^b_{K}$. This completes the proof of isomorphism (\ref{fibered product1}).

 \end{proof}

\begin{corollary}  \label{self-similarity} The morphisms $\phi_J^I: Y_J \to Y_I$ for $J\supset I$ are proper local embeddings.

Let $J=I \bigcup\{j\}$ for some $j\not\in I$. If $U$ is an \'etale atlas of $X$ satisfying the properties listed in Proposition \ref{U} for the morphism $Y\to X$, then $W_I\cong V^a_I$ is an \'etale atlas of $Y_I$ satisfying the same set of properties for the morphism $\phi_J^I: Y_J \to Y_I$, and the network of local embeddings associated to the local embedding $\phi_J^I$ with the \'etale atlas $W_I$ consists of $\{ \phi_K^H: Y_K\to Y_H \}_{K\supset H \supseteq J}$.

\end{corollary}
\begin{proof} The properness of the morphisms $\phi_J^I: Y_J \to Y_I$  is a direct consequence of formula (\ref{fibered product2}) and the fact that
$g:Y \to X$ is proper. Formulas (\ref{fibered product1}) and (\ref{fibered product2}) also show that $W_I\cong V^a_I$ satisfies the properties listed in Proposition \ref{U}. Thus for $J=I \bigcup\{j\}$, the network of local embeddings associated to $\phi_J^I$ with the \'etale atlas $V^a_I$ is made out of stacks with \'etale atlases given by $W_K$ for some $K\supseteq J$, and relations  given by
\bea  (\prod_{k\in K\setminus I})_{R_{I}}R_{Ik} \cong R_{ K}, \eea
 due to Definition \ref{network}. The stacks in the network are thus $\{Y_K\}_{K\supseteq J}$.
Accordingly, the morphisms of the network are exactly $\phi_K^H$ with $K\supset H \supseteq J$.

\end{proof}

Although the stacks $Y_I$ are not fibered products of stacks $Y_i$ with $i\in I$, the above arguments show that they can be constructed intrinsically from a succession of fibered products, after removing the diagonal components. In particular, there exists a closed embedding
$$ Y_I \hookrightarrow ( \prod_{i\in I} )_{X} Y_i.$$
More precisely,
% sa specific aici ca suntem in cazul generic 1-1?
\begin{corollary}
Each $\iota_{Ih}:=(\phi_{Ih}^I, \phi_{Ih}^h): Y_{Ih}\hookrightarrow Y_I\times_XY$ is a closed embedding, and
 \bea Y_I\times_XY \cong (\bigsqcup_{h\not\in I, b\in A_h}\Im\iota_{Ih})\bigsqcup (\bigsqcup_{i\in I}\Im (\phi_I^I, \phi_I^i) ),\eea  with the  diagonal morphisms $(\phi_I^I, \phi_I^i) : Y_I \to Y_I\times_XY$ yielding the higher  dimensional components of $Y_I\times_XY$.
%In particular, the spaces $Y_I$ in the network are intrinsic to the morphism $g:Y \to X$.
\end{corollary}
\begin{proof}
This is also a direct consequence of Lemma \ref{fibered product}, due to the Cartesian diagrams:
 \bea  \diagram  {\bigsqcup_{j\not\in I}V_{Ij}^a}\rto \dto & {V_{I}^a\times_{U}V} \rto \dto  &  V_I^a
\dto\\  {Y_{Ih}}       \rto &   {Y_I\times_XY }  \rto & {Y_I},
\enddiagram \eea
where $V=\bigsqcup_{j, b}V_j^b$,   while $V_{Ij}^a\cong V_{I}^a\times_{U}V_j^b$ for fixed $b\in A_j$.
\end{proof}
% atentie la indicele $a$. Aici continuam sa ne purtam ca in cazul 1-1!
\begin{remark} \label{minimality}
 By the previous corollary applied successively to each index $I$, the objects of the network of local embeddings associated to the proper local embedding $g:Y \to X$ and the \'etale atlas $U$ are in fact independent of the choice of atlas satisfying the properties listed in Proposition \ref{U}. The choice of $U$ determines only the number of copies of each $Y_I$ contained in the associated network. For example, by replacing the \'etale atlas $U$ of $X$ with $U \bigsqcup U$ and keeping the choice of \'etale atlases for  $Y$ unchanged, we obtain a network which, apart from the morphisms $Y_i\to X$, is a disjoint union of two copies of the network for $U$.

\end{remark}

% (integral ca sa nu fie grade diferite pe componente diferite?) %(from the construction of $Y_J$)
\begin{notation}
Let $g: Y \to X$ be a proper local embedding of Noetherian stacks satisfying the assumptions from Proposition \ref{U}.
%Assume that $Y$ is reduced and geometrically unibranch, that the morphism $Y \to g(Y)$ is equidimensional and that its degree is equal to a fixed number $d$ at all generic points of $g(Y)$.
For any two generic points $\xi_J$ and  $\xi_K$ of $Y_J$ and $Y_K$ respectively, such that $K \supset J$ and such that $\xi_J$ specializes to  $\varphi^J_K(\xi_K)$, we denote by $\left[ \xi_K \to \xi_J \right]$ the degree of $\varphi^J_K$  at $\varphi^J_K(\xi_K)$.
% Similarly
\end{notation}

% For $K \supseteq J$ such that $|K\setminus J|=k-j$, the following relations thus hold  \bea   \left[ Y_K \to Y_J \right] &=& \sum_{K=K_k\supset K_{k-1} \supset ... \supset K_{j+1} \supset K_{j} =J} \prod_{i=j+1}^{k}  \left[ Y_{K_{i} }\to  Y_{K_{i-1}} \right] \\
%\bea &= &N(K,J)\prod_{i=j+1}^{k}  \left[ Y_{K_{i} } \to  Y_{K_{i-1}} \right] \eea

 For $K\in \cP_k$, $J\in \cP_j$  as above, the following relation follows directly from Lemma \ref{fibered product}:  \bea   \left[ \xi_K \to \xi_J \right] &=&N(K,J) \nu(\xi_K, \xi_J), \eea
where $N(K,J)$ is the number of all maximal chains $K=K_k\supset K_{k-1} \supset ... \supset K_{j+1} \supset K_{j} =J$ and
\bea \nu(\xi_K, \xi_J) =|\{ K'\in \cP_{k}\mbox{; } \exists \xi_{K'} \mbox{ generic point of } Y_{K'} \mbox{ such that } \varphi^J_{K'}(\xi_{K'})=\varphi^J_{K}(\xi_{K}) \}|.\eea

 Thus if $K\in \cP_k$, $J\in \cP_j$ and $I\in \cP_i$ satisfy $K \supset J \supset I$, then by a count of chains
\bean \label{deg} \frac{ [\xi_{K} \to  \xi_{J}][ \xi_{J} \to  \xi_{I} ] }{ [ \xi_{K} \to  \xi_{I}]}=\frac{ |\{ K_k\in \cP_k; K_k \supset I\}| }{|\{ K_k\in \cP_k; K_k \supset J\} | |\{ K_j\in \cP_j; K_j \supset I \} |,}
\eean
for any generic points $\xi_I$, $\xi_J$ and  $\xi_K$ of $Y_I$, $Y_J$ and $Y_K$ respectively, such that $\xi_J$ specializes to  $\varphi^J_K(\xi_K)$ and $\xi_I$ to  $\varphi^I_J(\xi_J).$

\end{notation}

\begin{theorem}\label{X'}

Consider a network of proper local embeddings $\phi_J^I: Y_J \to Y_I$ for
$I\subseteq J$, $I \in \cP_i$ and $J\in \cP_j$, associated to a
proper local embedding $Y\to X$ by Definition \ref{network} under the assumptions of Proposition \ref{U}, where by
convention $Y_{\emptyset }=X$. For each such morphism $\phi_J^{I}$,
there exists a closed embedding of stacks $\phi'^I_J: Y'_J
\hookrightarrow Y'_I$, together with \'etale surjective morphisms $p_J: Y'_J
\to Y_J$ and $p_I: Y'_I \to Y_I$
 making the diagram $$   \diagram     Y'_J  \rto^{ \phi'^I_J } \dto_{ p_J } & Y'_I \dto^{ p_I } \\ Y_J \rto^{ \phi^I_J } & Y_I \enddiagram  $$
commutative, and such that $$ Y_J\times_{Y_I}Y'_I= \bigsqcup_{J'\in \cP_j, J'\supseteq I} Y'_{J'}$$
for $J\in \cP_j$.
\end{theorem}

\begin{proof}

We construct $Y'_I$ in decreasing order of $I$, with $Y'_I=Y_I$ for $I$ maximal.
 For  $I\in \cP_{n-1}$, the stack $N_I:= \bigsqcup_{J\supset I} Y_J$ comes with a natural map $n_I: N_I \to Y_I$ \'etale on its image, and the \'etale atlas $$\bigsqcup_{J\supset I} W_J\cong n_I(N_I)\times_{Y_I}W_I.$$
  $Y'_I$ is constructed as in Proposition \ref{etale}, such that it admits
 \begin{enumerate}
 \item a surjective \'etale morphism $p_I:Y'_I\to Y_I$,
 \item an embedding $N_I \hookrightarrow Y'_I$ whose composition with $p_I$ is $n_I$,
 \item an \'etale atlas $W_I$ satisfying  $Y'_J\times_{Y'_I}W_I\cong W_J$, and
 \item a groupoid presentation $R'_I \rightrightarrows W_I$ given by
  \end{enumerate}
  $$R'_I \cong W_I\times_X W_I \setminus \bigcup_{i\not=j; i \in I} S_{ij}^{ab} \setminus \bigcup_{k\not= l; k,l \not\in I} S_{lk}^{ac}.$$
 Given any $k<n-1$, $I \in \cP_k$ and assuming the stacks $Y'_J$ with the above properties constructed for all $J\supset I$, the network stack $N_I$ is constructed
out of all stacks $Y'_{ I\cup\{h\}}$, by gluing each pair $Y'_{I\cup\{h\} }$ and $Y'_{I\cup\{h'\} }$  along $Y'_{I\cup\{h, h'\} }$ as in \cite{abramovich}, Proposition A.1.1 and Corollary A.1.2. Accordingly, $N_I$ admits a natural groupoid presentation $\left[ \bigcup_{J\supset I}R'_J \rightrightarrows \bigcup_{J\supset I}W_J\right]$, where the unions are considered inside $U\times_XU$ and $U$, respectively.
% trebuie sa pun []?
From here, the composition maps $Y'_{I\cup\{h\} } \to Y_{I\cup \{h\} }\to Y_{I} $ glue together
to a morphism $n_I: N_I \to Y_{I} $ which is \'etale on its image, as noted from the groupoid presentations of $N_I$ and $Y_I$. Moreover, for every $J\supseteq I$ there exists a canonically defined closed embedding $Y'_J \hookrightarrow N_I$.  If $n_I$ were proper, the construction in Proposition \ref{etale}  applied to $n_I$ would yield the stack $Y'_I:= (Y_I)_{N_I}$ with the desired properties (1)--(4). However, $n_I$ is not necessarily proper, so we will obtain the same construction indirectly. We consider a canonical stratification
\bea  N^n_I \hookrightarrow N^{n-1}_I \hookrightarrow ... \hookrightarrow N^{k+1}_I=N_I \eea
 and a sequence of lifts $(Y_I)_{N_I^l}$, for $0\leq k \leq l \leq n-1$, where  $N^{n}_I=\bigsqcup_{J\supset I, J\in \cP_{n}}Y_J$ has a proper map into $Y_I$, and $N^{k+1}_I=N_I$ has a proper map into $(Y_I)_{N_I^{k+2}}$. This is discussed in the next lemma.

 Although the process described in the lemma is indirect,  the groupoid presentation of the resulting space can be constructed directly as in Proposition \ref{etale}, as indicated by  property (2) in the lemma. Moreover, the \'etale atlas $\bigcup_{J\supset I}W_J$ of $N_I$ is embedded in the \'etale atlas  $W_I$ of $Y_I$, and so the relations $R'_I$ defining  $(Y_I)_{N_I}$ are obtained, via Proposition \ref{etale}, by restricting the preimage  of  $\bigcup_{J\supset I}W_J$ in $R_I$ such that  $R'_{I|_{\bigcup_{J\supset I}W_J}}$  coincides with the image of $(\bigcup_{J\supset I}R'_J)$ in $R_I$.  This directly yields the presentation given at point (4) above.

At the final step, the stack $X'$ admits a groupoid presentation $\left[ R' \rightrightarrows U\right]$, with
 \bean \label{relations_for_X'} R' = U\times_XU \setminus \bigcup_{i\not=j}S_{ij}^{ab}.\eean

 \end{proof}

\begin{lemma}
With the notations above, let $k$ and $l$ be any integers such that $0\leq k \leq l \leq n-1$, and let any $I\in \cP_k$.  Define
\bea N_I^l:= \Im ( \bigsqcup_{J\supseteq I, J\in \cP_l} Y'_J  \to N_I ). \eea Then there exists a sequence of \'etale, surjective morphisms
\bea  Y'_I=(Y_I)_{N_I^{k+1}} \to (Y_I)_{N_I^{k+2}}\to ...\to (Y_I)_{N_I^{n-1}} \to (Y_I)_{N_I^{n}} \to Y_I, \eea
and morphisms $n^l_I: N_I=N^{k+1}_I \to (Y_I)_{N^l_I}$  for each $l$, \'etale on their images, such that the following properties hold:
\begin{enumerate}
\item  The restriction of $n^{l+1}_I$ to $N_I^{l}$ is a proper morphism. In particular, for  $l=k+1$,
the morphism $n_I: N_I \to (Y_I)_{N_I^{k+2}}$ is proper.
\item  With the notations from Proposition \ref{etale}, $((Y_I)_{N_I^{l+1}})_{N_I^{l}}\cong (Y_I)_{N_I^{l}}.$
\item For each $J\supset I$, $J\in \cP_s$ and $l> s$, there exist canonical Cartesian diagrams
\bea  \diagram  N^l_J \rto \dto & {N^l_I}\dto \\
  (Y_J)_{N^{l+1}_J} \rto &  {(Y_I)_{N^{l+1}_I}}. \enddiagram   \eea
\end{enumerate}

\end{lemma}

\begin{proof} The construction of $(Y_I)_{N^l_I}$ and the proofs of the properties (1)--(3) are completed by decreasing induction on $l$. When $l=n$, consider $N^{n+1}_I=\emptyset$ and $(Y_I)_{N^{n+1}_I}=Y_I$.  The map
\bea n^n_{I|_{N^{n}_I}}: N^{n}_I=\bigsqcup_{J\supset I, J\in \cP_{n}}Y_J \to Y_I \eea
is proper and \'etale on its image, so  $(Y_I)_{N_I^{n}}$ is constructed like in Proposition \ref{etale}. As $n_I: N_I \to Y_I$ is \'etale on
 its image and ${N^n_I}$ has a close embedding in $N_I$ satisfying the conditions of Corollary \ref{transitive} a), the map $n_I$ lifts to $n^n_I: N_I \to (Y_I)_{N_I^{n}}$.
Furthermore, for all $J\in \cP_{l}$, $l\geq k$, the Cartesian diagram
\bea  \diagram  {N^n_J} \rto \dto & {N^n_I}\dto \\
                  Y_J \rto &  Y_I\enddiagram   \eea
induces a proper morphism $(Y_J)_{N^n_J} \to (Y_I)_{N_I^{n}}$ as in Corollary \ref{X_Y functorial}. Gluing all $Y'_J$-s with $J \supset I$, $J\in \cP_{n-1}$ gives
$N^{n-1}_I$ with a proper morphism to $(Y_I)_{N_I^{n}}$, the restriction of $n^n_I$. This leads to the next step of induction, with the construction of
$(Y_I)_{N_I^{n-1}}:=((Y_I)_{N_I^{n}})_{N_I^{n-1}}.$
     Assume now the constructions of $(Y_J)_{N^{l'}_J}$, $n^{l'}_J$ and properties (1)--(3) known for all $l'> l>s-1$, $J\supseteq I$, $J\in \cP_s$.
     Then by property (3) and Corollaries \ref{X_Y functorial}, \ref{transitive}, there is a proper morphism
     \bea (Y_K)_{N^{l}_K}:= ((Y_K)_{N^{l+1}_K})_{N^{l}_K} \to  {((Y_J)_{N^{l+1}_J})_{N^{l}_J}}=:{(Y_J)_{N^{l}_J}}  \eea
  for any $K\supset J \supseteq I$.   Gluing the stacks $Y'_K=(Y_K)_{N^{l}_K}$ for all such $K\in \cP_{l-1}$ gives $N^{l-1}_J$ with a proper morphism to $(Y_J)_{N_J^{l}}$. This is the restriction of the morphism $n^l_J: N_J \to (Y_J)_{N^l_J}$ \'etale on its image, which was obtained  from $n^{l+1}_J$ by Corollary \ref{transitive}. Finally, these proper morphisms fit together in Cartesian diagrams like in (3) with  $l$ replaced by $l-1$, as the diagrams
  \bea  \diagram  {\bigsqcup_{K\supset J, K\in \cP_{l-1}}(Y_K)_{N^s_K}} \rto \dto & {\bigsqcup_{K\supset I, K\in \cP_{l-1}}(Y_K)_{N^s_K}}\dto \\
  (Y_J)_{N^{s}_J} \rto &  {(Y_I)_{N^{s}_I}}. \enddiagram   \eea
  are Cartesian for all $s\geq l$, due to Corollary \ref{X_Y functorial} and decreasing induction on $s$.
\end{proof}

\begin{definition}
For a proper local embedding $g: Y\to X$ and an \'etale atlas $U$, the morphism $p=p_{\emptyset} :X'\to X$
 introduced in Theorem \ref{X'} (in the particular case when $I=\emptyset$) will be called the \'etale lift of $g$ with respect to the \'etale atlas  $U$.
\end{definition}

\begin{remark} \label{self-similarity for etale lifts}
We note that with this terminology, the \'etale morphisms $p_I: Y'_I \to Y_I$ introduced in Theorem \ref{X'} are the \'etale lifts of $\phi_J^I$ with respect to  the \'etale atlas $W_I$, for any $J=I\cup\{ j\}$ and $j\not\in I$. Indeed, this is a direct consequence of Corollary \ref{self-similarity}, as all the stacks $N^s_J$ and $Y'_J$ for $J\supseteq I$ constructed in the course of the proof of Theorem \ref{X'} depend only on the network $\{ \phi_K^H: Y_K \to Y_H\}_{K\supset H \supseteq J}$.

\end{remark}

\begin{example} Consider a projective curve $X$ whose singular locus consists of  a simple node $x$, and let $g:Y\to X$ be its normalization, with $g^{-1}(x)=\{y_1, y_2\}$.  Then $X'$ is the union of two copies $Y^1$ and $Y^2$ of $Y$, glued together along $g^{-1}(x)$ such that $y_1^1=y_2^2$ and $y_2^1=y_1^2$.
 \end{example}

\begin{example}\label{example}
Let  $U^0:=\Spec k[x_1,x_2,x_3]$ and consider the action of $\ZZ_3\cong A_3$ on $U_0$ which permutes the coordinates,
\bea \sigma([x_1,x_2,x_3]):=[x_{\sigma(1)}, x_{\sigma(2)}, x_{\sigma(3)}].\eea
Let $V^0:=\Spec k[x_1,x_2] \bigsqcup \Spec k[x_2,x_3] \bigsqcup \Spec k[x_1,x_3]$, with the natural local embedding $g^0:V^0\to (x_1x_2x_3=0) \hookrightarrow U^0$, and the natural action of $\ZZ_3$ on $V^0$ which is compatible with $g^0$. Taking quotients yields a local embedding
 \bea g: \AA^2  \to [ \AA^3/\ZZ_3 ]. \eea
We denote $Y:=\AA^2$,  $X:= [ \AA^3/\ZZ_3 ]$. Then  $U:=\bigsqcup_{\alpha\in S_3}U^0$ is an \'etale atlas of $X$ satisfying all the properties listed in Proposition \ref{U}, where $Y\times_XU=V_1\bigsqcup V_2 \bigsqcup V_3$ and $V_i\cong W_i=\bigsqcup_{\alpha\in S_3}(x_{\alpha(i)}=0)\hookrightarrow \bigsqcup_{\alpha\in S_3}\AA^3$. The associated network of local embeddings will thus be of the form
\bea  \diagram
 & Y_{12} \rto \drto & Y_{1}  \drto &\\
 Y_{123} \urto \rto \drto & Y_{31} \urto \drto & Y_2 \rto & X, \\
 & Y_{23} \urto\rto & Y_3 \urto &
\enddiagram \eea
where $Y_1\cong Y_2 \cong Y_3 \cong Y=\AA^2$ with the map $g:Y \to X$. The spaces $Y_{12}\cong Y_{31} \cong Y_{23} = \AA^1\bigsqcup \AA^1$, and  $\phi_{12}^1: Y_{12} \to Y_1=\Spec k[x_1, x_2]$ maps each copy of $\AA^1$ into $(x_1=0)$ and $(x_2=0)$, respectively, while $\phi_{12}^2=\phi_{12}^1\circ\tau$, where $\tau$ is the transposition switching the two copies of $\AA^1$, and similarly for the other morphisms $\phi_{ij}^{i}$. Finally, $Y_{123}$ is a copy of two points, each mapped to zero on one of the lines of $Y_{ij}$, respectively. Thus, with the notations from Theorem \ref{X'},
$Y'_i=(Y_i)_{N_i}$ is obtained by gluing two copies of $\AA^2$ outside the union of two lines, $N$ is the union of $Y'_i$-s for $i\in \{1,2,3\}$, glued along the $Y_{ij}$-s as indicated by the arrows in the network.  The colors in Figure 1 show which pairs of lines are identified in $N$.  $X'=X_N$ is isomorphic to $X$ outside $N$, while
\bea N= g(Y)\times_XX' \mbox{ and } Y\times_XX'=\bigsqcup Y'_i, \eea
as shown in the figure.

\end{example}

\begin{figure}
  \includegraphics[width=5.5 in]{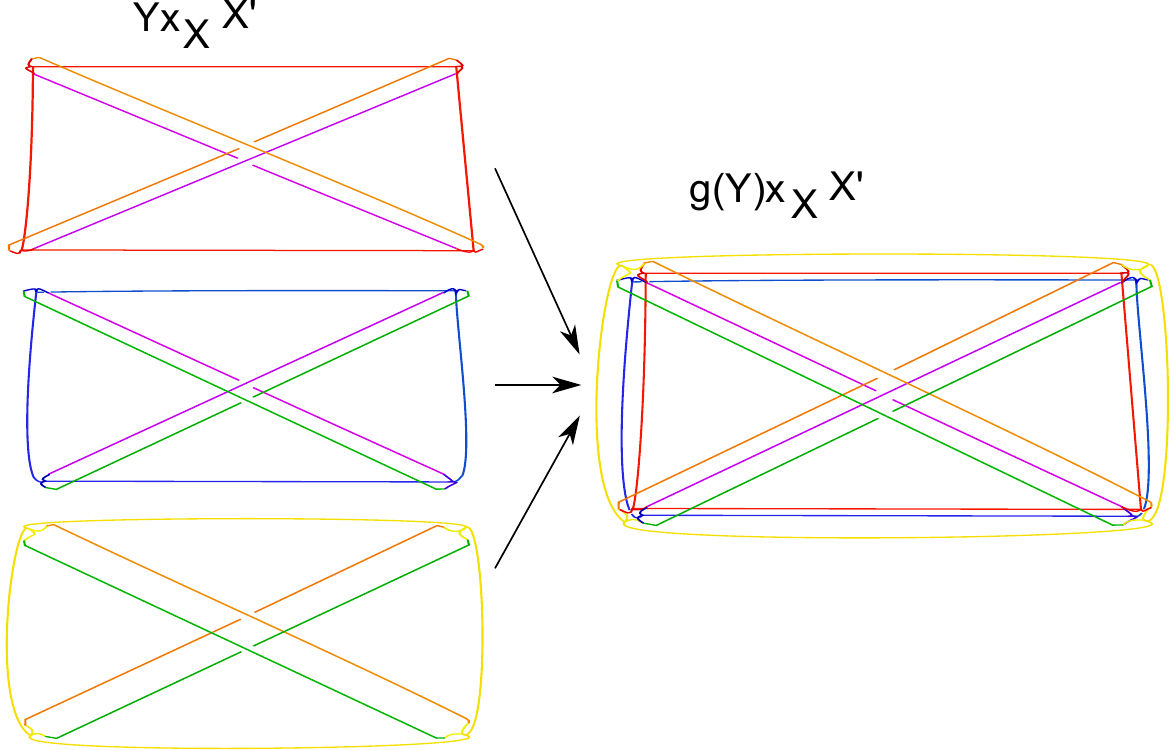}\\
  \caption{Example \ref{example}.}
  %\label{}
\end{figure}

\subsection{Chow rings and universally closed push-forwards}

Let $g: Y \to X$ be a proper, local embedding of Noetherian stacks. Assume that $Y$ is reduced and geometrically unibranch, that the morphism on the image $Y \to g(Y)$ is equidimensional, and that its degree is equal to a fixed number $d$ at all generic points.
The stacks $Y'_I$ constructed in the previous sections are in general non-separated. However, they do satisfy the existence part in the valuative criterion of properness.
\begin{definition} \label{preproper}
A morphism of stacks $f: F\to G$ will be called universally closed
if it is of finite type and, for any complete valuation ring $R$
with field of fractions $K$ and any commutative diagram
\bea \begin{CD} \Spec (K) @>{u} >> F\\
                 @VVV  @VV{f} V \\
                 \Spec(R) @>>> G,  \end{CD} \eea
there exist finite extensions $K'$ of $K$ such that, for the integral closure $R'$ of $R$ in $K'$, the composition morphism $\Spec(K')\to F$
extends to $\Spec(R')$.
 Like in \cite{vistoli}, we are mostly interested in stacks having
 coarse moduli schemes. Then by Proposition 2.6. and Definition 2.1
 in \cite{vistoli}, the image of such a morphism of stacks is defined
 and the definition above is a natural extension to stacks of the notion of
 universally closed morphism of schemes.
\end{definition}
\begin{definition} \label{probabilistic weight} Consider a universally closed morphism of stacks $f: F\to G$. A probabilistic weight $w$ of $f$ is a map defined on the set of all integral substacks of $F$, with values in the interval $[0,1]$, such that the weight of the generic point of $F$ is 1 and, for any commutative diagram like in Definition \ref{preproper},
$$ w(D) =\sum_i w(P_i),$$
where $D$ is the image in $F$ of the unique point in $\Spec (K)$, and $P_i$ are the images of the closed point in $\Spec(R')$ for all map extensions $\Spec(R')\to F$ as above.
\end{definition}

\begin{definition} Given a universally closed morphism of integral stacks $f: F\to G$ with probabilistic weight $w$, let
  \bea f_*[V]= w(V) \deg(V/W) [W]  \eea
  for any closed integral substack $V$ of $F$, where $W=f(V)$ and $\deg(V/W)$ is as in Definition 1.15, \cite{vistoli}. A homomorphism $f_*: Z_k(F)\to Z_k(G)$ is then defined by linear extension.
  \end{definition}
\begin{proposition}
The homomorphism $f_*: Z_k(F)\to Z_k(G)$ induces a well defined
universally closed push-forward homomorphism $f_*: A_k(F)\to
A_k(G)$.
\end{proposition}
\begin{proof}
Propositions 3.7 in \cite{vistoli} and Proposition 1.4,
\cite{fulton} deal with the proper push-forward in the cases of
stacks and schemes, respectively. The difference for universally
closed maps lies in the proof of the later, in the case when $\dim
F=\dim G$. The case when $f$ is finite is identical to Case 2 in
Proposition 1.4, \cite{fulton}. Following \cite{fulton} closely, we
take normalizations of the source and target, and the problem is
thus reduced to the case of a universally closed morphism of normal
varieties. Let $W$ be a codimension one subvariety of $G$, let $A$
be the local ring of $W$ on $G$, and $B$ the integral closure of the
$A$ in the field $k(F)$ of rational functions on $F$, such that $B$
is a discrete valuation ring. Then by Definition \ref{preproper},
for each maximal ideal $m_i$ of $B$ there is a finite number of
codimension one subvarieties $V_i^l$ of $F$ such that $B$ dominates,
and is therefore equal to, the local ring of each $V_i^l$ in $X$.
Then for any $r\in k(F)^*$,
$$ \sum_{i} \mbox{ ord}_{V_i^l}(r) \deg(V_i^l/W)= \mbox{ ord}_W(N(r)),$$
where for each $i$ a choice of the index $l$ has been fixed,
and $N(r)$ is the determinant of the $k(G)$-linear endomorphism of $k(F)$ given by multiplication by $r$.
Finally, $$f_*[ \mbox{ div }(r)]=\sum_{V_i^l}\mbox{ ord}_{V_i^l}(r) f_*[V_i^l]= \sum_{V_i^l}\mbox{ ord}_{V_i^l}(r) w(V_i^l)\mbox{ deg}(V_i^l/W)[W].$$
For $i$ fixed, $\mbox{ ord}_{V_i^l}(r)\mbox{ deg}(V_i^l/W)$ is constant and $\sum_lw(V_i^l)= w(F)=1$ from Definition \ref{probabilistic weight}. Thus
$$ f_*[ \mbox{ div }(r)]=\sum_{W}\mbox{ ord}_W(N(r)) [W]$$
 like in the case of proper morphisms. This finishes the part of the proof specific to the universally closeness of $f$.
\end{proof}
  Universally closed push-forwards enjoy the usual properties of their proper relatives: for example, they commute with flat pullbacks, and the usual projection formula holds for $f$ universally closed and flat.

\begin{theorem} \label{universally closed, probabilistic weight}
% \label{simplified}
 Let $g: Y \to X$ be a proper, local embedding of Noetherian stacks. Assume that $Y$ is reduced and geometrically unibranch, that the morphism on the image $Y \to g(Y)$ is equidimensional, and that its degree is equal to a fixed number $d$ at all points of $g(Y\setminus Y_1)$.
 There exists a
Deligne--Mumford stack $X'$ with a surjective \'etale morphism to $X$,
such that the fiber product
$Y'=g(Y)\times_XX'$ is a finite union of stacks $Y'_i$ mapping \'etale onto
$Y$, and such that the maps $Y'_i \to Y$ and $p: X'\to X$
 are universally closed. Moreover, $p$ admits a
probabilistic weight $w$.

% i.e. unramified, representable; to compare definitions of Kresch and Vistoli
\end{theorem}

% presupunem $Y_I$ integre, sau spunem undeva ca am impartit in componente.

\begin{proof}

Consider a network of local embeddings for  $g: Y \to X$ as in Definition \ref{network}. Let $X'$ be the corresponding lift of $X$ constructed in Theorem \ref{X'}.
Let $\zeta_I$ denote a generic point of  $Y'_I$, with $\zeta_{\emptyset }$ the generic point of $X'$ which specializes to the image of $\zeta_I$. With the notations from equation (\ref{deg}), define
$$w(\zeta_I):= \frac{1}{|\cP_k| [ \xi_I \to \xi_{\emptyset} ]}$$
for any $I \in \cP_k$. Here for any $K$, we let $\xi_K$ denote the image of $\zeta_K$ in $Y_K$. Then equation (\ref{deg}) becomes \bean \label{weight-degree} w(\zeta_I)= \sum_{} [\xi_J \to \xi_I ]w(\zeta_J), \eean
where the sum is taken after all $J \in \cP_j, J \supset I$ such that $\xi_I$ specializes to $\varphi^I_J(\xi_J)$.
As $\{ Y'_I \setminus ( \bigcup_{J\supset I} Y'_J )\}_I$ forms a locally closed stratification of $X'$,  the generic point of any integral substack $D$ of $X'$ will be found in exactly one of the above strata. We extend $w$ to a function on all points of $X'$ by identifying $w(D)$ with the weight of the generic point of its associated stratum which specializes to it.

The universally closeness property will result during the proof that
$w$ is a probabilistic weight of $p:X'\to X$. Consider a complete
discrete valuation ring $R$ with field of fractions $K$, a
commutative diagram
\bea \begin{CD} \Spec (K) @>{u} >> X'\\
                 @VVV  @VV{p} V \\
                 \Spec(R) @>{v}>> X,  \end{CD} \eea
such that the image $q_0$ in $X'$ of the generic point of $\Spec(R)$  lies in a stratum $Y'_I \setminus (\bigcup_{K\supset I} Y'_K)$, and the image $q_1$ in $X$ of the closed point of $\Spec(R)$  lies in $\Im (Y_J\setminus \bigcup_{K\supset J} Y_K \to X)$ for some $J\supseteq I$. Since $(Y'_I \setminus (\bigcup_{K\supset I} Y'_K))\cong (Y_I \setminus (\bigcup_{K\supset I} Y_K))$ and the map $Y_I \to X$ is proper, there is a unique extension  $v':\Spec(R')\to Y_I$ of a composition $\Spec(K') \to \Spec(K) \to Y'_I\to Y_I$, with the notations of Definition \ref{preproper}. Let $q\in \Im(Y_J \to Y_I)$ be the lift of $q_1$ through this extension. Then through each point in the preimage of $q$ in $Y'I\subset X'$ there is a unique lift  $\Spec R' \to Y'_I \hookrightarrow X'$ of the map $v'$. The generic point of each such lift has to be $q_0$, because the map $p_i:Y'_I \to Y_I$ restricts to the above mentioned isomorphism $(Y'_I \setminus (\bigcup_{K\supset I} Y'_K))\cong (Y_I \setminus (\bigcup_{K\supset I} Y_K))$. Let $y_i$ be the images of the closed points of these lifts. They all have the same weight $w(\xi_J)$. Then by equation (\ref{weight-degree}),
$$w(q_0)=\sum_i w(y_i),$$
which proves that $w$ is a probabilistic weight for the universally
closed morphism $p$.
\end{proof}

\begin{corollary} \label{pullback injective}
   For each $i\in \{0, ...., n\}$ and $I \in \cP_i$, there is a universally closed push-forward map $p_{I *}: A_k(Y'_I) \to A_k(Y_I)$, such that for any connected component $Z$ of $Y_I$, the restriction of the map
   $$p_{I *}\circ p_{I}^* :  A_k(Y_I) \to A_k(Y_I)$$
 to $A_k(Z)$ is $d\cdot \mbox{ id}_{A_k(Z)}$,
   where $d$ is the degree of the morphism $p_I^{-1}(Z) \to Z$.
   In particular, the flat pullback $p_{I}^*$ is injective.
\end{corollary}

We note that all maps $p_I$ are universally closed due to Proposition \ref{universally closed, probabilistic weight} in conjunction with Remark \ref{self-similarity for etale lifts}.

By convention, $X'=Y'_{\emptyset}$, and thus the Corollary shows how $A(X)$ can be regarded as a subgroup of $A(X')$, and how classes in $A(X)$ can be recovered from $A(X')$ via push-forward.
%the representability problem?

In addition to the assumptions of Theorem \ref{universally closed, probabilistic weight}, for the remainder of this section we will assume $X$ to be smooth, and the morphisms $\phi_I^J: Y_I \to Y_J$ to be local regular embeddings. An extended Chow ring of the network $\{ \phi_I^J: Y_I \to Y_J \}_{I,J}$, was introduced in Definition 3.6 of \cite{noi1}. We recall this definition with a slight variation that does without the action of a symmetry group on $\cP=\bigcup\cP_k$.

\begin{notation} Fix $I$ and $I\cup \{h\} \in \cP$.
 For any cycle $\alpha=[V] \in Z_l(\Im (Y_{I\cup \{h\} }\to Y_{I} ))$, let $\alpha_h \in  Z_l(Y_{I\cup \{h\} })$ be defined as follows:
 $$\alpha_h=\frac{\sum_i [V^i_h]}{\deg\left( (\phi^{I}_{I\cup\{ h\} })^{-1}(V)/V\right) }$$
where $\{V^i_h\}_i$ are the $l$--dimensional components of $(\phi^{I}_{I\cup\{ h\} })^{-1}(V)$.
\end{notation}
\begin{definition}
 The vector spaces $A_l(\{ Y_I\}_{I\in \cP};\QQ)$  are defined by
\bea    A_l(\{ Y_I\}_{I\in \cP};\QQ) :=\oplus_{I}
Z^{l-\codim_{Z}Y_{I}}(Y_{I}) / \sim , \eea  the  sum taken after all
$I\in \cP$ with $ \codim_{X}Y_I \leq l$. The equivalence relation
$\sim $ is generated by rational equivalence together with relations
of the type: $$ \alpha \sim \sum_{ h' }\alpha_{h'}, $$ for
any cycle $\alpha=\phi^I_{I\cup \{ h \}*}\alpha_h \in Z_l(  \Im (Y_{I\cup \{h\} }\to Y_{I})   )$.
\end{definition} The field  of coefficients $\mathbb{Q}$ will be
omitted in the notations throughout the rest of the text.
\begin{definition}
Multiplication is defined as
$$ \alpha\cdot_r\beta := {\phi}^{I*}_{I\cup J}(\alpha) \cdot {\phi}^{J*}_{I\cup
J}(\beta) \cdot \frac{c_{top}({\phi}^{I*}_{I\cup J}\cN_{Y_{I}|X})c_{top}({\phi}^{J*}_{I\cup J}\cN_{Y_{J}|X})}{c_{top}(\cN_{Y_{I\cup J|X}})}$$ in $A(Y_{I\cup J})$, for any two
classes $\alpha \in A(Y_I)$ and $\beta \in A(Y_J)$. Here ${\phi}^{I*}_{I\cup J},
{\phi}^{J*}_{I\cup J}$ are the (generalized) Gysin homomorphisms, as defined in \cite{vistoli}, while $c_{top}(\cN_{Y_{K}|X})$ denotes the highest Chern class of the normal bundle $\cN_{Y_{K}|X}$.
\end{definition}

\begin{theorem} \label{compare to network}
The following rings are isomorphic
$$A(X') \cong A(\{ Y_I\}_{I\in \cP}).$$
\end{theorem}

\begin{proof}

Compositions of the flat pullbacks $p_I^*$  with the push-forward of
embeddings $\phi'^{\emptyset }_{I *}$ add up to a morphism
$$\oplus_{I\in \cP} A( Y_I) \to A(X').  $$ which moreover factors through
     $$    F: A(\{ Y_I\}_{I\in \cP})  \to A(X').                    $$
The compatibility of the group morphism $F$ with the product operations is a direct consequence of the excess intersection formula for the embeddings of $Y'_I$ and $Y'_J$ into $Y'_{I\cap J}$, with intersection $Y'_{I\cup J}$.

Construct an inverse for $F$ as follows. For each $J\in \cP$, let $U_J$ denote the complement in $Y'_J$  of all the images of $Y'_I$, with $J\subset I$. (To define this complement  one can work with supports of the corresponding coarse moduli schemes, but there is a canonical stack structure on $U_I$). Note that  $U_J= Y_J \setminus (\bigcup_{I\supset J} \Im (Y_I \to Y_J) )$ as well. Working with the commutative diagram of open/closed exact sequences
$$ \begin{CD} \oplus_{k \not\in J} A( Y_{ J\cup\{ k\} }) @>>>  {A( Y_{ J })}  @>>>  A(U_J) @>>> 0  \\
@V { p^*_{ J\cup\{ k\} }  } VV    @V{ p^*_{ J }  }VV  @V{ = }VV \\
\oplus_{k \not\in J} A( Y'_{ J\cup\{ k\} }) @>>> A( Y'_{ J }) @>>> A(U_J) @>>> 0.
\end{CD} $$
one finds, for each $\alpha'_J \in A( Y'_{ J })$, classes $\alpha_J \in A( Y_{ J })$ and $\alpha'_{J\cup\{ k\}}
\in A( Y'_{ J\cup\{ k\} })$ such that $$\alpha'_J=p^*_{ J }\alpha_J +\sum_{k \not\in J} \phi'^{J\cup\{ k\} }_{J *}\alpha'_{J\cup\{ k\}}.$$
Reiterating this argument one finds a collection of classes $[\alpha_I ]\in\oplus_I A( Y_{ I }) $ for all $I$ containing $J$, such that
$$   \alpha'_J=\sum_{ I \supseteq J} p^*_{ J } \phi^I_{J *}\alpha_I.  $$
The choice of $\alpha_I $ is unique up to the equivalence relation $\sim$. Indeed, if $\sum_{ I \supseteq J} p^*_{ J } \phi^I_{J *}\beta_I=0$ then $\sum_{ I \supseteq J} \phi^I_{J *}\beta_I=0$ by Corollary \ref{pullback injective}, or, equivalently, $\beta_J =-\sum_{ I \supset J} \phi^I_{J *}\beta_I$. After applying the equivalence relation to each $\phi^I_{J *}\beta_I$, $\beta_J$ can then be replaced by a sum of classes from $A(Y_I)$ with $I \supset J$, and induction on the index set can then be applied by Lemma 3.13 in \cite{noi1}.
  When $J=\emptyset$, the collection  $[\alpha_I ]\in\oplus_I A( Y_{ I }) $  defines the desired inverse of $F$.
% aman asta: Indeed, if $\sum_{J\subseteq I} p^*_{ J } \phi^I_{J *}\alpha_I=0$ in $A( Y'_{ J })$, then... e nevoie de ea?

\end{proof}

%Example: $\overline{M}_{0,m}(\mathbb{P}^n, d)$

%For every $J \subset \cP$, there is an \'etale cover of $p_J: \ol{M}'_J \to \ol{M}_J$, and an embedding $\ol{N}_J \hookrightarrow \ol{M}'_J$, such that for any $I \supset J$, the diagram
% $$   \diagram     \ol{M}'_I  \rto^{ \phi'^J_I } \dto_{ p_I } & \ol{M}'_J \dto_{ p_J } \\ \ol{M}'_J \rto^{ \phi'^J_I } & \ol{M}_J \enddiagram  $$
%is commutative.

\section{The Chern classes of a weighted projective blow-up}

%work over a field of characteristic 0 instead of $\CC$?
In this section we extend the notion of a blow-up along a closed embedding to the case of a proper local embedding $g:Y \to X$ of Noetherian stacks.
We assume that $Y$ is reduced and geometrically unibranch, that the morphism on the image $Y \to g(Y)$ is equidimensional, and that its degree is equal to a fixed number $d$ at all generic points.

 We note that for practical purposes, it is often enough to work with just $\Bl_{Y'_i}X'$,
for universally closed \'etale lifts $p: X'\to X$, and a corresponding \'etale morphism $Y_i'\to Y$ constructed as in Section 1. For example, this is the case when one is interested in intersection theory on a smooth Deligne-Mumford stack. However, in other contexts, a stack $\tilde{X}$ with a proper morphism $f: \tilde{X} \to X$ is required. A natural construction of $\tilde{X}$ follows. For special morphisms $g:Y \to X$, the stack $\tilde{X}$ was defined in \cite{noi1}.

\begin{definition} \label{blow-up local emb}
 Let $g:Y \hookrightarrow X$ be a proper local embedding of smooth stacks, and consider a groupoid presentation  $ \left[ R\rightrightarrows U \right]$ of
$X$  such that  $Y \times_X U = \bigsqcup_{i,a} V_i^a$ has all the
properties listed in Proposition \ref{U}. The blow-up $\tilde{X}=\Bl_YX$ of
$X$ along $Y$ is defined as the Deligne-Mumford stack of \'etale groupoid presentation
$\tilde{X}=[\tilde{R} \rightrightarrows \tilde{U} ],$ where $\tilde{U}$ is the fibered product over $U$ of the blow-ups $\Bl_{W_i}U$  for all $i$, and $\tilde{R}$ is the fibered product over $R$ of the blow-ups $\Bl_{S_{ij}^{ab}}R$  for all $i,j, b$,
 where $a$ is fixed.

 The natural morphisms $\tilde{s}, \tilde{t}: \tilde{R} \rightrightarrows \tilde{U}$, as well as $\tilde{e}: \tilde{U} \to  \tilde{R}$, $\tilde{i}: \tilde{R}\to \tilde{R}$ and $\tilde{m}: \tilde{R}\times_{\tilde{U}}\tilde{R}\to \tilde{R}$ making up the groupoid structure are induced from the groupoid morphisms $s,t,e,i,m$ of $ \left[ R\rightrightarrows U \right]$ by the universal property of the blow-up. Indeed, this is due to the following Lemma.
  \end{definition}
\begin{lemma} With the notations from Proposition \ref{U}, the following hold:
\begin{enumerate}
\item $V_i^a\times_UR \cong \bigsqcup_{j,b} S^{ab}_{ij}$ via each of the \'etale morphisms $s$, $t$.
\item $\Bl_{\bigsqcup_{j,b}S_{ij}^{ab}}R \cong (\Bl_{W_i}U) \times_U R$ via each of the \'etale morphisms $s$, $t$.
\item $e^{-1}\cI_{V_i^a\times_UR} = \cI_{W_i}$, while $i^{-1}\cI_{S_{ij}^{ab}}= \cI_{S_{ji}^{ba}}$
 and $m^{-1}\cI_{S_{ij}^{ab}}= \cI_{\bigsqcup_{k,c}S_{ik}^{ac}\times_US_{kj}^{cb}}$.
\item $\tilde{R}\times_{\tilde{U}}\tilde{R} \cong (\prod_{i,a,j,b,k,c})_{R\times_UR}\Bl_{S_{ik}^{ac}\times_US_{kj}^{cb}}(R\times_UR)$.
\end{enumerate}
\end{lemma}

\begin{proof} We will identify $V_i^a$ with its image $W_i$ in $U$. Statements (1) and (3) follow from definitions of $V_i^a$ and $S_{ij}^{ab}$. Statement (2) is a consequence of (1), together with the observation that
$\cI_{\bigsqcup_{j,b} S^{ab}_{ij}}^n \cong (s^{-1}\cI_{W_i})^n \cong s^{*}\cI_{W_i}^n$ and
$\cI_{\bigsqcup_{j,b} S^{ab}_{ij}}^n \cong (t^{-1}\cI_{W_i})^n \cong t^{*}\cI_{W_i}^n$
due to the fact that $s$, $t$ are \'etale.
Statement (4) follows in the same way, due to (3) and the fact that $m$ is \'etale.
\end{proof}

In particular, statement (2) leads to
\begin{corollary}
Let $g:Y \hookrightarrow X$ be a proper local embedding of smooth stacks, and consider a groupoid presentation  $ \left[ R\rightrightarrows U \right]$ of
$X$  such that  $Y \times_X U = \bigsqcup_{i,a} V_i^a$ has all the
properties listed in Proposition \ref{U}.
Then the blow-up morphism  $f: \tilde{X} \to X$ is proper, and there exists a Cartesian diagram
\bea \diagram  \tilde{X}'
\rto^{f'}\dto_{\tilde{p}} & X' \dto^{p} \\ \tilde{X} \rto_{f} &X,
\enddiagram \eea
where $p: X'\to X$ is the
universally closed \'etale cover constructed in Section 1, and with the notations from Section 1,
\bea  \tilde{X}':=(\prod_i)_{X'} \Bl_{Y'_i}X'.\eea
\end{corollary}

\begin{proof} With the notations from Definition \ref{blow-up local emb}, we have
$\tilde{U}\cong \tilde{X}\times_XU$, due to the previous Lemma, (2). Thus $f$ is proper since $\tilde{U}\to U$ is. On the other hand, $U$ is also an \'etale atlas of $X'$, and so $\tilde{U}$ is also an \'etale atlas of $\tilde{X}'$.
It remains to study the morphisms induced at the level of relations. Due to formula (\ref{relations_for_X'}), we have $\bigsqcup_{j,b} S^{ab}_{ij}\bigcap R' = \bigsqcup_cS^{ac}_{ii}\bigcap R'$, which is the relation space for $Y'_i$ (Definition \ref{network}). This is enough to deduce that the diagram of stacks is Cartesian.

\end{proof}

The blow-up $\tilde{X}$ is independent of the choice of \'etale atlas $U$.  This follows by standard functorial arguments (a detailed proof to be part of \cite{noi5}).

 For the remainder of this article, whenever we will talk about the Chern class of $\tilde{X}$, we will assume that $X$ and $Y$ are smooth stacks, and that the images $W_i$ of $V_i^a$ in $U$ for all $i$ intersect each other, as well as their intersections, transversely, so that $\tilde{X}$ is a smooth stack.
 In this case, as \bea
\begin{array}{cccc} p^*\cT_{X}\cong \cT_{X'}, &  \tilde{p}^*\cT_{\tilde{X}}\cong \cT_{\tilde{X}'}, & p_i^*\cN_{Y|X}\cong \cN_{Y'_i|X'}, & q_i^*\cN_{{\tilde{Y}}|\tilde{X}}\cong \cN_{{\tilde{Y}}'_i|\tilde{X'}}  \end{array}   \eea
for $p_i: Y'_i\to Y$, and the map between exceptional divisors $q_i:
{\tilde{Y}}'_i\to {\tilde{Y}}$. Thus calculating the Chern invariants of $\tilde{X}$ reduces
to the calculation for $\tilde{X}'$ due to Corollary \ref{pullback
injective}. We can thus reduce the problem for a local
embedding of smooth stacks to that for a succession of smooth embeddings. If,
moreover, the ideals $\cI_i$ of $Y'_i$ have compatible filtrations
with weights such that weighted blow-ups can be defined, then the
above reasoning applies to weighted blow-ups as well.

\begin{remark}
If $X$ and $Y$ are smooth, but the images $W_i$ of $V_i^a$ in $U$  do not intersect each other or their intersections transversely for all $i$, then in general  $\tilde{X}$ and $\tilde{X}'$  are not smooth stacks. However, even in this case they come equipped with a Chow class suitable enough for intersection theory. Indeed, in this case $ \tilde{X}':=(\prod_i)_{X'} \Bl_{Y'_i}X'$ is still a fibered product of smooth stacks over a smooth stack, and the Cartesian diagram
\bea  \diagram  X' \rto^{\Delta} &  \prod_i X'    \\
\tilde{X}' \uto \rto & \prod_i \Bl_{Y'_i}X' \uto
 \enddiagram  \eea
produces a Chow class $\Delta^{!}([\prod_i \Bl_{Y'_i}X'])$ on $\tilde{X}'$, invariant under deformations of the map $g$ to $X$.
%This approach is developed by the authors in \cite{noi5}.
\end{remark}

\subsection{Weighted projective blow-up and locally trivial weighted projective
fibration.}\label{general blowup}

In this section we will work with stacks over $\CC$.
  Let $Y$ be a smooth substack of a smooth
stack $X$. Consider an increasing filtration $\{  \cI_n\}_{n\geq 0}$ of the ideal
 $\cI_Y$ of $Y$ in $X$, such that  $\cI_0=\cO_X$,  $\cI_1=\cI_Y$ and $\cI_n\cI_m\subseteq\cI_{m+n}$ for all $m,n\geq 0$.
 \begin{lemma}\label{filtration} Assume that $\{  \cI_n\}_{n\geq 0}$
 satisfies the following properties (\cite{noi1}, section 3)
\begin{enumerate}
\item  $\cI_k\cap \cI_Y^2=\sum_{j=1}^{k-1} \cI_j\cI_{k-j}$,
\item $ \cI_k/(\cI_k\cap\cI_Y^2)$ is a subbundle of the conormal bundle
$\cI_Y/\cI_Y^2$.
\end{enumerate}
Then $\Proj (\oplus_{n\geq 0} \cI_n)$ has only quotient singularities. This implies the existence of a natural desingularization $\tilde{X}$ of $\Proj (\oplus_{n\geq 0} \cI_n)$, (constructed as in \cite{vistoli}, Proposition 2.8.) such that locally in the \'etale topology the morphism
$f: \tilde{X} \to \Proj (\oplus_{n\geq 0} \cI_n)$ is of the form
\bea [ W/H ] \to W/H,\eea
where $W$ is a scheme and $H$ is a finite group acting on it.
\end{lemma}
\begin{proof}
Consider an \'etale atlas of $X$ made of affine schemes $U=\Spec R$, such that $V:=Y\times_XU$ is complete intersection in $U$, and such that there exists a set of generators $\{ x_{ni}\}_{n,i}$  of  $\cI(U)$ with $ x_{ni}\in \cI_n(U)\setminus \cI_{n+1}(U)$, and with the property that the images of $\{ x_{ni}\}_{n\leq k,i}$ in $ \cI_k(U)/(\cI_k(U)\cap\cI_Y^2(U))$ form a basis for
$ \cI_k(U)/(\cI_k(U)\cap\cI_Y^2(U))$.
% putem lucra asa pe intregul $U$ sau trebuie sa lucram pe fibre?
Let \bea R_1:=\frac{R[\{ y_{ni}\}_{n,i}]}{< \{ y_{ni}^n-x_{ni}\}_{n,i}>},\eea and let $Y_1\cong V$ be the zero locus of $I_1:=\{ y_{ni}\}_{n,i}$ in $U_1:=\Spec R_1$. Then $U_1$ is smooth and $Y_1$ is smooth too, because $Y$ is.
%use regular sequence?
The finite group $G\cong \oplus_{n,i} \ZZ_n$ has a natural action on $R_1$.
Due to condition (1),
\bean \label{powers of ideal}(I_1^n)^G = \sum_{\sum_kka_k=n} \prod_{k} \cI_k^{a_k}(U) = \cI_n(U). \eean
Here $a_k$ are non-negative integers.

 If $\bigsqcup_R\Spec R$ is an \'etale atlas for $X$, then $\bigsqcup_R (\Bl_{Y_1}U_1/G)$ is an \'etale atlas for the stack $\Proj (\oplus_{n\geq 0} \cI_n)$, where $\Bl_{Y_1}U_1$ represents the blow-up of $U_1$ along $Y_1$, with the natural action of $G$ induced from the action on  $R_1$.

As $\bigsqcup_R (\Bl_{Y_1}U_1/G)$  has only quotient singularities, we can construct a natural desingularization of  $\Proj (\oplus_{n\geq 0} \cI_n)$ by following \cite{vistoli}, Proposition 2.8: We can choose $T = \bigsqcup_{R,x}\tilde{U}_x/M$, for finitely many choices of $x\in \Bl_{Y_1}U_1$ and suitable neighborhoods $\tilde{U}_x$ in $\Bl_{Y_1}U_1$,
where $M$ is the largest small subgroup of the stabilizer for $x$, (which insures that $T$ is smooth).  Then the normalization $R_T$ of $T\times_{\Proj (\oplus_{n\geq 0} \cI_n)}T$, with the morphisms induced by the two projections on $T$, define an \'etale groupoid structure on $T$. Indeed, the multiplicative structure comes naturally via the isomorphism of $R_T\times_{T} R_T$ with the normalization of \bea (T\times_{\Proj (\oplus_{n\geq 0} \cI_n)}T)\times_T (T\times_{\Proj (\oplus_{n\geq 0} \cI_n)}T)\cong T\times_{\Proj (\oplus_{n\geq 0} \cI_n)}T\times_{\Proj (\oplus_{n\geq 0} \cI_n)}T.\eea (by the Purity of the Branch Locus).

By equation (\ref{powers of ideal}), the quotient $W/(G/M)$ of the smooth scheme $W:=\tilde{U}_x/M$ is an open in $\Proj (\oplus_{n\geq 0} \cI_{n|U})$. Moreover,
\bea W \times (G/M) \cong \ol{W\times_{\tilde{U}_x/G}W}, \eea
where $\ol{W\times_{\tilde{U}_x/G}W}$ denotes the normalization of $W\times_{\tilde{U}_x/G}W$ (by the Purity of the Branch Locus). The construction of $\tilde{X}$ now implies that the following diagram is Cartesian:
\bea  \diagram   [W/(G/M)] \rto \dto & \tilde{X} \dto \\
W/(G/M) \rto &  \Proj (\oplus_{n\geq 0} \cI_n),  \enddiagram  \eea
and the horizontal arrows are \'etale.

\end{proof}

\begin{definition} Let $Y$ be a smooth substack of a smooth
stack $X$, satisfying conditions (1) and (2) in Lemma \ref{filtration}. Then $\tilde{X}$ constructed above, with the natural morphism
 $\pi : \tilde{X}\to X$,  will be called the weighted blow--up of $X$
 along $Y$, with the filtration $\{  \cI_n\}_{n\geq 0}$.

  \end{definition}

 Let $X^{\#}:=\Proj (\oplus_{n\geq 0} \cI_n)$, with the morphism $\pi^{\#}:X^{\#}\to X$. By Lemma 3.1 in \cite{noi1},  the reduced structure of $X^{\#}\times_XY$ is
$Y^{\#}:=\Proj(\oplus_{n\geq 0} \cI_n/\cI_{n+1}), $,
and  $\cI_n = \pi^{\#}_*\cI_{Y^{\#}}^n$.

Recall the morphism $f:\tilde{X}\to X^{\#}$. The closure in $\tilde{X}$ of $f^{-1}(Y^{\#}\setminus \mbox{ Sing }(X^{\#}))$ will be called the exceptional divisor of $\pi$, and denoted by $\tilde{Y}$. With the notations from the proof of Lemma \ref{filtration},  $\tilde{Y}$ is given locally in the \'etale topology
by $[(\tilde{Y}_{1 x}/M)/(G/M)]$, where $\tilde{Y}_{1 x} \subset \tilde{U_x}$ is the restriction of the exceptional divisor $\tilde{Y}_1$ in $\Bl_{Y_1}\Spec R_1$ to the open set $\tilde{U}_x$.
% $\tilde{Y}:=\tilde{X}\times_X Y$  will be called the exceptional divisor of $\pi$.

 Lemma  \ref{filtration} also implies $f_*\cO_{\tilde{X}}\cong \cO_{X^{\#}}$ and $f_*\cI_{\tilde{Y}}^n\cong \cI_{Y^{\#}}^n$, as $\tilde{X}$ and $X^{\#}$ coincide outside a codimension two locus. Thus $\pi_*\cI_{\tilde{Y}}^n\cong \cI_n$ for any $n$.

\begin{definition} \label{equivar affine fibration}
Consider the weighted blow-up $\tilde{X} \to X$ constructed in Lemma \ref{filtration}. Define
 \bea A:=\Spec(\oplus_{n\geq 0}\cI_n/\cI_{n+1}). \eea
 $A\to Y$ is a $\CC^*$-equivariant affine fibration, with a trivialization (locally in the Zariski topology) on which $\CC^*$ acts linearly on the fibres. Thus the fixed point locus of $A$ with respect to $\CC^*$ is a section of $A\to Y$. We will call it the zero section of the fibration, and identify it with $Y$.

 Denote $ \cL:=\cO_{\tilde{Y}}\otimes_{\cO_{\tilde{X}}}\cO_{\tilde{X}}(-\tilde{Y})=\cI_{\tilde{Y}}/\cI_{\tilde{Y}}^2,$ the conormal bundle of the regularly embedded $\tilde{Y}$ in $\tilde{X}$.
\end{definition}

\begin{lemma} \label{simplest blow-up} Let $\pi : \tilde{X}\to X$ be a weighted blow--up of $X$
 along $Y\hookrightarrow X$, with the filtration $\{  \cI_n\}_{n\geq 0}$, and let $\tilde{Y}$ denote its exceptional divisor.
With the notations set up above, $\pi_*\cL^n \cong \cI_n/\cI_{n+1}$, and \bea \tilde{A}=\Spec(\oplus_n\cL^n)\to A = \Spec (\oplus_n \pi_* \cL^n)\eea
is the weighted blow-up of $A$ along its zero section $Y$, for the filtration $\{\cJ_n\}_n$  of the ideal $\cJ=\oplus_{n\geq 0}\cI_n/\cI_{n+1}$ given by $\cJ_n:= \oplus_{k\geq n}\cI_k/\cI_{k+1}$.
Then $\tilde{Y}=\Proj (\oplus_n\cL^n)$ is also the exceptional divisor of $\tilde{A}$.
\end{lemma}
\begin{proof}

We noted that $f_*\cI_{\tilde{Y}}^n\cong \cI_{Y^{\#}}^n$ and thus also  $f_*(\cI_{\tilde{Y}}^n/\cI_{\tilde{Y}}^{n+1}) \cong \cI_{Y^{\#}}^n/\cI_{Y^{\#}}^{n+1}$ as $f$ has finite dimensional fibres. This implies $\pi_*\cL^n \cong \cI_n/\cI_{n+1}$ due to \cite{noi1}, proof of Lemma 3.1.

Consider an \'etale atlas of $X$ made of affine schemes $U=\Spec R$, like in the proof of Lemma \ref{filtration}. We construct the weighted blow-up $\tilde{A}\to A$ by the method outlined in the above mentioned proof.  We will keep the notations found there throughout this proof as well. Let  $V=\Spec S:=Y\times_XU$. Then an \'etale atlas of $A$ is made of affine schemes $\Spec S[\{ x_{ni}\}_{n,i}]$. Consider covers $T_1:=\Spec S[\{ y_{ni}\}_{n,i}]$ with $y_{ni}^n=x_{ni}$, let $Z_1\cong V$ be the zero locus of $J_1:=\{ y_{ni}\}_{n,i}$ in $T_1$, and consider the blow-up $\tilde{T}_1=\Bl_{Z_1}T_1$,  with exceptional divisor $\tilde{Z}_1$.
The group $G\cong \oplus_{n,i} \ZZ_n$ acts on $T_1$ and its blow-up.

With these data, $Z_1\cong Y_1\cong V$. Moreover, there is an isomorphism of normal bundles $\cN_{Z_1|T_1} \cong \cN_{Y_1|U_1}$, compatible with the action of $G$ on them.  Thus
$\tilde{Z}_1\cong \tilde{Y}_1$, as well as $\cN_{\tilde{Z}_1|\tilde{T}_1} \cong \cN_{\tilde{Y}_1|\tilde{U}_1}$, and the largest small subgroups of the stabilizer for corresponding points in the exceptional divisors coincide as well (while for points not in the exceptional divisor, the stabilizer itself is a small group). An \'etale atlas of  $\tilde{A}$ is $\bigsqcup_x\tilde{T}_{1x}/M$,   for finitely many choices of $x\in \tilde{Z}_1$, with the corresponding small subgroup $M$ and suitable open neighborhoods $\tilde{T}_{1x}\subset \tilde{T}_1$. Then
\bea  & \tilde{T}_1\cong \Spec (\oplus_{n\geq 0} \cI^n_{\tilde{Z}_1}/\cI^{n+1}_{\tilde{Z}_1}) \cong  \Spec (\oplus_{n\geq 0} \cI^n_{\tilde{Y}_1}/\cI^{n+1}_{\tilde{Y}_1}) \Rightarrow &  \\
 & \tilde{T}_1/M\cong \Spec (\oplus_{n\geq 0} \cI^n_{\tilde{Z}_1}/\cI^{n+1}_{\tilde{Z}_1})^M  \cong \Spec (\oplus_{n\geq 0} (\cI^n_{\tilde{Y}_1}/\cI^{n+1}_{\tilde{Y}_1})^M),
& \eea
thus $\tilde{T}_{1 x}/M \cong \Spec(\oplus_{n\geq 0} \cI^n_{\tilde{Y}_{1 x}/M}/\cI^{n+1}_{\tilde{Y}_{1 x}/M})$. This proves $\tilde{A}=\Spec(\oplus_n\cL^n)$.
%partea cu invariantii de demonstrat in lema initiala?

\end{proof}

 Conversely, given a morphism $p: P\to Y$, with a sheaf $\cL$ on $P$ such that $P\cong \Proj(\oplus_n \pi_* \cL^n)$, then $P$ can be understood as the
exceptional divisor of the following weighted blow--up:
$$\Spec(\oplus_n\cL^n)\to \Spec (\oplus_n \pi_* \cL^n).$$

% Let   $\tilde{X}$ denote the desingularization of ${X}^{\#}$ constructed like in Lemma \ref{filtration}???? sunt indeplinite conditiile?

 \begin{lemma}
 Let $A\to Y$ be a $\CC^*$-equivariant affine fibration  with a trivialization (locally in the Zariski topology) on which $\CC^*$ acts linearly on the fibres, like in Definition \ref{equivar affine fibration}. Consider the natural filtration of the ideal $\cI$ of the zero section $Y$ in $A$  induced by the weights of the  $\CC^*$-- action, and let $\tilde{A} \to A$ be the corresponding weighted blow-up, with exceptional divisor $\tilde{Y}$. Then
 \bea \tilde{Y} \cong [ (A\setminus Y) /\CC^* ].  \eea
 \end{lemma}

 \begin{proof}
 Here we will employ the same notations as in the proof of Lemma \ref{simplest blow-up}. For $x\in \tilde{Z}_1$, with the corresponding small subgroup $M$ and suitable open neighborhood $\tilde{T}_{1x}\subset \tilde{T}_1$, let $\tilde{Z}_{1x}:=\tilde{T}_{1x}\bigcap \tilde{Z}_1$, and let $T_{1x}$ be the preimage of $\tilde{Z}_{1x}$ in the $\CC^*$-- bundle
 $ (T_1\setminus Z_1) \to \tilde{Z}_1.$
Consider the commutative diagram of GIT quotients
 \bea   \diagram  T_{1x} \rto \dto^{/\CC^*} & T_{1x}/M  \dto^{/\CC^*} \\
                   \tilde{Z}_{1x}  \rto & \tilde{Z}_{1x}/M. \enddiagram \eea
  Let $O_x$ denote the $\CC^*$-orbit parametrized by $x$.
 We claim that $M$ keeps $O_x$ pointwise fixed, so $\CC^*$ acts freely on $O_x/M \cong O_x$, and thus for suitable choice of $T_{1x}$ we have \bean  \label{stacky} \tilde{Z}_{1x}/M \cong [(T_{1x}/M) /\CC^*].\eean Indeed, any element in $\sigma_i \in M$ is a reflection, meaning that its fixed point locus is a divisor $\tilde{D}$ in $\tilde{T}_{1x}$. Moreover, from the definition of the $G$--actions, the divisor  $\tilde{D}$  is a strict transform of a divisor $D=(y_{ni}=0)$ fixed by $\sigma_i$. But $ x \in \tilde{D} \Leftrightarrow O_x \subseteq D$, and thus $M=\mbox{ Stab}_{q}$ for any  $q\in O_x$.

  From (\ref{stacky}) it follows that
  \bea  [(\tilde{Z}_{1x}/M)/(G/M)] \cong [[(T_{1x}/M) /\CC^*]/(G/M)]\cong  [[(T_{1x}/M) /(G/M)]/\CC^*] \cong [(T_{1x}/G)/\CC^*], \eea
  where $T_{1x}/G$ is an open in $A\setminus Y$.

 \end{proof}

\begin{notation}
  For any $\CC^*$-equivariant affine fibration   $A\to Y$  with a trivialization on which $\CC^*$ acts linearly on the fibres, we denote
\bea  \cP^w(A):= [ (A\setminus Y) /\CC^* ].\eea
 We will say that $\cP^w(A)$ is a weighted projective fibration.
\end{notation}

We return now to the weighted blow--up $\pi : \tilde{X}\to X$ of $X$
 along $Y$, with filtration $\{  \cI_n\}_{n\geq 0}$.
 We saw that the relation $\pi_*\cI_{\tilde{Y}}^n\cong \cI_n$ holds for any $n$.
The proof of Lemma \ref{filtration} also implies isomorphisms between Chow groups:
\bea    A(\tilde{X}) \cong A(X^{\#}) \mbox{ and }  A(\tilde{Y}) \cong A(Y^{\#}),   \eea as $\tilde{X}$ and $X^{\#}$ have the same coarse moduli space.
With this, we have the following  description of the Chow ring of ${\tilde{Y}}$ (Lemma 3.2 in \cite{noi1}).

\begin{lemma} \label{Chow ring of projective fibration}
a) The normal bundle in $A=\Spec(\oplus_{n\geq 0}\cI_n/\cI_{n+1})$
of the fixed locus $Y$ under the natural $\CC^*$ action on $A$ is
$$\cN_{Y|A}= \oplus_{n\geq 1}\cN_n/\cN_{n+1},$$
where $\{\cN_n\}_n$ is the filtration of the normal bundle
$\cN_{Y|X}$ dual to the filtration $ \{\cI_n/(\cI_n\cap\cI_Y^2)\}_n
$ of $\cI_Y/\cI_Y^2$.

b) There is a ring isomorphism $$A({\tilde{Y}}; \QQ)\cong \frac{A(Y; \QQ)[\tau ]}{ < P_{Y|X}(\tau )>},$$ where $P_{Y|X}(t)$ is the top
equivariant Chern class of the bundle $ \cN_{Y|A}$. In particular,
the free term of $P_{Y|X}(t)$ is the top Chern class of $\cN_{Y|X}$. Here
$\tau$ is the first Chern class of $\cO_{\tilde{Y}}(1):= \cN^{\surd}_{\tilde{Y}|\tilde{X}}$.

\end{lemma}

  Lemma \ref{Chow ring of projective fibration} sets up the context
  for calculating the Chern classes of the locally trivial weighted projective fibration
$p: {\tilde{Y}}\to Y$ by deforming ${\tilde{Y}}$ to a weighted projective bundle on $Y$
and applying the Euler's sequence from Appendix. Here by a weighted projective bundle on $Y$ we mean a stacky quotient
$[(N\setminus Y)/\CC^*]$ where $N\to Y$ is a vector bundle with a linear $\CC^*$--action.

Indeed, with the
notations above, consider the standard deformation \bea D=\Bl_{Y\times\{
\infty\} }(A\times \PP^1) \setminus \tilde{A}\eea  of $A$ to the total
space $N_{Y|A}$ of the normal bundle $\cN_{Y|A}$, where $\tilde{A}=
\Bl_YA$ is one of the components of the fibre over $\infty$ of
$\Bl_{Y\times\{ \infty\} }(A\times \PP^1)\to \PP^1$. Thus the fibre
over $\infty$ of $D \to \PP^1$ is $N_{Y|A}$, and the action of
$\CC^*$ over $A\times \PP^1$, with fixed locus $Y\times \PP^1$,
induces a natural $\CC^*$-- action on $D$. Moreover, due to Lemma
\ref{Chow ring of projective fibration}, a), the quotient of
$N_{Y|A}\setminus Y$ by $\CC^*$ is a weighted projective bundle, where we have identified the zero section in $N_{Y|A}$
with $Y$. Push-forward by the composition \bea  [(N_{Y|A}\setminus Y)/\CC^* ] \hookrightarrow
[(D\setminus Z)/\CC^*] \to {\tilde{Y}} \times \PP^1 \to {\tilde{Y}} \eea (where $Z$ is the fixed locus in $D$) induces an isomorphism between
the Chow rings of ${\tilde{Y}}$ and of the weighted projective bundle
 $[(N_{Y|A}\setminus Y)/\CC^* ]$. We obtain the following
\begin{proposition}\label{proj fibr}
Let $p: \tilde{Y}\to Y$ be a weighted projective fibration as above, and
$Q_n:=\cN_{w_n}/\cN_{w_{n+1}}$, for all indices $w_n$ such that
$\cN_{w_n}\not= \cN_{w_{n+1}}$, on which $\CC^*$ acts with weight
$w_n$. Then the total Chern class
$$c(\tilde{Y})= p^*c(Y)\prod_n c(Q_n\otimes \cL^{\otimes w_n}),$$
with $\cL:=\cO_{\tilde{Y}}(1)$ and $c(Q_n\otimes \cL^{\otimes
w_n}):=\prod_i(1+a_i+w_nc_1(\cL))$ where $a_i$ are the Chern roots
of $Q_n$ .
\end{proposition}
\begin{proof}
Let ${ i_{\infty} }: [(N_{Y|A}\setminus Y)/\CC^* ] \hookrightarrow [(D\setminus Z)/\CC^*]$ and
$i_0: \tilde{Y} \to [(D\setminus Z)/\CC^*]$ be embeddings of fibres in the flat family $
[(D\setminus Z)/\CC^*] \to \PP^1$ and let $q: [(D\setminus Z)/\CC^*] \to \tilde{Y}$ be the natural
projection obtained after taking quotients of $D \to A$, such that
$q\circ i_0= \mbox{ id}_{\tilde{Y}}$. Then by the projection formula and the
rational equivalence of fibres, \bea  i_{0 *}c(\cT_{\tilde{Y}})=i_{0
*}i_0^*c(\cT_{[(D\setminus Z)/\CC^*]})=  i_{\infty
*}i_{\infty }^*c(\cT_{[(D\setminus Z)/\CC^*]}) = i_{\infty
*}c([(N_{Y|A}\setminus Y)/\CC^* ]), \eea
and thus after composing with $q_*$,
  $$c(\cT_{\tilde{Y}})=q_*i_{\infty *}c([(N_{Y|A}\setminus Y)/\CC^* ]),$$
  which by Appendix and Lemma \ref{Chow ring of projective fibration}
  is of the form described in this Proposition.
\end{proof}

\subsection{Model for a weighted blow-up}

We start our Chern class calculations with the most approachable
type of weighted blow-ups: when the blow-up locus is the fixed locus
of a $\CC^*$-action on the entire space. In this case, equivariant
cohomology techniques permit the recovery of Chern classes of the
blow-up from their pull-backs to the exceptional divisor, which in
turn are easily computable.

Let $Y$ be a stack, $A$ a $\CC^*$--equivariant affine fibration on $Y$ like in Definition \ref{equivar affine fibration},
such that the $\CC^*$--action on $A$ induces a decomposition of the
normal bundle of the fixed locus $Y$ $$\cN_{Y|A}=\oplus_h \cQ_n$$
with weights $\{w_n\}_n$ and $\mbox{ rk }\cQ_n =k_n$. Let $O_Y$
denote the trivial line bundle on $Y$. Consider the torus
$T:=\CC^*\times\CC^*$ action on $A\times O_Y$ coming from the
individual action of the first $\CC^*$ on $A$ and the second on
$\cO_Y$. In this subsection we will denote by $X := \cP^w(A \oplus
\cO_Y)$, the locally trivial weighted projective fibration  obtained
as a quotient of $A\times O_Y\setminus Z$ by $\CC^*$ embedded diagonally in
$T$ (where $Z$ denotes the fixed locus), by $\tilde{Y}:=\cP^w(A)$, and by
$\tilde{X}:=\cP(\cO_{\tilde{Y}}(-1) \oplus \cO_{\tilde{Y}})$.

We obtain a  blow-up diagram
\[ \diagram
   \tilde{Y}      \rto^{    j} \dto_{g}        &  \tilde{X} \dto^{f}\\
   Y     \rto^{ i } & X,
\enddiagram \]
where $Y=\PP(\cO_Y) \hookrightarrow \cP^w(A\oplus
\cO_Y)$, and similarly $j: \tilde{Y} \hookrightarrow \tilde{X}$ is the
embedding $\tilde{Y}=\cP(\cO_{\tilde{Y}} ) \hookrightarrow
\cP(\cO_{\tilde{Y}}(-1) \oplus \cO_{\tilde{Y} })$, the exceptional
divisor in the weighted projective blow-up $f: \tilde{X}\to X$.

%de decis cand lucram cu inele Chow si cand cu cohomologie

\begin{proposition} \label{special case}
Keeping notations from above, assume that for each $n$, the total
Chern class $c(\cQ_n)=c_{k_n}(\cQ_n)+...+c_1(\cQ_n)+1$ can be
written as the pullback of a class
$p(\cQ_n)=p_{k_n}(\cQ_n)+...+p_1(\cQ_n)+1 \in A(X; \QQ)$. Then
 \bea c(\tilde{X})= f^*c(X)\frac{ (E
+1) \prod_{n=1}^l p(\cQ_n (-w_n E))}{\prod_{n=1}^lp(\cQ_n)}, \eea
where $E$ is the class of the exceptional divisor  and $p(\cQ_n
(w_ns)):=\prod_{i=1}^{k_n}(a_i+w_ns+1)$ where the pullbacks of $a_i$
on $Y$ are the Chern roots of $\cQ_n$.
\end{proposition}

\begin{proof}
The morphism $f$ is equivariant with respect to the natural $\CC^*:=
T/\CC^*$--actions on $\tilde{X}$ and $X$ with the weights specified
above, such that $\tilde{Y}$ and the section at infinity are the
fixed loci for the $\CC^*$--action on $\tilde{X}$, and $Y$ and the
section at infinity are the fixed loci for the $\CC^*$--action on
$X$. Thus the diagram above yields another weighted blow-up diagram
\[ \diagram
   \tilde{Y}_{\CC^*}     \rto^{    j_{\CC^*}} \dto_{g_{\CC^*}}        &  \tilde{X}_{\CC^*} \dto^{f_{\CC^*}}\\
   Y_{\CC^*}     \rto^{ i_{\CC^*} } & X_{\CC^*},
\enddiagram \]
where $Z_{\CC^*}:= Z\times_{\CC^*} E\CC^*$ for each stack $Z$, and
 $B\CC^*$ is the classifying space of $\CC^*$, with universal family $E\CC^*$. As noticed above, $Y_{\CC^*}\cong Y$ and $\tilde{Y}_{\CC^*}\cong
\tilde{Y}$.

 In the following, for easiness of notation,
 we will drop the subscript $\CC^*$ for maps, and only employ it to denote equivariant classes.
%$f_{\CC^*}$ ramane weighted blow-up?

In the equivariant Chow ring $A_{\CC^*}(\tilde{Y})$
\bea\begin{array}{cc}  e^{\CC^*}(j^*(\cT_{\tilde{X}}))=
e^{\CC^*}(\tilde{Y})e^{\CC^*}(\cN_{\tilde{Y}| \tilde{X}}), &
 e^{\CC^*}(i^*(\cT_{X}))= e^{\CC^*}(Y)e^{\CC^*}(\cN_{Y| X}), \end{array} \eea
  $$\Rightarrow \frac{j^*e^{\CC^*}(\tilde{X})}{j^*f^*e^{\CC^*}(X)}=
  \frac{e^{\CC^*}(\cN_{\tilde{Y}| \tilde{X}})}{g^*e^{\CC^*}(\cN_{Y| X})}\frac{e^{\CC^*}(\tilde{Y})}{g^*e^{\CC^*}(Y)}=
  \frac{ (-\xi +t+1) \prod_{n=1}^l c(\cQ_n (w_n \xi))}{\prod_{n=1}^lc(\cQ_n (w_n t))},$$
where $\xi$ is the first Chern class of $\cO_{\tilde{Y}}(1)$.

 Let $\alpha:= e^{\CC^*}(\tilde{X})/f^*e^{\CC^*}(X)-1$ and
$P(t) := \frac{ (-\xi +t+1) \prod_{n=1}^l p(\cQ_n (w_n
\xi))}{\prod_{n=1}^lp(\cQ_n (w_n t))} -1$.
  We note that $P(\xi )=0$.
%mai e nevoie de justificare aici
   Thus $\beta := P(t)/(t-\xi )$ is well defined and by the self
   intersection formula on the exceptional divisor,
\bean \label{$j^*$} j^*\alpha=j^*j_*\beta. \eean By the Atiyah-Bott
localization theorem, \bean \label{a-b} \alpha =j_*\frac{j^*
\alpha}{e^{\CC^*}(\cN_{\tilde{Y}|\tilde{X}})}+ l_*\frac{l^*
\alpha}{e^{\CC^*}(\cN_{E_{\infty}|\tilde{X}})}, \eean where $j:
\tilde{Y} \hookrightarrow \tilde{X}$ and $l: E_{\infty}
\hookrightarrow \tilde{X}$ are the two fixed point loci of
$\tilde{X}$ under the $\CC^*$--action. On the one hand, by relation
(\ref{$j^*$}),  \bea j_*\frac{j^*
\alpha}{e^{\CC^*}(\cN_{\tilde{Y}|\tilde{X}})}= j_*\frac{j^* j_*\beta
}{e^{\CC^*}(\cN_{\tilde{Y}|\tilde{X}})} = j_*\beta \eea and on the
other hand, as the section at infinity $E_{\infty}= \cP^w(N)$ and
$\tilde{Y}$ are disjoint, $l^*\cT_{\tilde{X}} \cong l^* f^* \cT_{X}$ and
thus the second term on the right hand side of equation (\ref{a-b}) is
zero. Thus in fact $ \alpha = j_*\beta $ which implies \bea
e^{\CC^*}(\tilde{X})=f^*e^{\CC^*}(X)\frac{ (E +t+1) \prod_{n=1}^l p(\cQ_n (-w_n
E))}{\prod_{n=1}^lp(\cQ_n (w_n t))}. \eea The classical limit $t\to
0$ yields the desired relation.

\end{proof}

\subsection{Deformation to the weigthed normal cone}

Returning to the general case, let $Y$ be a smooth
substack of a smooth stack $X$. Let $\pi : \tilde{X}\to X$ be the
weighted blow--up of $X$  along $Y$ for an increasing filtration $\{
\cI_n\}_{n\geq 0}$ of the ideal $\cI_Y$, satisfying properties (1)
and (2) in Lemma \ref{filtration}. The ideal sheaf $\cJ$ of $Y\times\{
\infty\}$ in $X\times\PP^1$ admits a filtration formed by the
sheaves $\cJ_n:= \sum_{k=0}^{n} \cI_k \cK^{n-k}$ where $\cK$ is the
ideal of $\infty$ in $\PP^1$ pulled back to $X\times \PP^1$. Let
$M$ be the weighted projective blow-up of
$X\times\PP^1$ along $Y\times\{ \infty\}$ with the filtration $\{
\cJ_n\}_n$.

The usual properties of the deformation to the normal cone carry out
for this construction with the suitable changes in weights:

(1) There is a natural closed regular embedding $J:  Y
\times\PP^1\hookrightarrow M$.

(2)  The composition $\rho= p_2 \circ \Pi$ of the blow-up map $\Pi: M
\to X\times\PP^1$ with the projection $p_2: X\times\PP^1 \to \PP^1$
is a flat morphism of stacks, and the following diagram commutes
\[ \diagram
      Y \times\PP^1  \rto^{    J} \dto_{pr_2}        &  M \dlto^{\rho}\\
   \PP^1.
\enddiagram  \]

(3) Over $\PP^1\setminus\{\infty\}$, the preimage
$\rho^{-1}(\AA^1)=X\times \AA^1$ and $J$ is the trivial embedding.

(4) As a Cartier divisor,
$$M_{\infty}:=\rho^{-1}(\infty)=P + \tilde{X},$$
% dar daca divizorul exceptional nu e redus?
where $P=\cP^w(\Spec (\oplus_{n\geq 0}\cJ_n/\cJ_{n+1} ))$ is a locally
trivial weighted projective fibration and $\tilde{X}$ is the weighted blow--up of $X$
 along the locus $Y$, with the filtration $\{\cI_n\}$.
Both $P$ and $\tilde{X}$ are Cartier divisors of $M$, intersecting
in $\tilde{Y}:=\cP^w(\Spec( (\oplus_{n\geq 0}\cI_n/\cI_{n+1} ))$, which is embedded
as the section at infinity in $P$ and as the exceptional divisor in
$\tilde{X}$. On the other hand, $Y_{\infty}=Y\times\{\infty\}$
embeds in $M_{\infty }$ as the zero section in $P$ and is thus
disjoint from $\tilde{X}\hookrightarrow M_{\infty }$.

The proof is analogous to \cite{fulton}, (5.1).

\begin{theorem}\label{blow-up}
 Consider  the weighted blow--up $f: \tilde{X}\to X$ of a smooth stack $X$
along a smooth substack $Y$ with an increasing
filtration $\{ \cI_n\}_{n\geq 0}$ of the ideal $\cI_Y$, satisfying
conditions (1)-(2) in \ref{general blowup}. Let $\{\cN_n\}_n$ denote
the filtration of the normal bundle $\cN_{Y|X}$ dual to the
filtration $ \{\cI_n/(\cI_n\cap\cI_Y^2)\}_n $ of $\cI_Y/\cI_Y^2$.
Let $\cQ_n=\cN_{w_n}/\cN_{w_{n+1}}$ for all indices $w_n$ such that $\cN_{w_n}\not=\cN_{w_{n+1}}$.

$\CC^*$ acts with weight $w_n$ on $\cQ_n$.

 Assume that for each $n$, the total Chern class
$c(\cQ_n)=c_{k_n}(\cQ_n)+...+c_1(\cQ_n)+1$ can be written as the
pullback of a class $p(\cQ_n)=p_{k_n}(\cQ_n)+...+p_1(\cQ_n)+1 \in
A(X; \QQ)$. Then
 \bea c(\tilde{X})= f^*c(X)\frac{ (E
+1) \prod_{n=1}^l p(\cQ_n (-w_n E))}{\prod_{n=1}^lp(\cQ_n)}, \eea
where $E$ is the class of the exceptional divisor in $\tilde{X}$ and
$p(\cQ_n (w_ns)):=\prod_{i=1}^{k_n}(a_i+w_ns+1)$ where the pullbacks
of $a_i$ on $Y$ are the Chern roots of $\cQ_n$.
\end{theorem}

\begin{proof}
Let $\tilde{M}$ be the weighted blow-up of $M$ along $Y\times
{\PP^1}$ with the filtration $r^*\cI_n$, were $r=p_1\circ \Pi$ is the
composition $M\to X\times\PP^1\to X$. Looking at the fibres over $0$
and $\infty \in \PP^1$,
\bea \diagram
      \tilde{X}  \rto^{ \tilde{j}_0 } \dto_{ f_0 } & \tilde{M} \dto^{ F } &
      \widetilde{P} + \tilde{X} \lto_{ \tilde{j}_{\infty } } \dto_{ f_{\infty } }\\
 {X}  \rto^{ {j}_0} \dto_{ } & {M} \dto^{ \rho } &
      {P} + \tilde{X} \lto_{ {j}_{\infty } } \dto \\
  0 \rto   &\PP^1 & \infty \lto.
\enddiagram  \eea

 The embeddings $\tilde{j}_{\infty }$, respectively ${j}_{\infty }$ split into $\tilde{k}:
  \widetilde{P}\to \tilde{M}$, $\tilde{l}: \tilde{X} \to \tilde{M}$, respectively
  $k: P\to M$,
 $l:\tilde{X} \to M$, where $F\circ \tilde{l}$ and $l$ can be naturally identified.
 Pullbacks of the quotient sheaf $\cG_M:= \cT_{\tilde{M}}/ F^*\cT_M$
yield $\tilde{j}_{0 }^*\cG_M = \cG_0:= \cT_{\tilde{X}}/f_0^*\cT_X$
and $\tilde{j}_{\infty }^*\cG_M = \cG_{\infty }$ which on $
\widetilde{P}$ is the quotient $ \cT_{\tilde{P}}/f^*\cT_{P}$, and on $\tilde{X}$ is the zero sheaf. We note that
the maps of locally free sheaves on stacks
$$\begin{array}{ccc} \cT_{\tilde{M}} \to F^*\cT_M,  & \cT_{\tilde{X}} \to f_0^*\cT_X & \mbox{
and } \cT_{\widetilde{P}} \to f_{\infty }^*\cT_{P}
\end{array}$$ are monomorphisms, as $\tilde{M}$
and $M$ are isomorphic outside the exceptional divisor and its
image, etc. Moreover, \bea \tilde{j}_{0
*}c(\cG_0)= \tilde{j}_{0*}\tilde{j}_{0}^*c(\cG_M)= c(\cG_M)\cdot \left[ \tilde{M}_0
\right] ,\eea which by rational equivalence is identified with \bea
c(\cG_M) \cdot \left[ \tilde{M}_{\infty} \right] = \tilde{j}_{\infty
*}c(\cG_{\infty})= k_* c(\cG_{\infty | \widetilde{P}})+
 \left[ \tilde{X} \right],\eea  as $\tilde{l}^* c(\cG_M) =
c(\tilde{l}^*\cG_M) =1$ on $\tilde{X}$. On the one hand, the formula
for $c(\cG_{\infty | \widetilde{P}})$ is given by Proposition \ref{proj fibr}.

%de completat cu notatiile potrivite

On the other hand, $\tilde{M}$ is isomorphic to the blow-up of
$\tilde{X}\times\PP^1$ along ${\tilde{Y}}\times\{ \infty \}$, where ${\tilde{Y}}$ is the
exceptional divisor of $\tilde{X}$. As such, there is a map
$\tilde{\Pi}: \tilde{M}\to \tilde{X}\times\PP^1$ whose composition $q$
with the projection $\tilde{X}\times\PP^1 \to \tilde{X}$, satisfies
$q\circ \tilde{j}_0= \mbox{ id }_{\tilde{X} }$. From this and the
model case in  Proposition \ref{special case} we recover the general
formula in the theorem.

% ce se intampla cu clasa lui X tilda?
\end{proof}

\section{The moduli space $\ol{M}_{0,m}'(\PP^n, d)$ }

\subsection{} Let $m, n, d$ be nonnegative integers. In this section we apply
the general theory for weighted blow-ups along local
embeddings of smooth stacks to  calculate the total Chern class of the moduli space
of stable maps $\ol{M}_{0,m}(\PP^n,d)$. We first set up the context
by briefly recalling the blow-up constructions of
$\ol{M}_{0,m}(\PP^n,d)$ from \cite{noi1} (for $m=1$), \cite{noi2}
(for $m>1$), and \cite{noi3} (for $m=0$). They pertain to a family
of smooth Deligne-Mumford stacks $\ol{M}_{0,\cA}(\PP^n,d,a)$, and of
weighted blow-ups
$$\ol{M}_{0,\cA}(\PP^n,d,a) \to \ol{M}_{0,\cA'}(\PP^n,d,a'),$$
where  $a, a'\in \QQ$,  $\cA=(a_1,...,a_m),\cA'=(a'_1,...,a'_m)\in
\QQ^m$ such that
  \bea \begin{array}{cccc}  \sum_{i=1}^{m}a_i+da>2, & \sum_{i=1}^{m}a'_i+da'>2, & 1\geq a\geq a'>0, & 1\geq a_j\geq
  a_j'\geq 0 \end{array} \eea
for all $j=1,...,m$. Here
 $\ol{M}_{0,\cA}(\PP^n,d,a)$ is the stack
of $ (\cA, a)$-- weighted stable maps as defined in \cite{noi2}, parametrizing $(\cA, a)$--stable maps.
\begin{definition} An  $(\cA, a)$--stable map consists of a family of rational curves $\pi \colon C \to S$,
whose fibers are either smooth or with nodes as singularities, with
$m$ marked sections not intersecting the nodes of the fibers, with
 a line bundle $\cL$ on $C$ of degree $d$ on each fiber $C_s$, and a morphism
$e:\cO_C^{n+1}\to \cL$ (specified up to isomorphisms of the target)
satisfying a series of stability conditions:
\begin{enumerate}
\item $\omega_{C|S}(\sum_{i=1}^ma_ip_i)\otimes \cL^a$ is relatively ample over $S$,
\item $\cG :=\Coker e$, restricted over each fiber $C_s$, is a skyscraper sheaf supported only
on smooth points of $C_s$, and
\item for any $s\in S$ and $ p \in C_s$ and for any $ I \subseteq \{ 1,...,m \} $ (possibly empty) such that
$ p= p_i$ for all $i \in I$ the following holds  $$\sum_{i \in I} a_i + a\dim\cG_{p}\leq 1.$$
\end{enumerate}
\end{definition}

Let $C$ be a curve. A tail of $C$ is a closed connected subcurve
$C'$ of $C$ with the property  that $C \setminus C'$ is connected.

We will now consider the case when $\cA=\{ 1\}$. The basic definitions for this case are introduced in \cite{noi1}. All the other cases of weights $\cA$ which will be of interest to us can be deduced from this case (\cite{noi2}, \cite{noi3}).

\begin{definition} Let $D=\{1, ..., d\}$. We say
$I\subset \cP(D) \setminus\{\emptyset, D\}$ is a nested set if, for
any two $h, h'\in I $, the intersection $h\cap h'$ is either $h$,
$h'$ or $\emptyset$.

For any number $l>0$, we denote $I_{\leq l}:=\{ h\in I; |h| \leq l\}$ and $I_{> l}:=\{ h\in I; |h| > l\}$.
\end{definition}

Fix a positive number $a<1$ and a nested set $I\subset \cP\setminus\{\emptyset, D\}$
such that $h\cap h'=\emptyset$ for any distinct  $h, h'\in  I_{\leq 1/a}$.
In \cite{noi1}, (Proposition 2.3), a boundary map $\ol{M}^a_I \to \ol{M}_{0,1}(\PP^n,d,a)$ was described.
\begin{definition}
 With our notations, $\ol{M}^a_I$ is  the
stack of $I$--type, $a$--stable, degree $d$ maps from a rational
curve into $\PP^n$, i.e.
 $$(C, p_1, \{p_h\}_{h\in I_{\leq 1/a}}, \{C_{h}\}_{h\in I_{> 1/a}}, \cL, e)$$ made of an
 $a$--stable, degree $d$ pointed map $(C, p_1, \cL, e)$, together with  marked points
 $\{p_h\}_{h\in I_{\leq 1/a}}$ and connected subcurves $ \{C_{h}\}_{h\in I_{> 1/a}}$
 satisfying the following properties:
\begin{enumerate}
\item $\forall h\in  I_{> 1/a}$,  $p_1 \not\in C_h\subset C$ and $\deg \cL_{|_{C_h}}=|h|$;
\item $\forall h\in  I_{\leq 1/a}$, $\dim \Coker e_{p_h}=|h|$;
\item compatibility of incidence relations:
              \begin{itemize}    \item $\forall h\in  I_{\leq 1/a}$, $\forall h'\in  I_{> 1/a}$,
              $h\subset h'$ iff $p_h\in C_{h'}$;
\item  $\forall h, h'\in  I_{> 1/a}$, if $h'\subset h$, then $C_{h'}\subset
C_{h}$, if $h\subset h'$, then $C_{h}\subset C_{h'}$, otherwise $C_{h}\cap
C_{h'}=\emptyset $; \end{itemize}  \end{enumerate}
\end{definition}

%A curve $C$ which admits a set of points and components with the above properties is said to be of  $I$--splitting type.

% the rigid cover

\begin{notation} By convention,  $\ol{M}^a_{\emptyset}= \ol{M}_{0,1}(\PP^n,d,a)$.
When $I=\{ h\}$, we will denote $\ol{M}^a_I$ simply as $\ol{M}^a_h$.

  For each nested set $I$, we let $G_I\subset S_d$ be the largest subgroup that keeps each $h\in I$ invariant. In particular, $|G_h|=|h|!(d-|h|)!$.
   %$G_I$ decomposes into a direct sum of permutation groups $S_{h'}$, where $h'\in \cP$ is one of the sets $h\setminus (\bigcap_{h''\in I, h''\not=h}h'')$ for $h\in I$, or $D\setminus (\bigcap_{h''\in I}h'')$.
\end{notation}

 When $|h|>1/a$, the stack $\ol{M}^a_h$ maps to a codimension 1 substack of $\ol{M}_{0,1}(\PP^n,d,a)$, and its generic point represents a map whose source is split as a union of two curves, one containing the marked point and the other of degree $|h|$. If $|h|\leq 1/a$, the stack $\ol{M}^a_h$ maps to a substack of higher codimension in $\ol{M}_{0,1}(\PP^n,d,a)$, and its generic point represents a map whose source contains the  marked point $p_1$ and another marked point of weight $|h|a$.

 There is a special \'etale atlas for $\ol{M}_{0,1}(\PP^n,d,a)$. Its definition is based on the notion of $\bar{t}$--rigid weighted stable maps, for any system  $\bar{t}$  of homogeneous coordinates on $\PP^n$. These are an adaptation of the rigid stable maps introduced in \cite{fulton-pand}, and were discussed in more detail in \cite{noi1}.
 \begin{definition} A $\bar{t}$--rigid stable map is given by the same data as a point of the moduli space $\ol{M}_{0,1}(\PP^n,d,a)$, together with an extra set of sections $ \{q_{i,j} \}_{0 \leq i \leq n, 1 \leq j \leq d}$ of weights $a_{i,j}=\frac{a}{(n+1)}$, and satisfying, via $e$, $$ ( \bar{t}_i) = \sum_{j=1}^d q_{i,j} $$ for each homogeneous coordinate $\bar{t}_i$ in $\bar{t}$.
 \end{definition}
 The moduli space of $\ol{M}_{0,1}(\PP^n,d,a, \bar{t})$ of $\bar{t}$--rigid weighted stable maps is represented by a torus bundle over an open subset of $\ol{M}_{0,\cA'}$, Hassett's moduli spaces of weighted stable curves. There is a natural action on $\ol{M}_{0,1}(\PP^n,d,a, \bar{t})$ by the finite group $(S_d)^{n+1}$, which permutes the extra marked points $q_{i,j}$ defined above.

 \begin{lemma} \label{etale weighted} (\cite{noi1}, Proposition 1.11, and in more generality in \cite{noi2}, Proposition 1.7.) An \'etale atlas for the stack  $\ol{M}_{0,1}(\PP^n,d,a)$ is given by \bea U:= \bigsqcup_{x, \bar{t}} U_x(\bar{t}), \eea where for suitable choices of finitely many coordinate systems $\bar{t}$   and of finitely many points $x$ in $\ol{M}_{0,1}(\PP^n,d,a, \bar{t})$, the smooth scheme $U_x(\bar{t})$ represents the quotient of an appropriately small affine neighborhood of $x$ by the largest small subgroup $H_x$ of the stabilizer $\mbox{ Stab}_x \subseteq (S_d)^{n+1}$. \end{lemma}
 This construction is based on \cite{vistoli} (proof of Proposition 2.8).
\begin{notation}
We choose $k$ to be the integer in $\{2,...,d\}$ with the property
$\frac{1}{k}\leq a < \frac{1}{k-1}$, or $k=1$ if $a=1$. (We note that if $0<a<\frac{1}{d}$, then the space $\ol{M}_{0,1}(\PP^n,d,a)$ is empty).
\end{notation}
With these notations, the following holds:

\begin{lemma} \label{U_x}

For each nested set $I$ whose elements $h$ satisfy $|h|\geq k$, the map $\phi_I: \ol{M}_I^a \to \ol{M}_{0,1}(\PP^n,d,a)$ is a proper local embedding. Moreover, for choices of $U_x(\bar{t})$ appropriately small, the \'etale atlas $U$ for the stack  $\ol{M}_{0,1}(\PP^n,d,a)$  defined in Lemma \ref{etale weighted} satisfies the properties listed in Proposition \ref{U} for each of these local embeddings $\phi_I$. Thus
\bea  U\times_{\ol{M}_{0,1}(\PP^n,d,a)}\ol{M}^a_I = \bigsqcup_{g\in S_d/G_I}V_{g(I)} \mbox{ and, more generally,  } V_I\times_{\ol{M}^a_I}\ol{M}^a_J = \bigsqcup_{g\in G_I/G_J}V_{g(J)}, \eea
 for $J\supset I$ whose elements also satisfy $|h|\geq k$. Here $V_{g(J)}$ are \'etale atlases of $\ol{M}_J^a$ naturally embedded in $U$, and $V_J \hookrightarrow V_I$ whenever $J\supset I$.

  Moreover, $V_J \bigcap V_K= V_{J\cup K}$ is a  transverse intersection in $V_{J\cap K}$ for all $J, K$ nested sets whose elements $h$ satisfy $|h|\geq k$.

Let $V_{\emptyset}:=U$. For each nested set $I$ and $h\not\in I$ with $|h|\geq k$,
 the network of local embeddings associated to $\phi_{Ih}^I:\ol{M}_{Ih}^a\to \ol{M}_I^a$ and the \'etale atlas $V_I$ consists of morphisms $\phi^J_K:\ol{M}_K^a\to \ol{M}_J^a$, where $K$ and $J$ are nested sets with $K\supset J\supseteq I$, whose elements $h'$ satisfy $|h'|= |h|$.
\end{lemma}

\begin{proof} We recall the relevant construction from  \cite{noi1} (proof of Proposition 2.3).
There exists a sequence of smooth varieties and morphisms:
\bea \ol{M}_{0,1}(\PP^n,d,a,\bar{t}) \hookrightarrow  \begin{CD}
\ol{M}_{0,0}(\PP^n\times\PP^1,(d,1),(a,1),\bar{t}) @>{f_{\bar{t}}}>> \PP^n_d(\bar{t}) @>{p_{\bar{t}}}>>
(\PP^1)^{d(n+1)}, \end{CD} \eea  where $\PP^n_d(\bar{t})$ is a
$(\CC^*)^n$--torus over an open subset of $(\PP^1)^{d(n+1)}$, and
$f_{\bar{t}}$  is a composition of blow--ups along some diagonals and their strict transforms, in a suitable order.

Let $N=\{0,...,n\}$, let $\Delta_{N\times h}$ denote the
diagonal in $(\PP^1)^{d(n+1)}$ where the coordinates corresponding to the set $N\times h \subset
N\times \{1,...,d\}$ agree. More precisely, the blow-ups are along the (strict transforms of) diagonals $\Delta_{N\times h}$ with $|h|\geq k$, in increasing order of dimension. At each blow-up step $l$, all (strict transforms of) diagonals with $|h|\geq l$ will intersect each other, and intersections of other (strict transforms of) diagonals, transversely (\cite{andrei}).

$D_{N\times h}^a(\bar{t}) \subset \ol{M}_{0,1}(\PP^n,d,a,\bar{t})$
is constructed from
$\Delta_{N\times h}^0(\bar{t}) =p_{\bar{t}}^{-1}(\Delta_{N\times h})$ by taking its strict transforms through the successive blow-ups, respectively the exceptional divisor at the $(d-|h|)$-th step, and finally intersecting the resulting space with
$ \ol{M}_{0,1}(\PP^n,d,a,\bar{t})$. We then let $ D_{N\times I}^a(\bar{t}):=\bigcap_{h\in I}
D_{N\times h}^a(\bar{t})$ for any nested set $I\subset\cP$.

Define
 \bean \label{VI}  V_I:= \bigsqcup_{x, \bar{t}}\bigsqcup_{[g'] \in (S_d/G_I)^n}(U_x(\bar{t}) \bigcap D^a_{g'(N\times I)}(\bar{t}))/H'_x, \eean
for finitely many $x$, $\bar{t}$ also employed in the construction of $U$, where the set of orbits $(S_d/G_I)^n$ corresponds to permutations on $(N\setminus\{0\})\times I$, and $H'_x$ is the largest small subgroup of the stabilizer of $x$ for the action of $G_I^{n+1}$ on $D^a_{g'(N\times I)}(\bar{t})$ (via conjugation by $g'$).

With the notations from the previous Lemma,  $H'_x=H_x$  whenever all the elements $h\in I$ satisfy $a|h|\geq 1$. Then each such $V_I$ embeds in $U$ and is an \'etale atlas of $\ol{M}^a_I$, canonically constructed like in \cite{vistoli} (proof of Proposition 2.8). Moreover, from the construction of $V_I$,
\bea  \label{etale for moduli} U\times_{\ol{M}_{0,1}(\PP^n,d,a)}\ol{M}^a_I = \bigsqcup_{g\in S_d/G_I}V_{g(I)}; \mbox{ as well, } V_I\times_{\ol{M}^a_I}\ol{M}^a_J = \bigsqcup_{g\in G_I/G_J}V_{g(J)}, \eea
for $J\supset I$ whose elements also satisfy $a|h|\geq 1$.
These formulae are equivalent to Lemma \ref{fibered product} for the spaces $\ol{M}^a_I$, which uniquely defines networks of local embeddings. Accordingly, for each nested set $I$ and $h\not\in I$ with $|h|\geq k$,
the  morphisms $\phi^J_K:\ol{M}_K^a\to \ol{M}_J^a$, where $K$ and $J$ are nested sets with $K\supset J\supseteq I$, whose elements $h'$ satisfy $|h'|= |h|$, form  the network of local embeddings associated to $\phi_{Ih}^I:\ol{M}_{Ih}^a\to \ol{M}_I^a$ and the \'etale atlas $V_I$.

Furthermore, since the diagonals and their transforms intersect each other, and intersections of other diagonals, transversely, it follows that $V_J \bigcap V_K= V_{J\cup K}$ is a  transverse intersection in $V_{J\cap K}$ for all $J, K$ nested sets whose elements $h$ satisfy $a|h|\geq 1$.

\end{proof}

Let $d$ be a positive integer. As described in \cite{noi2}, contractions $\ol{M}_{0,m}(\PP^n, d, a)$  of the space $\ol{M}_{0,m}(\PP^n, d)$  can be thought of as locally embedded in $\ol{M}_{0,1}(\PP^n, d+m-1,
a)$ when $a<1$. More precisely, after choosing a privileged point $1_M$ among the $m$ marked points and adjoining the rest of marked points $M=\{
2_M, ..., m_M \}$ to the set $\{ 1_D,..., d_D \}$ to form $D'=\{1_D,..., d_D, 2_M, ..., m_M \}$, one can write
\bea \ol{M}_{0,\cA}(\PP^n, d, a) \cong \ol{M}_{J}^{a} \eea for $J=\{\{2_M\}, \{3_M\}, ..., \{d_M\} \}$
and the $m$-tuple $\cA=(1,a,...,a)$.

For this reason we introduce the following notation.
\begin{notation}
From now on, we will let $\cA$ denote the $m$-tuple $(1,a,...,a)$, and work with subsets $h$ of $D'=\{1_D,..., d'_D, 2_M, ..., m_M \}$, and all nested subsets will be elements in $\cP(D')$.
We let $d':=d+m-1$. As before, we choose $k>1$ to be the integer with the property
$\frac{1}{k}\leq a < \frac{1}{k-1}$, or $k=1$ if $a=1$.
\end{notation}

% de reformulat paragraful urmator

 We note that while for the generic curve parametrized by
 $\ol{M}^a_I$, all tails and base locus points are marked by
 elements of $I$, other curves, represented by points in the
 boundary of $\ol{M}^a_I$, may have unmarked tails and base
 points. It is due to such points that the maps
$\ol{M}^a_J \to \ol{M}^a_I $ are in general only local embeddings. This suggests that marking all components and
base points of curves and their maps will result in moduli stacks $\ol{M}'^a_I$ embedded in $\ol{M}'_{0,\cA}(\PP^n, d, a)$, constructed like in Theorem \ref{X'}.

% de inlocuit $k$--stable with $(\cA, a)$--stable or  $(\cA_k, a_k)$--stable everywhere, as needed; de inlocuit $\ol{M}_I^k$ cu
%$\ol{M}_I^a$

\begin{definition} \label{semi-rigid}
Consider the moduli functor from schemes to sets, associating to any scheme $S$ the set of
 $(\cA, a)$--stable, degree $d$ pointed maps $(C\to S, \{ p_i\}_ {i\in\{1,..., m\} }, \cL, e)$, together with a collection $\{P_s\}_{s\in S}$ of partitions
 $$ D = \sqcup_{ \alpha \in N_s } B_{\alpha}$$
 of the set $D=\{1,..., d\}$, one for each $s \in S$, such that for every $s \in S$, the set
 %$$N_s \cong \{ \mbox{ irreducible components of the curve } C_s \} \sqcup \{ x \in C_s;   \Coker e_x \not= \emptyset \}, \mbox{ and }$$
 $$        N_s \cong \{ \mbox{ irreducible components of the curve } C_s \}, \mbox{ and }$$
\begin{enumerate}
\item The partition is compatible with the structure of the map  given by $(\cL, e)$: If $\alpha\in N_s$ corresponds to the component $C\subset C_s$ then $|B_{\alpha}|= \deg \cL_{|_{C}}$;
\item The partition is compatible with specialization in $S$: If $s_1, s_2 \in S$ then  $s_1\in \ol{\{s_2\} } $ $\Leftrightarrow$ $P_{s_1}$ is a refinement of $P_{s_2}$.
\end{enumerate}

A set of data as above will be called a semi-rigid $(\cA, a)$--stable map over $S$. An isomorphism of
semi-rigid $(\cA, a)$--stable maps is an isomorphism of $(\cA, a)$--stable maps which also preserves the partitions of $D$.

\end{definition}

\begin{definition} \label{M'}
Let $\ol{M}'_{0,\cA}(\PP^n, d,a) \to \ol{M}_{0,\cA}(\PP^n, d,a)$ be the \'etale surjective morphism constructed inductively as follows:
\bea  \begin{array}{cc} \ol{M}^{(d')}_{0,\cA}(\PP^n, d,a):= \ol{M}_{0,\cA}(\PP^n, d,a), & \ol{M}^{a, (d')}_h:=\ol{M}^{a}_h, \end{array} \eea while
 for $l \in \{k, ..., d'-1\}$,
  \bea \ol{M}^{(l)}_{0,\cA}(\PP^n, d,a)\to \ol{M}^{(l+1)}_{0,\cA}(\PP^n, d,a)\eea is the \'etale lift for the proper local embedding
$\ol{M}^{a, (l+1)}_h \to \ol{M}^{(l+1)}_{0,\cA}(\PP^n, d,a)$, for some $h$ satisfying $|h|=l$, with the \'etale atlas $U$ from Lemmas \ref{etale weighted} and \ref{U_x}, while for each nested set $I$ such that $h\not\in I$,
 \bea \ol{M}^{a, (l)}_{I}\to \ol{M}^{a,(l+1)}_{I}\eea is the \'etale lift for the proper local embedding
$\ol{M}^{a, (l+1)}_{hI} \to \ol{M}^{a,(l+1)}_{I}$ with the \'etale atlas $V_{I}$.
Finally,
\bea \ol{M}'_{0,\cA}(\PP^n, d,a):= \ol{M}^{(k)}_{0,\cA}(\PP^n, d,a).\eea
\end{definition}

The following Lemma ensures that Theorem \ref{X'} can be applied at each step in the definition above.
\begin{lemma} \label{etale lift-local embeddings commute} The following properties hold for the spaces $\ol{M}^{a,(l)}_{I}$:
\begin{enumerate}
\item There exist natural proper local embeddings $\phi_J^{I, (l)}: \ol{M}^{a, (l)}_{J} \to \ol{M}^{a,(l)}_{I}$;
\item For $h\not\in I$, the  maps $\phi_K^{J, (l)}:\ol{M}^{a, (l)}_{K} \to \ol{M}^{a,(l)}_{J}$, with $K\supset J \supseteq I$ such that the elements $h'\in K\setminus J$ satisfy $|h'|= |h|$, form the network of local embeddings associated to the proper local embedding $\phi_{Ih}^{I, (l)}$ with the \'etale atlas $V_I$ defined by formula (\ref{VI}).
\item The following diagram is Cartesian
\bea \diagram  \ol{M}^{a, (l)}_{J\bigcup K} \rto \dto & \ol{M}^{a, (l)}_{J} \dto \\\ol{M}^{a, (l)}_{K} \rto & \ol{M}^{a, (l)}_{K\bigcap J},
\enddiagram \eea
\item  For $J \supset I$ and $l>k$, the following diagram is Cartesian
\bea \diagram  \bigsqcup_{g\in G_{J_{l-1}}/G_J}\ol{M}^{a, (l-1)}_{g(J)} \rto \dto & \ol{M}^{a, (l-1)}_{I} \dto \\\ol{M}^{a, (l)}_{J} \rto & \ol{M}^{a, (l)}_{I},
\enddiagram \eea
where $J_{l-1}=\{ h\in J;  |h|=l-1 \mbox{ or } h\in I\}$.
\end{enumerate}
all properties holding as long as all the elements $h$ of $J\setminus I$, $K\setminus I$ satisfy $|h|\geq k$. Here, according to our convention, $\ol{M}^{a, (l)}_{\emptyset}=\ol{M}^{(l)}_{0,\cA}(\PP^n, d,a)$.

\end{lemma}
\begin{proof}
Properties (1)-(3) follow by  decreasing induction on $l$. The first step, when $l=d'$, is true due to Lemma \ref{U_x}. If (1)-(3) hold for $l>k$, then by (3), Corollary \ref{X_Y functorial} can be applied at each step in the proof of Theorem \ref{X'}, yielding the following Cartesian diagram
  \bea \diagram  \ol{M}^{a, (l-1)}_{J} \rto \dto & \ol{M}^{a, (l-1)}_{I} \dto \\\ol{M}^{a, (l)}_{J} \rto & \ol{M}^{a, (l)}_{I},
\enddiagram \eea
 for $J \supset I$ such that $J\setminus I$ contains no elements $h$ with $|h|=l-1$, and $l>k$, while for $J \supset I$ such that $J\setminus I$ contains only elements $h$ with $|h|=l-1$, Theorem \ref{X'} implies that the following  diagram is Cartesian
 \bea \diagram  \bigsqcup_{g\in G_{I}/G_J}\ol{M}^{a, (l-1)}_{g(J)} \rto \dto & \ol{M}^{a, (l-1)}_{I} \dto \\\ol{M}^{a, (l)}_{J} \rto & \ol{M}^{a, (l)}_{I},
 \enddiagram \eea
 In general, splitting $J=J_{l-1}\bigsqcup (J\setminus J_{l-1})$ yields
 (4). As well,  at the level of \'etale atlases:
 \bea \ol{M}^{a, (l)}_{J} \times_{\ol{M}^{a, (l)}_{I}}V_I \cong  \bigsqcup_{g\in G_{J_{l-1}}/G_J}V_{g(J)}.\eea
 This ensures that (1)-(3) are true when $l$ is replaced by $(l-1)$ as well.
 Indeed,  (1) and (3) follow directly from the Cartesian diagrams above, while (2) follows from the relation between \'etale atlases, which is equivalent to Lemma \ref{fibered product} for our spaces.
\end{proof}

\begin{remark} Alternatively, $\ol{M}'_{0,\cA}(\PP^n, d,a)$ can be constructed by applying the construction steps in the proof of Theorem \ref{X'} directly to the network $\{\phi^I_J:\ol{M}_J^a\to \ol{M}_I^a\}$, where $I$ and $J$ are nested sets whose elements $h$ satisfy $|h|\geq k$. Indeed, the \'etale atlases do not change throughout the construction, while the groupoid relations of $\ol{M}'_{0,\cA}(\PP^n, d,a)$, obtained by applying equation (\ref{relations_for_X'}) to the \'etale atlases of Lemma \ref{U_x}, are independent of the order in which the \'etale lifts above are performed.
\end{remark}

\begin{remark}
One could ask if simply applying Proposition \ref{U} and Theorem \ref{X'} directly to the local embedding
\bea Y:=\bigsqcup_h\ol{M}^a_h \to X:=\ol{M}_{0,\cA}(\PP^n, d,a) \eea
 (when the union is taken after one copy of $h$ for each cardinality $|h|\geq 1/a$) would  not yield the same outcome as in Definition \ref{M'}.
However we note that  for the map above, the \'etale atlas $U$ introduced in Lemma \ref{etale weighted} does admit a partition of $Y\times_XU$  satisfying all properties listed in Proposition \ref{U}. Moreover, the \'etale lift associated to the map above would contain the same number of copies of boundary divisors mapping onto $\ol{M}^a_h$ irrespectively of $|h|$, which is different from the case of $\ol{M}'_{0,\cA}(\PP^n, d,a)$ constructed by us.

We also note, in view of Remark \ref{minimality}, that while the \'etale atlas $U$ from Lemma \ref{etale weighted} is suitable for all the local embeddings
\bea \ol{M}^a_h \to \ol{M}_{0,\cA}(\PP^n, d,a) \eea
(with $|h|\geq 1/a$) at once, it is not minimal for each map taken separately. For example, if $|h|>d/2$ then  $\ol{M}^a_h$ is embedded in $\ol{M}_{0,\cA}(\PP^n, d,a)$ and yet, $\ol{M}'_{0,\cA}(\PP^n, d,a)$ will still contain $\left(\begin{array}{l}d\\|h|\end{array}\right)$ copies of
$\ol{M}'^a_h$.
\end{remark}

\begin{theorem} \label{rigid spaces}
The moduli problem of semi-rigid $(\cA, a)$--stable maps is finely
represented by the Deligne-Mumford stack $\ol{M}_{0,\cA}'(\PP^n,d,a).$
\end{theorem}

\begin{proof}

Let $p:\ol{M}_{0,\cA}'(\PP^n,d,a) \to \ol{M}_{0,\cA}(\PP^n,d,a)$ denote the universally closed, \'etale map resulting from the construction in Definition \ref{M'}, and let $\cU$ denote the universal family on $\ol{M}_{0,\cA}(\PP^n,d,a)$. For each $s\in  \ol{M}_{0,\cA}'(\PP^n,d,a)$, we will define a partition $P_s$ of $D$ compatible with $\cU_{p(s)}$ in the sense of Definition \ref{semi-rigid}.
Then  $\cU$ together with the family of partitions $\{ P_s\}_s$ will form the universal family $\cU'$ on $\ol{M}_{0,\cA}'(\PP^n,d,a)$.

Indeed, for every $s\in \ol{M}_{0,\cA}'(\PP^n,d,a)$ there exists a uniquely associated nested set $I$ (possibly $I=\emptyset$) such that $s\in (\ol{M}^a_I)'\setminus (\bigcup_{J\supset I}(\ol{M}^a_J)')$, and if $I_1$ and $I_2$ are associated to $s_1, s_2$, respectively, then \bea  s_1\in \ol{\{s_2\} } \Rightarrow I_1 \supseteq I_2.  \eea
There exists a correspondence between nested sets $I\in \cP(D)$ and partitions $P_I$ of $D$, such that
\bea J \supseteq I  \Longleftrightarrow  P_J \mbox{ is a refinement of } P_I:  \eea

The elements of the partition $P_I$  are the sets $h\setminus (\bigcup_{h''\in I, h''\not=h}h'')$ for all $h\in I$, together with $D\setminus (\bigcup_{h''\in I}h'')$.

Conversely, given any semi-rigid $(\cA, a)$--stable map $(C\to S, \{ p_i\}_ {i}, \cL, e)$, then for each point $s\in S$, the associated partition $P_s$ of $D$ uniquely defines a nested set $I(s)$: The elements $h$ of $I(s)$ correspond to chains of components of $C$, of length at least two, which start from the component containing the special marked point $p_1$ and end with a tail. Then $h=\bigcup_{\alpha} B_{\alpha}$, where the union is taken after all the components $C_\alpha$ in the chain with the exception of the first one.
This leads to a stratification of $S$ indexed by nested sets, with locally closed strata
\bea  \begin{array}{ll} \{  s\in S;  I(s)=I   \}. \end{array} \eea
Let $S_I$ denote the closure of the above set in $S$. Then by condition (2) in Definition \ref{semi-rigid},
\bea S_I = \begin{array}{ll} \{  s\in S;  I(s)\supseteq I   \}. \end{array} \eea
If $J$ and $K$ are nested sets such that $J\cup K$ is nested as well, then $S_J\bigcap S_K = S_{J \cup K}$.
For each nested set $I$, there exists a natural map $f_I: S_I \to \ol{M}^a_I$, obtained by forgetting the partitions $P_s$ but remembering the associated nested sets. When $I=\emptyset$ we get $f: S \to \ol{M}_{0,\cA}(\PP^n, d,a)$. If $J \supset I$, then the diagram below is commutative
\bea \diagram  S_J \rto^{f_J} \dto_{\subset } & \ol{M}^a_J \dto^{\phi_J^I} \\
 S_I \rto^{f_I} & \ol{M}^a_I. \enddiagram  \eea
Moreover, if $G_I\subset S_d$ is the largest subgroup that keeps each $h\in I$ fixed,
 \bean \label{prod} S_I\times_{\ol{M}^a_I}\phi_J^I(\ol{M}^a_J) \cong \bigcup_{g\in G_I/G_J} S_{g(J)}. \eean
  These conditions are sufficient to define $f': S \to \ol{M}'_{0,\cA}(\PP^n, d,a)$. Indeed,  lifts $f'_K: S_K \to (\ol{M}^a_K)'$ can be constructed by decreasing induction on $K$. For the largest sets $K\supset I$, relation (\ref{prod}) implies $S_I\times_{\ol{M}^a_I}\phi_K^I(\ol{M}^a_K) \cong \bigsqcup_{g\in G_I/G_K} S_{g(K)}\to \bigsqcup_{g\in G_I/G_K}(\ol{M}^a_{g(K)})'$, which by Theorem \ref{universality} induces a morphism from $S_I$ to the \'etale lift $(\ol{M}^a_I)^{(K)}$ of $\phi_K^I$; the commutative diagram above lifts as well.
   Now for any nested set $K$, moving on to the step when $(\ol{M}^a_I)'$ have been constructed for all $I\supset K$, then by \cite{abramovich}, (Appendix 1), the lifts of the commutative diagram above glue to a morphism
  \bea  S_K\times_{\ol{M}^a_K}\phi_I^K(\ol{M}^a_I) \cong \bigcup_{g\in G_I/G_K} S_{g(I)}\to N_K,\eea
where $N_K$ is obtained by gluing all $(\ol{M}^a_{g(I)})'$ along $(\ol{M}^a_{g(I)\cup g'(I)})'$. We note that $\phi_I^K(\ol{M}^a_I)$ is the image of $N_K$ in $\ol{M}^a_K$. Now  the lift $(\ol{M}^a_{K})'$ is constructed so that $N_K \hookrightarrow (\ol{M}^a_{K})'$.  Again by  Theorem \ref{universality}, one obtains a lift $S_K \to (\ol{M}^a_K)'$.

 We note that  $(C\to S, \{ p_i\}_ {i }, \cL, e)$ is the pullback through $f$ of the universal family $\cU$. As for every $s\in S$, the partition
 $P_s$ on $C_s$ is completely determined by the nested set $I$ such that  $s\in S_I\setminus(\bigcup_{J\supset I}S_J)$, it follows that the partition $P_s$
 is also inherited by pullback through $f'$ from the universal family on $\ol{M}_{0,\cA}'(\PP^n, d,a)$.
\end{proof}

\begin{notation}
For each  integer $k\in\{2,...,d\}$, let $a_k$ be any real number with the property
$\frac{1}{k}\leq a_k < \frac{1}{k-1}$. Let  $a_1=1$.
\end{notation}

%For $m\geq 1$, the moduli space of stable maps with marked points $\ol{M}_{0,m}(\PP^n,d)$ corresponds to weights $a>1$ and $\cA=(1,1,...,1)$.
By \cite{noi2}, a sequence of birational contractions of $\ol{M}_{0,m}(\PP^n,d)$ is given by the spaces $\ol{M}_{0,\cA_k}(\PP^n,d,a_k)$ with $\cA_k=(1,a_k,...,a_k)$ for $d>k>0$.

\begin{lemma}\label{blowup etale lift commute}
For each integer $l$ with $d'\geq l> k$, the morphism $f_I^{k,(l)}:\ol{M}^{a_{k-1}, (l)}_I \to \ol{M}^{a_{k}, (l)}_I$ is a weighted blow-up along the local embedding $\phi_{Ih}^{I, (l)}: \ol{M}^{a_{k}, (l)}_{Ih}\to \ol{M}^{a_{k}, (l)}_I$ with $h\not\in I$ such that $|h|=k$.

If $k=l$, then $\phi_{Ih}^{I, (l)}$ is an embedding and $f_I^{k,(l)}$ is the weighted  blow-up along all such embeddings with $|h|=l$.
\end{lemma}

\begin{proof}
The statement follows by decreasing induction on $l$. The paragraph above the Lemma sets up the initial step $l=d$. Assume that the Lemma holds for $l>k$.
Consider $h$ and $h'$ such that $|h|=k$ and $|h'|=l-1$. By the induction hypothesis, both $f_I^{k,(l)}$ and $f_{Ih'}^{k,(l)}$ are  weighted blow-ups along the local embeddings, and the corresponding local embeddings are connected by a Cartesian diagram due to Lemma \ref{etale lift-local embeddings commute}(3) applied to $K=I\cup\{h\}$ and $J=I \cup\{h'\}$. Thus the natural diagram containing $f_I^{k,(l)}$ and $f_{Ih'}^{k,(l)}$ is Cartesian, following the definition of  blow-ups along the local embeddings and the
natural \'etale atlases described in the proof of Lemma \ref{etale lift-local embeddings commute}. We can now apply Corollary
\ref{X_Y functorial} at each step in the construction of the $(l-1)$--\'etale lifts, leading to the following diagram being Cartesian:
\bea \diagram  \ol{M}^{a_{k-1}, (l-1)}_{I} \rrto^{f_I^{k,(l-1)}} \dto && \ol{M}^{a_k, (l-1)}_{I} \dto \\\ol{M}^{a_{k-1}, (l)}_{I} \rrto^{f_I^{k,(l)}} && \ol{M}^{a_k, (l)}_{I}.
\enddiagram \eea
If $l-1>k$, the following diagram is Cartesian both when $j=k$ and when $j=k-1$ by Lemma \ref{etale lift-local embeddings commute}(4).
\bea \diagram  \ol{M}^{a_{j}, (l-1)}_{Ih} \rrto^{\phi_{Ih}^{I, (l-1)}} \dto && \ol{M}^{a_j, (l-1)}_{I} \dto \\\ol{M}^{a_{j}, (l)}_{Ih} \rrto^{\phi_{Ih}^{I, (l)}} && \ol{M}^{a_j, (l)}_{I}.
\enddiagram \eea
%$\mbox{ and } \diagram \ol{M}^{a_{k-1}, (l-1)}_{Ih} \rto \dto & \ol{M}^{a_{k-1}, (l-1)}_{I} \dto \\ \ol{M}^{a_{k-1}, (l)}_{Ih} \rto & \ol{M}^{a_{k-1},(l)}_{I}.\enddiagram  \eea
Since $f_I^{k,(l)}$ was assumed to be a  weighted blow-up along the local embedding $\phi_{Ih}^{I, (l)}$, from the above it follows that $f_I^{k,(l-1)}$ is a weighted blow-up along the local embedding $\phi_{Ih}^{I, (l-1)}$.
Finally, when $l-1=k$, Lemma \ref{etale lift-local embeddings commute}(4) implies that the following diagrams are Cartesian:
\bea \diagram  \bigsqcup_{|h''|=k}\ol{M}^{a_{j}, (l-1)}_{Ih''} \rrto^{\sqcup\phi_{Ih''}^{I, (l-1)}} \dto && \ol{M}^{a_j, (l-1)}_{I} \dto \\\ol{M}^{a_{j}, (l)}_{Ih} \rrto^{\phi_{Ih}^{I, (l)}} && \ol{M}^{a_j, (l)}_{I},
\enddiagram \eea
and $f_I^{k,(k)}$ is the weighted blow-up along $\bigsqcup_{|h''|=k}\ol{M}^{a_{j}, (k)}_{Ih''}$.
\end{proof}

 \begin{theorem}
 Consider the normalization $\ol{M}^{a_{k}}_I$ of a boundary stratum in  $\ol{M}^{a_{k}}_{\emptyset}:=\ol{M}_{0,1}(\PP^n, d, a_k)$.
The total Chern class of $\ol{M}^{a_{k}}_I$ is written in $A(\ol{M}^{a_{k}, (k)}_{\emptyset})$ as
\bea    &  c(\ol{M}^{a_{k}}_I)= (1+H)^{n+1}(1+\psi)^{s_I-1}\prod_{i=1}^{d-l_I}(1+H+i\psi)^{n+1} & \\
& \prod_{h; |h|>k} \frac{(1+D_h)(1+\psi_{h})^{|I_h|-1}
\prod_{j=1}^{|h \setminus \cup_{h'\in
I}h'|}(1+H_{I;h}+j\psi_{h})^{n+1}}{(1+\psi_{h}^0)^{|I_h|-1}
\prod_{j=1}^{|h \setminus \cup_{h'\in
I}h'|}(1+H_{I;h}+j\psi_{h}^0)^{n+1}} & \eea where the product is
taken after all $h$ such that $I\cup \{ h\}$ is still a nested set.
Here $l_I := | \cup_{h\in I} h|$ and $s_I$ is the number of maximal
elements of $I$.

 $H$ denotes the pullback of the hyperplane divisor from $\PP^n$ and $\psi$ denotes the cotangent line class, $D_h$ represents the class of the divisor $\ol{M}^{a_k}_{\{ h \} }$,
while
\bea \psi_{h}:=\psi-\sum_{h'\supseteq h} D_{h'} \mbox{ and }
\psi_{h}^0:=\psi-\sum_{h'\supset h} D_{h'},  \eea \bea
H_{I;h}:=H+(d-|h \cup(\cup_{h'\in I}h')|)\psi-\sum_{h'\supset h}|h'\backslash (h\cup(\cup_{h''\in I}h''))|D_{h'}.\eea

 \end{theorem}

\begin{proof}

   We proceed by induction on $k$. The last member in the sequence of contractions of $ \ol{M}_{0,1}(\PP^n,d) $ is  $\ol{M}_{0,1}(\PP^n, d, a_{d-1})$, a weighted projective fibration $\cP(A)$ over $\PP^n$ described in \cite{noi1}, Lemma 3.3.
%de introdus notatii consecvente mai devreme cand definesc weighted projective fibration
The normal bundle of the zero section in $A$ splits by weights as $\cN_{\PP^n|A} = \oplus_l\cN_l/\cN_{l+1}$, where $$\cN_l=(\oplus_{i=1}^{(n+1)(d-l)}\cO_{\PP^n}(1))/\cO_{\PP^n}.$$
   Thus by Proposition \ref{proj fibr}, the total Chern class
              $$c(\ol{M}_{0,1}(\PP^n,d,a_{d-1})) = (1+H)^{n+1}(1+\psi)^{-1}\prod_{l=1}^{d}(1+H+l\psi)^{n+1}.$$

    More generally, for any $I\in \cP(D)$,  denote $l_I := | \cup_{h\in I} h|$ and let $s_I$ be the number of maximal elements of $I$, i.e. elements $h\in I$ such that there is no $h'\in I$ with $h\subset h'$.

     The normal bundle $$\left.
      \cN_{\ol{M}^{a_{d-1}}_I |\ol{M}^{a_{d-1}}} = \cN_{\PP^1\times\PP^n |
(\PP^1)^{s_I}\times\PP^n_{d-l_I}}\right|_{\PP^n}=
(\cO_{\PP^n}^{s_I}\oplus\oplus_{i=1}^{(n+1)(d-l_I)}\cO_{\PP^n}(1))/\cO_{\PP^n}$$
admits a natural filtration  $\{\cN_l\}_l$ described in Lemma 3.2 of \cite{noi1}. Here
$$\cN_l=(\cO_{\PP^n}^{s_I}\oplus\oplus_{i=1}^{(n+1)(d-l_I-l)}\cO_{\PP^n}(1))/\cO_{\PP^n}.$$
   $\CC^*$ acts on the bundle
$\cN=\oplus_l \cN_l/\cN_{l+1}$  with weights $(1,...,d-l_I)$. Thus by Proposition \ref{proj fibr}, the total Chern class of the weighted projective fibration $\ol{M}^{a_{d-1}}_I  \to \PP^n$ is
$$ c(\ol{M}^{a_{d-1}}_I )=(1+H)^{n+1}(1+\psi)^{s_I-1}\prod_{i=1}^{d-l_I}(1+H+i\psi)^{n+1}. $$
 The morphism $f_I^{a_{k} }: \ol{M}^{a_{k-1}}_I\to \ol{M}^{a_{k}}_I $ is a weighted blow-up along the local embedding $\ol{M}^{a_{k}}_{Ih} \to \ol{M}^{a_{k}}_I$, for $h\not\in I$ with $|h|=k$. By Lemma \ref{blowup etale lift commute}, the $k$-th \'etale lift $\ol{M}^{(a_{k-1}, (k))}_I\to \ol{M}^{(a_{k}, (k))}_I $ is obtained by successive weighted blow-ups of all embeddings $ \ol{M}^{(a_{k}, (k))}_{Ih}  \to \ol{M}^{(a_{k}, (k))}_I $ where $h\not\in I$ with $|h|=k$. As the \'etale covers are preserved through \'etale lifts, then with the notations from  Lemma $\ref{U_x}$ , the \'etale atlas $V_I$ of $\ol{M}^{(a_{k}, (k))}_I $ satisfies $V_I\times_{\ol{M}^{(a_{k}, (k))}_I }\ol{M}^{(a_{k}, (k))}_{Ih} =V_{Ih}$ by construction, while by Lemma $\ref{U_x}$ , all strata $V_{Ih}$ with $|h|=k$ intersect each other, and their intersections, transversely.
 The normal bundle of the blow-up locus for $\ol{M}^{a_{k-1}}_I\to \ol{M}^{a_{k}}_I $ when $k\in \{ 1,..., d-1\}$ and the weights of the appropriate $\CC^*$ action on it have been calculated in \cite{noi1}, Lemma 3.21. Let $h\in \cP(D)$ be such that $|h|=k$ and denote by $I_h\subseteq I$ the set of all $h'\in I$ such that $h'\subset h=\emptyset$.

 The top equivariant Chern class of the normal bundle $\cN_{\ol{M}^{a_{k}}_{I\cup\{h\} }|\ol{M}^{a_{k}}_I}$ when  evaluated at $t=D_h$ is $$ (1+\psi_{h})^{|I_h|-1} \prod_{j=1}^{|h \setminus \cup_{h'\in I_h}h'|}(1+H_{I;h}+j\psi_{h})^{n+1}  $$
 and thus formula (\ref{general blowup}) reads in this case
%nu-i nevoie de $I_h$?
  \bea c(\ol{M}^{a_{k-1}}_I)= f_I^{a_{k} *}c(\ol{M}^{a_{k}}_I)\prod_{h; |h|=k}\frac{(1+D_h)(1+\psi_{h})^{|I_h|-1} \prod_{j=1}^{|h \setminus \cup_{h'\in I}h'|}(1+H_{I;h}+j\psi_{h})^{n+1}}{(1+\psi_{h}^0)^{|I_h|-1} \prod_{j=1}^{|h \setminus \cup_{h'\in I}h'|}(1+H_{I;h}+j\psi_{h}^0)^{n+1}}.  \eea
 Iterating the formula above for all $k\in \{ 1,..., d-1\}$, we obtain the relation stated in the Theorem.
 \end{proof}

% $\psi_h^I$ mai degraba.

In particular,
\bea  & c(\ol{M}_{0,1}(\PP^n, d))= (1+H)^{n+1}\frac{\prod_{i=1}^{d}(1+H+i\psi)^{n+1}}{(1+\psi)}
 \prod_{h} \frac{(1+D_h)(1+\psi_{h}^0)\prod_{j=1}^{|h|}(1+H_{h}+j\psi_{h})^{n+1}}{(1+\psi_{h}) \prod_{j=1}^{|h|}(1+H_{h}+j\psi_{h}^0)^{n+1}}, & \eea
where $H_h=H+(d-|h|)\psi-\sum_{h'\supset h}|h'\backslash h|D_{h'}= H_{\emptyset;h}$ from above.
We recall that the
spaces $\ol{M}_{0,m}(\PP^n, d, a_k)$ with $k\in \{1,..., d-1 \}$ can
be thought of as normalized strata of $\ol{M}_{0,1}(\PP^n, d+m-1,
a_k)$. Indeed, as mentioned before,  one can write
$$ \ol{M}_{0,m}(\PP^n, d, a_k) \cong \ol{M}_{I}^{a_k} $$ for $I=\{
\{2_M\}, \{3_M\}, ..., \{d_M\} \} \subset \cP(\{1_D,..., d_D, 2_M, ..., m_M \})$. Thus \bea  &
c(\ol{M}_{0,m}(\PP^n, d))= (1+H)^{n+1}(1+\psi)^{m-2}\prod_{i=1}^{d}
(1+H+i\psi)^{n+1} & \\
  & \prod_{h\in \cP(D')\setminus \{ \emptyset, \{2_M\},...,\{m_M\}\} }
  \frac{(1+D_h)(1+\psi_{h})^{|h_M|-1}\prod_{j=1}^{|h \cap D|}
  (1+H_{m;h}+j\psi_{h})^{n+1}}{(1+\psi_{h}^0)^{|h_M|-1}
  \prod_{j=1}^{|h \cap D|}(1+H_{m;h}+j\psi_{h}^0)^{n+1}},
  & \eea
where $h_M=h\cap \{2_M, ...,m_M\}$ and
 \bea H_{m;h}=H+|D\setminus h|\psi-\sum_{h'\supset h}|D\cap h'\backslash
h|D_{h'}.\eea

 In particular, one recovers the formulae for the first Chern classes
as calculated in \cite{pand} and in a more general setup in
\cite{starr-dejong}.

Let $f: \ol{M}_{0,1}(\PP^n, d) \to \ol{M}_{0,0}(\PP^n, d)$ be the
forgetful morphism. In \cite{noi3} this has been split into a
sequence of simple blow-ups and a relatively simpler fibration  $
\ol{U}^{\lfloor (d-1)/2\rfloor }\to \ol{M}_{0,0}(\PP^n, d)$.

Choose $I:=\{h\subset \{1,...,d\}, |h|>d/2\}$ if $d$ is odd, and let
$I$ additionally contain half of the sets $h$ with $|h|=d/2$ if $d$
is even, under the condition that no two sets $h$, $\bar{h}$ are
simultaneously in $I$. The class
$$\psi'_I:= \psi-\sum_{h\in I}D_h   $$
can be considered as the $\cO(1)$--line bundle for the fibration $
\ol{U}^{\lfloor (d-1)/2\rfloor }\to \ol{M}_{0,0}(\PP^n, d)$
above. When $d$ is odd, $\psi'_I$ is the pullback of the relative
cotangent class for the morphism $\ol{U}^{\lfloor (d-1)/2\rfloor
}\to \ol{M}_{0,0}(\PP^n,d)$. The classes $D_h$ with $h\in I$ are
pullbacks of the classes of the exceptional divisors on
$\ol{M}_{0,1}(\PP^n, d)$. From the sequence of blow-ups mentioned above we obtain a formula comparing the total Chern
classes of moduli spaces with/without marked points:
$$ f^*(c(\ol{M}_{0,0}(\PP^n, d)))= c(\ol{M}_{0,1}(\PP^n,
d))\frac{\prod_{h\in I}(1+D_h)P_h(1-D_h)}{P(1+\psi_I)\prod_{h\in
I}P_h(1)},$$ where $P(\psi_I)$ and $P_h(-D_h)$ are the quadratic
expressions  in \cite{noi2}, Theorem 3.3., (1)--(2).

 Proposition
2.1. in \cite{noi3} shows how to recover the class
$c(\ol{M}_{0,0}(\PP^n, d))$ from its pullback.

%$$D_I(h):= \sum_{h' \in I, h'\subset h} D_{h'} \mbox{ and } D_I(\bar{h}):= \sum_{h' \in I, h'\subset h} D_{\bar{h}'}$$
% for any $h \in I$. The class $\psi_I (h)$ is defined as $\psi_I (h):=\psi'_I +D_I(h)$.

%de facut comparatia cu cazul fara weight-uri si de ce nu merge

% trebuie rezolvat anterior cazul unei "weighted projective fibration"; prin deformare la conul normal obtinem aceleasi clase Chern?

\section*{Appendix. Euler's sequence for a weighted projective bundle}

%de regandit acest prim paragraf impreuna cu materialul compilat in introducere,
%caci aici de fapt lucrez cu fibrat, nu fibrare
Let $g: P\to Y$ be a weighted projective bundle and a smooth
morphism of stacks. With the notations from Section 2, consider the
splitting $\cN_{Y|A}=\oplus_i L_i$, such that for all $n$ with $
\cN_n\not= \cN_{n+1}$ there is a unique index $i$ with
$L_i=\cN_n/\cN_{n+1}$. Denote by $w_i=n$ the weight of the naturally
induced $\CC^*$--action on $L_i$.

Let $\cT_{P|Y} = \Ker ( \cT_{P} \to g^*\cT_Y )$ denote the relative
tangent bundle of $g:P\to Y$.

\begin{lemma}
There is an exact sequence of vector bundles on $P$ \bean \diagram 0
\rto & \cO_P \rrto^{ \sigma    } && {\oplus_i g^* L_i \otimes
\cO_P(w_i)} \rrto^{    e } && \cT_{P|Y} \rto & 0.
\enddiagram  \eean
\end{lemma}

\begin{proof}

 The weighted projective bundle $P$ is locally trivial over $Y$,
i.e. $P_{|U}\cong \PP[w_0:...:w_n]\times U$ for each open set $U$ in
some open cover of $Y$. We first discuss briefly the case of the
weighted projective space $\cP[w_0:...:w_n]= [ \Spec \CC[x_0, ...,
x_n] \setminus\{ 0\} /\CC^* ]$, where $\CC^*$ acts on $x_i$ with
weight $w_i$.
% open or \'etale cover

An \'etale presentation of the Deligne-Mumford stack
$\cP[w_0:...:w_n]$ is given by
%\'etale presentation e terminologia buna?
  $$ \bigsqcup_{i,j\in\{0,...,n\} } V_{ij} \rightrightarrows
   \bigsqcup_{i\in\{0,...,n\} } V_i $$
  where $V_i = \CC^n$ with coordinates $\{ u_{i}^k \}_{k\in\{0,...,n\}\setminus\{ i\} }$ and  $V_{ij}=V_{ji}=\CC^{n-1}\times \CC^*$ with
  coordinates $\{ v_{ij}^k \}_{k\in\{0,...,n\}\setminus\{ i\} }$ such
  that $v_{ij}^j\not=0$, and the \'etale maps $\phi_{ij}^i: V_{ij} \to
  V_i $ given by $$ \phi_{ij}^i (\{ v_{ij}^k \}_{k} ) = ( v_{ij}^0,...,
  (v_{ij}^j)^{w_j}, ..., v_{ij}^n),$$
  with $(v_{ij}^j)^{w_j}$ in the $j$--th position, and the change of
  coordinates
\bea \begin{array}{cc} v_{ji}^i=\frac{1}{v_{ij}^j}, &
v_{ji}^k=\frac{v_{ij}^k}{(v_{ij}^j)^{w_k}} \mbox{ for } k\not=i,j \end{array}
\eea on $V_{ij}$.

If $\{ u_{i}^k \}_{k\in\{0,...,n\}\setminus\{ i\} }$ are
coordinates on $V_i$,  the map to the coarse moduli space
$\cP[w_0:...:w_n]\to  \PP[w_0:...:w_n]$ sends
 $(u_i^k)_{k\in\{0,...,n\}\setminus\{ i\} } \in V_i$ to the point
 $[u_i^0:...:1:...:u_i^n]$ with $1$ in the $i$-th position.

The line bundle $\cO_{\cP[w_0:...:w_n]}(1)$ is determined by the
trivialisations ${L_i}\cong V_i\times\CC$, with gluing maps
$$\phi_{ij}^{i
*}{L_i}  \to \phi_{ij}^{j *}{L_j}$$
given by $\phi_{ij}^{j *}s_j=v_{ij}^j \phi_{ij}^{i*} s_i$, where $s_i$
and $s_j$ are the unitary sections on $L_i$ and $L_j$.
%Thus $\cO_{\cP[w_0:...:w_n]}(1)$ is the pullback of the line bundle $\cO_{\PP[w_0:...:w_n]}(1)$ on the coarse moduli scheme, whose total space is $$\PP[w_0:...:w_n: d]\setminus \{ [0:...:0:1]\} \to \PP[w_0:...:w_n].$$
% Here $d=\mbox{ gcd }(w_0, ..., w_n)=1$ and the map above is the projection on the first $(n+1)$ coordinates.
Thus  the total space of $\cO_{\cP[w_0:...:w_n]}(1)$ is $$\cP[w_0:...:w_n: 1]\setminus \{ [0:...:0:1]\} \to \cP[w_0:...:w_n].$$
Each weighted projective coordinate $x_i$ gives rise to a global
section of $\cO_{\cP[w_0:...:w_n]}(w_i)$,  which will be denoted by $x_i$ as well.

%Note that the condition gcd $(w_0, ..., w_n)=1$ implies the  smoothness of the stack $\cP[w_0:...:w_n]$.
% is equivalent to?
The tangent bundle to $\cP[w_0:...:w_n]$ is determined by the vector
bundles $\cT_{V_i}\cong V_i\times\CC^n$ with isomorphisms $$
\phi_{ij}^{i
*}\cT_{V_i} \cong \phi_{ij}^{j
*}\cT_{V_j},$$ identifying
\bea  \phi_{ij}^{j*}    \left( \frac{\partial }{\partial u_j^k}
\right) = (v_{ij}^j)^{w_k} \phi_{ij}^{i*} \left(
\frac{\partial}{\partial u_i^k} \right), \eea \bea
 \phi_{ij}^{j*} \left( \frac{\partial}{\partial u_j^i}
\right) = (v_{ij}^j)^{w_i} \phi_{ij}^{i*}  \left( -\sum_{k\not=i}
\frac{w_k}{w_i} v_{ij}^k \frac{\partial}{\partial u_i^k} \right). \eea
Indeed, $\phi_{ij}^{i *}\cT_{V_i} = \langle \left\{ \phi_{ij}^{i*}
\left( \frac{\partial}{\partial u_i^k} \right)=
\frac{\partial}{\partial v_{ij}^k} \right\}_{k\in\{0,...,n\}\setminus\{
i\} } , \phi_{ij}^{i*}\left( \frac{\partial}{\partial u_j^i}
\right)=\frac{1}{w_j} (v_{ij}^j)^{w_j-1} \frac{\partial}{\partial
v_{ji}^i} \rangle, $ and the identifications above are naturally
derived from the change of coordinates $v_{ij}^k$ into $v_{ji}^k$.

 The short exact sequence
\bean \diagram 0 \rto & \cO_{\cP[w_0:...:w_n]} \rto^{ \sigma } & {\oplus_i \cO_{\cP[w_0:...:w_n]}(w_i)} \rto^{    e  } &
\cT_{\cP[w_0:...:w_n]} \rto & 0
\enddiagram  \eean
is defined locally by $$\sigma (1):= (w_0x_0, ..., w_nx_n),$$ \[
\begin{array}{cc} e(s_i^{\otimes a_k}):= \frac{\partial}{\partial u_i^k}, &
 e(s_i^{\otimes a_i}):= -\sum_{k\not=i}
\frac{w_k}{w_i} u_i^k \frac{\partial}{\partial u_i^k}
\end{array}
\]
and by the presentation of the bundles above this is well defined
globally.

%weighted projective bundle

Given a weighted projective bundle $g:P= [A/\CC^* ] \to Y$
 with weights $(w_0, .., w_n)$, consider two open embeddings $U, U'\hookrightarrow Y$ and  trivializations
 $x_U: A_{|U}\cong U\times \CC^{n+1}$ and $x_{U'}: A_{|U'}\cong
U'\times \CC^{n+1}$, with gluing morphisms over $U\times_Y U'$
   $$\varphi=(\varphi_0, ..., \varphi_n) : (U\times_Y U')\times
   \CC^{n+1} \to \CC^{n+1}, $$
such that $\varphi_i$ is homogeneous of degree $w_i$ with respect to
the given weights. For each trivialisation above there is an \'etale
presentation $$ \bigsqcup_{i,j\in\{0,...,n\} } (V_{ij}\times U)
\rightrightarrows
   \bigsqcup_{i\in\{0,...,n\} } (V_i\times U) \mbox{ for } P_{|g^{-1}(U)},$$
   $$ \mbox{ respectively }
    \bigsqcup_{i,j\in\{0,...,n\} } (V'_{ij}\times U') \rightrightarrows
   \bigsqcup_{i\in\{0,...,n\} } (V'_i\times U') \mbox{ for } P_{|g^{-1}(U')}.$$
   Let $g_U$ denote the restriction of $g$ to $P_{|g^{-1}(U)}$, and
   similarly for $U'$. Pullback of the trivialization $x_U$ to
   $V_i\times_Pg^*_U(A_{|U})$, of  $x_{U'}$ to $V_i\times_Pg^*_U(A_{|U})$
   and of the gluing morphism $\varphi$ to $((V_{i}\times U)\times_P(V'_i\times U'))\times
   \CC^{n+1}$ amounts to choosing a root $\varphi_i^{1/w_i}$ of
   $\varphi_i$ such that a point $(V_{i}\times U)\times_P(V'_i\times U')$
   admits a change of coordinates $$u'^j_i =    \varphi_j (g(u), p_i(u))  \varphi_i^{-w_j/w_i}(g(u), p_i(u)),   $$
where $u=( g(u), ( u_i^j )_{j\not=i} )$ and $p_i(u)= (u_i^0, ...,
1,..., u_i^n)\in \CC^{n+1}$, with 1 in the $i$--th position. Thus
differentiating, \bea \frac{\partial}{\partial u_i^j}=
\sum_k\varphi_{kj} \frac{\partial}{\partial
u'^k_i}(\varphi_i^{-w_k/w_i}) - \varphi_{ij} \left(\sum_{k}
\frac{w_k}{w_i}u'^k_i \frac{\partial}{\partial u'^k_i}
\right)(\varphi_i^{-1}), \eea for $j\not=i$, where the functions
$\varphi_{kj}:=\frac{\partial \varphi_k}{\partial x_j}$ define a
local change of basis for the pullback of the normal bundle
$\cN_{Y|A}$ and $\varphi_i^{-1/w_i}$ defines a local change of basis
for $\cO_{P}(1)$. Furthermore, as $P$ is a weighted projective
bundle, $\varphi_{kj}=0$ unless $w_k = w_j$
%atentie la splitarea fibratului normal

This proves that the trivial extensions of the exact sequence (3.1) to
$P_{|g^{-1}(U)}$ and $P_{|g^{-1}(U')}$ glue to a restriction of the
sequence (3.2) on $P_{|g^{-1}(U\times_YU')}$.

\end{proof}

%References

\providecommand{\bysame}{\leavevmode\hbox to3em{\hrulefill}\thinspace}

\end{document}